\documentclass[reqno]{amsart}
\usepackage{amscd}

\def\beq{\begin{equation}}
\def\eeq{\end{equation}}
\def\ba{\begin{array}}
\def\ea{\end{array}}

\def\R{\mathbb R}

\def\nn{\nonumber}

\def\la{\langle}
\def\ra{\rangle}

\def \ds{\displaystyle}
\def \vs{\vspace*{0.1cm}}

\def\a{\alpha}

\def\bee{\begin{eqnarray}}
\def\beee{\begin{eqnarray*}}
\def\eee{\end{eqnarray}}
\def\eeee{\end{eqnarray*}}
\def\nn{\nonumber}
\def\ba{\begin{array}}
\def\ea{\end{array}}

\def\slashii#1{\setbox0=\hbox{$#1$}             
   \dimen0=\wd0                                 
   \setbox1=\hbox{\sl/} \dimen1=\wd1            
   \ifdim\dimen0>\dimen1                        
      \rlap{\hbox to \dimen0{\hfil\sl/\hfil}}   
      #1                                        
   \else                                        
      \rlap{\hbox to \dimen1{\hfil$#1$\hfil}}   
      \hbox{\sl/}                               
   \fi}                                         %
\def\slashiii#1{\setbox0=\hbox{$#1$}#1\hskip-\wd0\hbox to\wd0{\hss\sl/\/\hss}}
\newcommand{\A}{{\mathcal A}}

\newcommand{\C}{{\mathbf C}}

\newcommand{\N}{{\mathcal N}}

\newtheorem{thm}{Theorem}[section]
\newtheorem{lm}[thm]{Lemma}
\newtheorem{prop}[thm]{Proposition}

\theoremstyle{definition}
\newtheorem{rem}[thm]{Remark}

\newtheorem{df}[thm]{Definition}

\theoremstyle{remark}

\begin{document}
\pagestyle{plain}
\date{\today}

\title{Energy quantization for a singular super-Liouville boundary value problem}

\thanks{}
\author{J\"urgen Jost, Chunqin Zhou, Miaomiao Zhu}

\address{Max Planck Institute for Mathematics in the Sciences, Inselstr. 22, 04103 Leipzig, Germany}
\email{jjost@mis.mpg.de}

\address{Department of Mathematics and MOE-LSC, Shanghai Jiaotong University, 200240, Shanghai, P. R. China}
\email{cqzhou@sjtu.edu.cn}

\address{School of Mathematical Sciences, Shanghai Jiao Tong University, Dongchuan Road 800, 200240, Shanghai, P. R. China
\newline
Max Planck Institute for Mathematics in the Sciences, Inselstr. 22,	04103 Leipzig,	Germany
}
\email{mizhu@sjtu.edu.cn}

\begin{abstract}{In this paper, we develop the blow-up analysis and establish the energy quantization for solutions to super-Liouville type equations on Riemann surfaces with conical singularities at the boundary. In other problems in geometric analysis, the blow-up analysis usually strongly utilizes conformal invariance, which yields a Noether current from which strong estimates can be derived. Here, however, the conical singularities destroy conformal invariance. Therefore,
we develop another, more general,  method that uses the vanishing of the  Pohozaev constant for such solutions to deduce the removability of boundary singularities.}
\end{abstract}

\maketitle



\

\section{Introduction}
Many problems with a noncompact symmetry group, like the conformal group, are limit cases where the Palais-Smale condition no longer applies, and therefore, solutions may blow up at isolated singularities, see for instance \cite{Lions}. Therefore, a blow-up analysis is needed, and this has become one of the fundamental tools in the geometric calculus of variations. This usually depends on the fact that the invariance yields an associated Noether current whose algebraic structure can be turned into estimates.  In the case of  conformal invariance  this Noether current is  a holomorphic quadratic differential. For harmonic map type problems,  finiteness of the energy functional in question implies that that differential is in $L^1$. This then can be used to obtain fundamental  estimates. For other problems, however, like (super-) Liouville equations, finiteness of the energy functional is not sufficient to get the $L^1$ bound of that differential and
hence this is an extra assumption leading to the removability of local singularities (Prop 2.6, \cite{JWZZ1}).

But for (super-) Liouville equations  on surfaces with conical singularities, we do not even have conformal invariance, because the scaling behavior at the singularities is different from that at regular points, see \cite{JZZ3}. It turns out, however,  that for an important class of two-dimensional geometric variational problems, there is a condition that is weaker than conformal invariance, the vanishing of a so-called Pohozaev constant (i.e.  the Pohozaev identity),  that is not only sufficient but also necessary for the blow-up analysis. This  Pohozaev constant on one hand measures the extent to which the Pohozaev identity fails and on the other hand provides a characterization of the singular behavior of a solution at an isolated singularity.  This vanishing condition is already known to play a crucial role in geometric analysis (see e.g. \cite{St}), but for super-Liouville equations, as mentioned,  this identity by itself suffices for the blow-up analysis.

In this paper, we shall apply this strategy to the blow-up analysis of the (super-)Liouville boundary problem on surfaces with conical singularities. To this purpose, let $M$  be a compact Riemann surface with nonempty boundary $\partial M$ and with a spin structure. We also denote this compact Riemann surface as $(M,\A, g)$, where $g$ is its Riemannian metric with the conical singularities of divisor
$$\A= \Sigma^m_{j=1}\alpha_j q_j$$
(for definition of $\A$, see Section \ref{preli}). Associated to the metric $g$, one can define the gradient $\nabla$ and the Laplace operator $\Delta $ in the usual way.

We then have our main object of study,  the {\bf super-Liouville functional} that couples a real-valued function $u$ and a spinor
$\psi$ on $M$
\begin{equation}
\label{eq1} E_B\left( u,\psi \right) =\int_{M}\{\frac 12 \left|
\nabla u\right| ^2+K_gu+\left\langle (\slashiii{D}+e^u)\psi ,\psi
\right\rangle_g -e^{2u}\}dv+\int_{\partial M}\{h_gu-ce^u \}d\sigma,
\end{equation} where $K_g$ is
the Gaussian curvature in $M$, and $h_g$ is the geodesic curvature of
 $\partial M$£¬and
$c$ is a  given positive constant. The Dirac operator $\slashiii{D}$ is defined by
$\slashiii{D}\psi :=\sum_{\alpha =1}^2e_{_\alpha }\cdot \nabla
_{e_\alpha }\psi ,$ where $\left\{ e_1,e_2\right\} $ is an
orthonormal basis on $TM$, $\nabla$ is the Levi-Civita connection on
$M$ with respect to $g$ and $\cdot$ denotes Clifford multiplication
in the spinor bundle $\Sigma M$ of $M$. Finally, $ \la \cdot,\cdot
\ra_g $ is the natural Hermitian metric on  $\Sigma M$ induced by $g$. We also write $|\cdot|_g^2$  as $ \la \cdot, \cdot
\ra_g $.
For the geometric background, see \cite{LM} or \cite{Jo}.

The Euler-Lagrange system for $E_B(u,\psi)$ with Neumann / chirality
boundary conditions is
\begin{equation}
\left\{
\begin{array}{rcll}
-\Delta_g u &=& \ds\vs 2e^{2u}-e^u\left\langle \psi ,\psi
\right\rangle_g  -K_g,\qquad & \text { in } M^o\setminus \{q_1,q_2,\cdots, q_m\},
\\
\slashiii{D}_g\psi &=&\ds  -e^u\psi, \quad & \text { in } M^o\setminus \{q_1,q_2,\cdots, q_m\},\\
\ds\vs \frac {\partial u}{\partial n}& = & ce^u-h_g,\qquad & \text { on } \partial M\setminus \{q_1,q_2,\cdots, q_m\},\\
B^{\pm}\psi&=& 0,\qquad & \text { on } \partial M\setminus \{q_1,q_2,\cdots, q_m\}.
\end{array}
\right.   \label{Eq-NB}
\end{equation}
Here $B^{\pm}$ are the chirality operators (see Section \ref{preli} for the definition).

When $\psi=0$ and $(M, g)$ is a closed smooth Riemann surface, we obtain the classical Liouville functional
\begin{equation*}
 E\left( u \right) =\int_{M}\{\frac 12 \left|
\nabla u\right| ^2+K_gu -e^{2u}\}dv.
\end{equation*}
The Euler-Lagrange equation for $E(u)$ is the Liouville equation
\begin{equation*}
-\Delta_g u = \ds\vs 2e^{2u}  -K_g.
\end{equation*}
Liouville \cite{L} studied this equation in the plane, that is, for $K_g=0$.
The Liouville equation comes up in many problems of complex analysis
and differential geometry of Riemann surfaces, for instance
the prescribing curvature problem. The interplay between the
geometric and analytic aspects makes the Liouville equation
mathematically very interesting.

When $\psi\neq 0$ and $(M, g)$ again is a closed smooth Riemann surface, we obtain the super-Liouville funtional
\begin{equation*}
E\left( u,\psi \right) =\int_{M}\{\frac 12 \left| \nabla u\right|
^2+K_gu+\left\langle (\slashiii{D}+e^u)\psi ,\psi \right\rangle_g
-e^{2u}\}dv.
\end{equation*}
The Euler-Lagrange system for $E(u,\psi )$ is
\begin{equation*}
\left\{
\begin{array}{rcl}
-\Delta_g u &=& \ds\vs 2e^{2u}-e^u\left\langle \psi ,\psi
\right\rangle_g  -K_g\qquad
\\
\slashiii{D}_g\psi &=&\ds  -e^u\psi
\end{array} \text {in } M.
\right.
\end{equation*}

The supersymmetric version of the Liouville functional and equation have been studied extensively in the physics literature, see for instance \cite{Pr}, \cite{ARS} and \cite{FH}. As all supersymmetric functionals that arise in elementary particle physics, it needs anticommuting variables.

Motivated by the super-Liouville functional, a mathematical version of this functional  that works with commuting variables only, but otherwise preserves the structure and the invariances of it,   was introduced in \cite{JWZ1}.  That model couples the bosonic scalar field to a fermionic spinor field. In particular, the super-Liouville functional is conformally invariant, and it possesses a very interesting mathematical structure.

The analysis  of classical Liouville type equations was developed in  \cite{BM, LS, Li, BCLT} etc, and the corresponding analysis  for super-Liouville equations in \cite{JWZ1, JWZZ1, JZZ2}. In particular, the complete blow-up theory for sequences of solutions was established, including the energy identity for the spinor part, the blow-up value at blow-up points and the profile for a sequence of solutions at the blow-up points. For results by physicists about super-Liouville equations, we refer to \cite{Pr}, \cite{ARS} and \cite{FH} etc.

When $(M,\A, g)$ is a closed Riemann surface (without boundary) with conical singularities of divisor $\A $ and with a spin structure, we obtain that
\begin{equation*}
E\left( u,\psi \right) =\int_{M}\{\frac 12 \left| \nabla u\right|
^2+K_gu+\left\langle (\slashiii{D}+e^u)\psi ,\psi \right\rangle_g
-e^{2u}\}dv_g.
\end{equation*}
The Euler-Lagrange system for $E(u,\psi )$ is
\begin{equation}
\left\{
\begin{array}{rcl}
-\Delta_g u &=& \ds\vs 2e^{2u}-e^u\left\langle \psi ,\psi
\right\rangle_g  -K_g\qquad
\\
\slashiii{D}_g\psi &=&\ds  -e^u\psi
\end{array} \text {in } M\backslash \{q_1,q_2,\cdots, q_m\}.
\right.   \label{eq-2}
\end{equation}
This system is closely related to the classical Liouville equation, or the prescribing curvature equation on M with conical singularites
(see \cite{T1}, \cite{CL1}). \cite{BT, BT1, B, Ta, BCLT, BaMo} studied the blow-up theory of the following Liouville type equations with singular data:
$$
-\Delta_g u=\lambda \frac{Ke^u}{\int_M K e^{u}dg}-4\pi(\Sigma_{j=1}^{m}\alpha_j\delta_{q_j}- f),
$$ where $(M,g)$ is a smooth surface and the singular data appear in the equation. For system (\ref {eq-2}),   \cite{JZZ3} provides an analytic foundation  and the blow-up theory.

For Liouville boundary problems on $(M, g)$ with or without conical singularites, there are also lots of results on the blow-up analysis, see \cite{JWZ2, BWZ, GL, ZZ, ZZZ}. For super-Liouville boundary problems  on a smooth Riemann surface $M$, the corresponding results can be found in \cite {JZZ1, JWZZ2}.

 In this paper, we aim to provide an analytic foundation and to establish the blow-up analysis for the system (\ref{Eq-NB}). Our main result is the following energy quantization property for solutions to (\ref{Eq-NB}):

\begin{thm} \label{thmsin} Let $(u_n,\psi_n)$ be a sequence of solutions of (\ref{Eq-NB}) with energy conditions:
$$
\int_{M}e^{2u_n}dg<C,~~~~~~\int_{M}|\psi_n|_g^4dg<C.
$$
Define
$$
\Sigma _1=\left\{ x\in M,\text{ there is a sequence
}y_n\rightarrow x\text{ such that }u_n(y_n)\rightarrow +\infty
\right\}.
$$
If $\Sigma_1\neq \emptyset$, then  the possible values of
$$\lim_{n\rightarrow \infty}\{ \int_{M}2e^{2u_n}-e^{u_n}|\psi_n|^2_{g}dv_g+\int_{\partial M}ce^{u_n}d\sigma_g\}$$
is
$$  4\pi \N+2\pi \N+\Sigma_j 4\pi(1+\alpha_j)\{0, 1\}+\Sigma_j 2\pi(1+\alpha_j)\{0,1\},$$
where $\N=
\{0,1,2,\cdots, k\}$.
\end{thm}

\
\

From the energy quantization property, one can deduce the
concentration properties of conformal volume and the compactness of
solutions. It turns out that understanding of this property is the key
step to study existence from  a variational point of view by a refined Moser-Trudinger inequality, see e.g. \cite{DJLW, DM}.

If we assume that the points $q_1, q_2, \cdots, q_l$ are in $M^o$ for $1\leq l<m$ and the points $q_{l+1}, q_{l+2}, \cdots, q_m$ are on $\partial M$ for the surface $(M,\A,g)$ with the divisor $\A=\Sigma_{j=1}^{m}\alpha_jq_j$, $\alpha_j>0$,  we have the following  Gauss-Bonnet formula
\begin{equation*}
\frac 1{2\pi}\int_{M}K_gdv_g+\frac{1}{2\pi}\int_{\partial M}h_gd\sigma_g = \mathcal{X}(M) + |\A|,
\end{equation*}
where $\mathcal{X}(M)=2-2g_M$ is the topological Euler characteristic of $M$ itself, $g_M$ is the genus of $M$ and
$$|\A|= \Sigma_{j=1}^{l}\alpha_j+\Sigma_{j=l+1}^{m}\frac{\alpha_j}{2}$$
is the degree of $\A$, see \cite{T1}. From (\ref{Eq-NB}) we obtain that
$$
\int_{M}2e^{2u_n}-e^{u_n}|\psi_n|^2_gdv_g+\int_{\partial M}ce^{u_n}d\sigma_{g}=\int_{M}K_gdv_g+\int_{\partial M}h_gd\sigma_g=2\pi(\mathcal{X}(M) + |\A|).
$$

\

Then we can use Theorem \ref{thmsin} to get the following:

\begin{thm}\label{global-blowup}
Let $(M,\A,g)$ be as above.  Then
\begin{itemize}
\item[(i)] if
$$2\pi(1-g_M)+2\pi\Sigma_{j=1}^{l}\alpha_j+ \pi\Sigma_{j=1+1}^{m}\alpha_j= 2\pi,$$
then the blow-up set $\Sigma_1$ contains at most one point. In particular, $\Sigma_1$ contains at most one point if $g_M=0$ and $\A=0$.
\item[(ii)] if
$$2\pi(1-g_M)+2\pi\Sigma_{j=1}^{l}\alpha_j+ \pi\Sigma_{j=1+1}^{m}\alpha_j < \pi,$$
then the blow-up set $\Sigma_1=\emptyset$.
\end{itemize}
\end{thm}

To show Theorem \ref{thmsin}, a key step is to compute the blow-up value
$$m(p)=\lim_{R\rightarrow 0}\lim_{n\rightarrow \infty}\{\int_{B^M_R(p)}(2e^{2u_n}-e^{u_n}|\psi_n|_{g}^2-K_{g})dv_{g}+\int_{\partial M\cap B^M_R(p)}(ce^{u_n}-h_{g})d\sigma_{g}\},
$$ at the blow-up point $p\in \Sigma_1$ for a blow-up sequence $(u_n,\psi_n)$. Here $B^M_R(p)$ is a geodesic ball of $(M,g)$ at $p$. For this purpose, we need to study the following local super-Liouville boundary value problem (see Section \ref{loc}):
\begin{equation}\label{Eq-B}
\left\{
\begin{array}{rcll}
-\Delta u(x) &=& 2V^2(x)|x|^{2\alpha}e^{2u(x)}-V(x)|x|^{\alpha}e^{u(x)}|\Psi|^2,  \quad &\text {in } D^+_{r}, \\
\slashiii{D}\Psi &=& -V(x)|x|^{\alpha}e^{u(x)}\Psi, \quad &\text {in } D^+_{r}, \\
\frac {\partial u}{\partial n}& = & cV(x)|x|^{\alpha}e^{u(x)},\quad & \text { on } L_r, \\
B^{\pm}\Psi &=& 0,\quad & \text { on } L_r.
\end{array}
\right.
\end{equation}
Here $\alpha \geq 0$, $V(x)$ is in $C^{1}_{loc}(D^+_r\cup L_r)$ and satisfies $0< a\leq V(x)\leq b$.  $L_r$ and $S_r^+$ here and in the sequel are  portions of  $\partial D^+_r$, which are defined in section 3. Then we have the following Brezis-Merle type concentration compactness theorem:
\begin{thm}\label{mainthm}
Let $(u_n,\Psi _n)$ be a sequence of regular solutions to $\left(
\ref{Eq-B}\right) $
 satisfying
 $$\int_{D^+_r}|x|^{2\alpha}e^{2u_n}+|\Psi_n|^4dx+\int_{L_r}|x|^{\alpha}e^{u_n}ds<C.$$ Define
\[\begin{array}{rcl}
\Sigma _1&=&\left\{ x\in  D^+_r\cup L_r,\text{ there is a sequence
}y_n\rightarrow x\text{ such that }u_n(y_n)\rightarrow +\infty
\right\},\\
\Sigma _2&=&\left\{ x\in D^+_r\cup L_r,\text{ there is a sequence
}y_n\rightarrow x\text{ such that }\left| \Psi_n(y_n)\right|
\rightarrow +\infty \right\} .\end{array}
\]
Then, we have $\Sigma_2\subset\Sigma_1$. Moreover, $(u_n,\Psi _n)$
admits a subsequence, still denoted by $( u_n,\Psi _n),$ that
satisfies

\begin{enumerate}
\item[a)] $|\Psi _n|$ is bounded in $%
L_{loc}^{\infty} ( (D^+_r\cup L_r)\backslash \Sigma _2)$ .

\item[b)]  For $u_n$, one of the following alternatives holds:
\begin{enumerate}
\item[i)]  $u_n$ is bounded in $L^{\infty}_{loc} (D^+_r\cup L_r).$

\item[ii)]  $u_n$ $\rightarrow -\infty $ uniformly on compact subsets of $D^+_r\cup L_r$.

\item[iii)]  $\Sigma _1$ is finite, nonempty and either
\begin{equation}
u_n\text{ is bounded in }L_{loc}^\infty ((D^+_r\cup L_r)\backslash \Sigma _1)
\label{aaa}
\end{equation}
or
\begin{equation}
u_n\rightarrow -\infty \text{ uniformly on compact subsets of }
(D^+_r\cup L_r)\backslash \Sigma _1.  \label{ddd}
\end{equation}
\end{enumerate}
\end{enumerate}
\end{thm}

To show the quantization property of the blow-up value, we need to rule out (\ref{aaa}) in the above theorem. To this end, the decay estimates of the spinor part $\Psi_n$, the Pohozaev identity  of the local system and the energy identity of $\Psi_n$, which means there is no energy contribution on the neck domain, play the essential roles. The corresponding theorem is the following:

\begin{thm}\label{engy-indt}
Let $(u_n,\Psi _n)$ be a sequence of regular solutions to $\left(
\ref{Eq-B}\right) $
 satisfying
 $$\int_{D^+_r}|x|^{2\alpha}e^{2u_n}+|\Psi_n|^4dx+\int_{L_r}|x|^{\alpha}e^{u_n}ds<C.$$
 Denote by $\Sigma_1=\{x_1, x_2,\cdots,
x_l\}$ the blow-up set of $u_n$. Then there are finitely many solutions
$(u^{i,k},\Psi^{i,k})$ that satisfy
\begin{equation}\label{b1}
\left\{
\begin{array}{rcll}
-\Delta u^{i,k} &=& \ds\vs 2|x|^\alpha e^{2u^{i,k}}-|x|^\alpha e^{u^{i,k}}\left\langle
\Psi^{i,k} ,\Psi^{i,k} \right\rangle-1 ,\qquad &\text {in } S^2,
\\
\slashiii{D}\Psi^{i,k} &=&\ds  -|x|^\alpha e^{u^{i,k}}\Psi^{i,k}, \quad &\text {in } S^2,\\
\end{array}
\right.
\end{equation}
for $i=1,2,\cdots , I$, and  $k=1,2,\cdots, K_i$, and $\alpha\geq 0$, or there are finitely
many solutions $(u^{j,l},\Psi^{j,l})$ that satisfy
\begin{equation}\label{b2}
\left\{
\begin{array}{rcll}
-\Delta u^{j,l} &=& \ds\vs 2|x|^\alpha e^{2u^{j,l}}-|x|^\alpha e^{u^{j,l}}\left\langle
\Psi^{j,l} ,\Psi^{j,l} \right\rangle-1 ,\quad &\text { in }
S^2_{c'},
\\
\slashiii{D}\Psi^{j,l} &=&\ds  -|x|^\alpha e^{u^{j,l}}\Psi^{j,l}, \quad &\text { in } S^2_{c'},\\
\ds\vs \frac {\partial u^{j,l}}{\partial n}& = & c|x|^\alpha e^{u^{j,l}}-c', \quad &\text{ on } \partial S^2_{c'},\\
B^\pm \Psi^{j,l} &=& 0, \quad &\text{ on } \partial S^2_{c'},
\end{array}
\right.
\end{equation}
for $j=1,2,\cdots , J$,  and  $l=1,2,\cdots, L_j$,  and $\alpha\geq 0$. Here $S^2_{c'}$ is a
portion of the sphere cut out by a 2-plane with constant geodesic curvature
$c'$. After selection of a subsequence, $\Psi_n$ converges in
$C_{loc}^{\infty}$ to $\Psi$ on $(B_r^+\cup L_r)\backslash \Sigma_1$ and we have
the energy identity:
\begin{equation}\label{1.2} \lim_{n\rightarrow
\infty}\int_{D^+_r}|\Psi_n|^4dv=\int_{D^+_r}|\Psi|^4dv+\sum_{i=1}^{I}
\sum_{k=1}^{K_i}\int_{S^2}|\Psi^{i,k}|^4dv+\sum_{j=1}^{J}
\sum_{l=1}^{L_j}\int_{S^2_{c'}}|\Psi^{j,l}|^4dv. \end{equation}
\end{thm}

A crucial step in proving the above theorem is to show the removability of isolated singularities at the boundary, which is equivalent to the vanishing of the Pohozaev constant (see Theorem \ref{thm-sigu-move1-B}). Once the energy identity for the spinor part  ($\ref{1.2}$) is established, we can then obtain
\begin{thm}\label{mainthm-1}
Let $(u_n, \Psi_n)$ be solutions as in Theorem \ref{mainthm}. Assume that $(u_n, \Psi_n)$ blows up and the blow-up set $\Sigma_1\neq \emptyset$. Then
\[
u_n\rightarrow -\infty\quad \text { uniformly on compact subsets of } (D^+_r\cup L_r)\setminus\Sigma_1.
\]
Furthermore,
\[\int_{D^+_r(0)}[2V(x)|x|^{2\alpha}e^{2u_n}-V(x)|x|^{\alpha}e^{u_n}|\Psi_n|^2]\phi dx+ \int_{L_r}cV(x)|x|^{\alpha}e^{u_n}\rightarrow \sum_{x_i\in \Sigma_1}m(x_i)\phi(x_i)\]
for every $\phi\in C^\infty_o(D^+_r\cup L_r)$ and   $m(x_i)>0$.
\end{thm}

In the end, with the help of the Pohozaev identy (see Proposition \ref {prop-poho}) and the Green function at some singular points, we have the following:

\begin{thm}\label{BV}
Let $(u_n, \Psi_n)$ be solutions as in Theorem \ref{mainthm}. Assume that $(u_n, \Psi_n)$ blows up and the blow-up set $\Sigma_1\neq \emptyset$. Let $p\in \Sigma_1$ and assume that $p$ is the only blow-up point in $\overline{D_{\delta_0}^+(p)}$ for some $\delta_0>0$. If there exists a positive constant $C$ such that
$$\max_{S_{\delta_0}^+(p)}u_n-\min_{S_{\delta_0}^+(p)}u_n\leq C,$$
then the blow-up value $m(p)=4\pi$ when $p\notin L_{\delta_0}(p)$, $m(p)=2\pi$ when $p\in L_{\delta_0}(p)\setminus \{0\}$, and $m(p)=2\pi(1+\alpha)$ when $p=0$.
\end{thm}

\

\section{Preliminaries}\label{preli}

In this section, we will first recall the definition of surfaces with conical singularities, following \cite{T1}. Then we shall recall  the chirality boundary condition for the Dirac operator $\slashiii{D}$, see e.g. \cite{HMR}. In particular, we will see that under the chirality boundary conditions $B^{\pm}$, the Dirac operator $\slashiii D$ is self-adjoint.

A conformal
metric $g$ on a Riemannian surface $\Sigma$ (possibly with
boundary) has a conical singularity of order $\alpha$ (a real number
with $\alpha>-1$) at a point $p\in \Sigma\cup
\partial \Sigma$ if in some neighborhood of $p$
$$g=e^{2u}|z-z(p)|^{2\alpha}|dz|^2$$
where $z$ is a coordinate of $\Sigma$ defined in this neighborhood
and $u$ is smooth away from $p$ and  continuous at $p$. The point
$p$ is then said to be a conical singularity of {\it angle}
$\theta=2\pi(\alpha+1)$ if $p\notin\partial \Sigma$ and a {\it
corner} of angle $\theta =\pi(\alpha +1)$ if $p\in \partial \Sigma$.
 For example, a (somewhat idealized) American football has two singularities of equal angle, while
a teardrop has only one singularity. Both these examples correspond
to the case $-1<\alpha <0$; in case $\alpha >0$, the angle is larger
than $2\pi$, leading to a different geometric picture. Such
singularities also appear in orbifolds and branched coverings. They
can also describe the ends of complete Riemann surfaces with
finite total curvature. If $(M, g)$ has conical
singularities of order $\alpha_1, \alpha_2, \cdots, \alpha_m$ at
$q_1, q_2, \cdots, q_m$, then $g$ is said to represent the
divisor $\A= \Sigma^m_{j=1}\alpha_j q_j$. Importantly, the presence of such conical singularities destroys conformal invariance, because the conical points are different from the regular ones.

The chirality boundary condition for the Dirac operator $\slashiii{D}$ is  a natural boundary
condition for spinor part $\psi$. Let $M$ be a compact
Riemann surface with $\partial M\neq \emptyset$ and with a fixed
spin strcuture, admitting a chirality operator $G$, which is an
endomorphism of the spinor bundle $\Sigma M$ satisfying:
$$
G^2=I,\qquad \la G\psi ,G\varphi\ra=\la \psi, \varphi\ra,
$$
and $$ \nabla_{X}(G\psi)=G\nabla_{X}\psi,\qquad X\cdot
G\psi=-G(X\cdot \psi),
$$ for any $X\in \Gamma (TM), \psi, \varphi\in \Gamma (\Sigma M)$.
Here $I$ denotes the identity endomorphism of $\Sigma M$.

We usually take $G=\gamma(\omega_2)$, the Clifford multiplication by
the complex volume form $\omega_2=ie_1e_2$, where $e_1,e_2$ is a
local orthonormal frame on $M$.

Denote by $$ S:=\Sigma M|_{\partial M}$$ the restricted spinor bundle
with induced Hermitian product.

Let $\overrightarrow{n}$ be the outward unit normal vector field on
$\partial M$. One can verify that
$\overrightarrow{n}G:\Gamma(S)\rightarrow \Gamma(S)$ is a
self-adjoint endomorphism  satisfying
$$
(\overrightarrow{n}G)^2=I,\qquad \la \overrightarrow n G\psi
,\varphi\ra=\la \psi, \overrightarrow n G\varphi\ra,
$$
Hence, we can decompose $S=V^{+}\bigoplus V^{-}$, where $V^{\pm}$ is
the eigensubbundle corresponding to the eigenvalue $\pm 1$. One
verifies that the orthogonal projection onto the eigensubbundle
$V^{\pm}$:
\begin{eqnarray*}
B^{\pm}:L^2(S)&\rightarrow& L^2(V^{\pm})\\
\psi &\mapsto &\frac 12(I\pm\overrightarrow nG)\psi,
\end{eqnarray*}
defines a local elliptic boundary condition for the Dirac operator
$\slashiii{D}$ , see e.g. \cite{HMR}. We say that a spinor $\psi \in
L^2(\Gamma(\Sigma M))$ satisfies the chirality boundary conditions
$B^{\pm}$ if $$ B^\pm\psi|_{\partial M}=0.
$$ It is well known (see e.g. \cite{HMR}) that if $\psi,\phi\in L^2(\Gamma(\Sigma M))$
satisfy the chirality boundary conditions $B^\pm$ then
$$ \la \overrightarrow n\cdot \psi,\varphi\ra=0, \quad \text{ on }
\partial M.
$$
In particular,
\begin{equation}\label{2.11}
\int_{\partial M}\la \overrightarrow n\cdot \psi,\varphi\ra=0.
\end{equation}
It follows that  the Dirac operator $\slashiii D$ is self-adjoint
under the chirality boundary conditions $B^{\pm}$.

It may be helpful if we recall that on a surface the (usual) Dirac operator
$\slashiii{D} $ can be seen as the (doubled) Cauchy-Riemann
operator. Consider $\R^2$ with the Euclidean metric $ds^2+dt^2$. Let
$e_1=\frac{\partial}{\partial s}$ and $e_2=\frac{\partial}{\partial
t}$ be the standard orthonormal frame. A spinor field is simply a
map $\Psi:\R^2\to \Delta_2=\C^2$, and the actions of $e_1$ and $e_2$ on
spinor fields can be identified by multiplication with matrices
\[e_1=\left(\begin{matrix}0& i\\ i&0 \end{matrix}\right),
\quad e_2=\left(\begin{matrix}0& 1\\ -1&0 \end{matrix}\right).\] If
$\Psi:=\left(\begin{matrix} \ds f
\\ \ds g\end{matrix}\right)
:\R^2\to \C^2$ is a spinor field, then the Dirac operator is
\[\slashiii{D}\Psi=\ds \left(\begin{matrix}0& i\\ i&0
\end{matrix}\right) \left(\begin{matrix} \ds \frac{\partial
f}{\partial s}\\ \ds \frac{\partial g}{\partial s}
\end{matrix}\right)+
\left(\begin{matrix}0& 1\\ -1&0 \end{matrix}\right)
\left(\begin{matrix} \ds \frac{\partial f}{\partial t} \\
\ds\frac{\partial g}{\partial t}
\end{matrix}\right)=
2i\left(\begin{matrix} \ds \frac{\partial g}{\partial  z}
\\ \ds\frac{\partial f}{\partial \bar z}\end{matrix}\right),\]
where
\[\frac{\partial}{\partial z}=\frac 12 \left(\frac{\partial }{\partial s}
- i\frac{\partial }{\partial t}\right), \quad
\frac{\partial}{\partial \bar z}=\frac 12 \left(\frac{\partial
}{\partial s} + i\frac{\partial }{\partial t}\right).\] Therefore,
the elliptic estimates developed for (anti-) holomorphic functions
can be used to study the Dirac equation.

If $M$ be the upper-half Euclidean space $\R^{2}_{+}$, then the
chirality operator is $G=ie_1e_2=\left(\begin{matrix}1& 0\\ 0&-1
\end{matrix}\right)$. Note that $\overrightarrow n=-e_2$, we get
that $$ B^{\pm}=\frac 12(I\pm \overrightarrow n\cdot G)=\frac
12\left(\begin{matrix}1& \pm 1\\ \pm 1& 1 \end{matrix}\right).
$$ By the standard chirality decomposition, we can write $\psi =\left(\begin{matrix}\psi_+\\ \psi_{-}
\end{matrix}\right)$, and then the boundary condition  becomes
$$
\psi_{+}=\mp \psi_{-} \quad \text{ on } \partial M.
$$

Without loss of generality, in the sequel, we shall only consider the chirality boundary
condition $B=B^+$.

We have the following geometric property:
\begin{prop}\label{prop-c} The functional $E_B(u,\psi )$ is invariant under conformal diffeomorphisms $\varphi
:M\rightarrow M$ preserving the divisor, that is, $\varphi*\A=\A$. In other word, if we write that $\varphi^*(g)=\lambda^2 g$, where $\lambda>0$ is the conformal factor of the conformal map $\varphi$, and set
\beq\label{n3.1}\ba{rcl}
\ds\vs \widetilde{u} &=&\ds u\circ \varphi -ln\lambda, \\
\ds \widetilde{\psi } &=&\ds \lambda^{-\frac 12}\psi \circ \varphi,
\ea \eeq  then  $E_B(\tilde
u,\tilde \psi)=E_B(u,\psi )$. In particular, if $(u,\psi)$ is a solution of (\ref{Eq-NB}), so is $(\tilde
u,\tilde \psi)$.
\end{prop}


\

\section{The local singular  super-Liouville boundary problem}\label{loc}

In this section, we shall first derive the local version of the super-Liouville boundary problem. Then we shall analyze the regularity of solutions under the small energy condition.

First we take a point $p\in M^o$,  choose a small neighborhood $U(p) \subset M^o$, and define an isothermal coordimate system $x=(x_1,x_2)$ centered at $p$, such that $p$ corresponds to $x=0$ and
$g=e^{2\phi}|x|^{2\alpha}(dx_1^2+dx_2^2)$ in $D_{r}(0)=\{(x_1,x_2) \in \R^2 \mid x_1^2+x_2^2< r^2\}$, where $\phi$ is smooth away from $p$,
continuous at $p$ and $\phi(p)=0$. We can choose such a neighborhood small enough so that  if $p$ is a conical singular point of $g$,
then $U(p)\cap {\A}=\{p\}$ and $\alpha>0 $, while, if $p$ is a smooth point of $g$, then  $U(p)\cap {\A}=\emptyset$
and $\alpha=0 $. Consequently, with respect to the isothermal coordinates, we can obtain the local version of the  singular super-Liouville-type equations,
\begin{equation}\label{Eq-L}
\left\{
\begin{array}{rcl}
-\Delta u(x) &=& 2V^2(x)|x|^{2\alpha}e^{2u(x)}-V(x)|x|^{\alpha}e^{u(x)}|\Psi|^2  \quad\\
\slashiii{D}\Psi &=& -V(x)|x|^{\alpha}e^{u(x)}\Psi
\end{array} \text { in } D_{r}(0),
\right.
\end{equation}
which is no any boundary condition since $p$ is a interior point of $M$. Here $\Psi=|x|^{\frac {\alpha}2}e^{\frac{\phi(x)}{2}}\psi$, $V(x)$ is a $C^{1,\beta}$ function and satisfies $0< a\leq V(x)\leq b$. The detailed arguments can be found in the section 3 of \cite{JZZ3}.  We also assume that $(u,\Psi)$ satisfy the energy condition:
\begin{equation}\label{Co-L}
\int_{D_r(0)}|x|^{2\alpha}e^{2u}+|\Psi|^4dx<+\infty.
\end{equation}

\
\
We put $D_r:=D_r(0)$. We say that $(u, \Psi)$ is a weak solution of  \eqref{Eq-L} and \eqref{Co-L}, if $u\in W^{1,2}(D_r)$
and $\Psi \in W^{1,\frac 43}(\Gamma (\Sigma D_r))$ satisfy
\begin{eqnarray*}
\int_{D_r} \nabla u\nabla\phi dx &=& \int_{D_r} (2V^2(x)|x|^{2\alpha}e^{2u}-V(x)|x|^{\alpha}e^u |\Psi|^2)\phi dx, \\
\int_{D_r} \langle \Psi,\slashiii{D} \xi \rangle dx &=&-\int_{D_r} V(x)|x|^{\alpha}e^u \langle \Psi,\xi \rangle dx,
\end{eqnarray*}
for any $\phi\in C^\infty_0(D_r)$ and any spinor $\xi \in C^\infty \cap  W_0^{1,\frac 43} (\Gamma (\Sigma D_r))$.
A weak solution is a classical solution by the following:

\begin{prop}\label{prop-a} Let $(u,\Psi)$ be a weak solution  of \eqref{Eq-L} and \eqref{Co-L}.
Then $(u,\Psi)\in C^2(D_r)\times C^2(\Gamma(\Sigma D_r))$.
\end{prop}

Note that when $\alpha=0$ this proposition is proved in \cite{JWZ1} (see Proposition 4.1).  When $\alpha>0$, this proposition is proved in \cite {JZZ3}(see Proposition 3.1).

\
\

For $p\in \partial M$,  we also can choose a small geodesic ball $U(p) \subset M$ and define an isothermal coordimate system $x=(x_1,x_2)$ centered at $p$, such that $p$ corresponds to $x=0$ and
$g=e^{2\phi}|x|^{2\alpha}(dx_1^2+dx_2^2)$ in $\overline{D}^+_{r}(0)=\{(s,t)\in \R^2 \mid s^2+t^2< r^2, t\geq 0 \}$, where $\phi$ is smooth away from $p$ and continuous at $p$. We can choose such a geodesic ball small enough so that  if $p$ is a conical singular point of $g$,
then $U(p)\cap {\A}=\{p\}$ and $\alpha>0 $, while, if $p$ is a smooth point of $g$, then  $U(p)\cap {\A}=\emptyset$
and $\alpha=0 $. Set $L_r=\partial D_r^+\cap \partial {\R}^2_+$, and $S^+_r=\partial D^+_r\cap \R^2_+$. Also in the sequel, we will set $L_r(x_0)=\partial D_r^+(x_0)\cap \partial {\R}^2_+$, and $S^+_r(x_0)=\partial D^+_r(x_0)\cap \R^2_+$. Consequently, with respect to the isothermal coordinates, $(u,\psi)$ satisfies
\begin{equation}
\left\{
\begin{array}{rcll}
-\Delta u(x) &=&  e^{2\phi(x)}|x|^{2\alpha}(2e^{2u(x)}-e^{u(x)}|\psi|^2(x) -K_g)\quad &\text {in } D^+_{r},\\
\slashiii{D}(e^{\frac {\phi(x)}2}|x|^{\frac {\alpha}2}\psi) &=& -e^{\phi(x)}|x|^{\alpha}e^{u(x)}(e^{\frac {\phi(x)}2}
|x|^{\frac {\alpha}2}\psi)\quad & \text {in } D^+_{r},\\
\ds\vs \frac {\partial u}{\partial n}& = & e^{\phi(x)}|x|^{\alpha}(ce^u-h_g),\quad & \text { on } L_r,\\
B(e^{\frac {\phi(x)}2}|x|^{\frac {\alpha}2}\psi) &=& 0,\quad & \text { on } L_r.
\end{array}
\right.   \label{eq-5}
\end{equation}
Here $\Delta=\partial^2_{x_1x_1}+\partial^2_{x_2x_2}$ is the usual Laplacian, and the Dirac operator
$\slashiii{D} $ can be seen as doubled Cauchy-Riemann operator, $B$ is the chirality boundary operator of spinors.

Note that the relation between the two Gaussian curvatures and between the two geodesic curvatures are respectively
\begin{equation*}
\left\{
\begin{array}{rcl}
-\Delta\phi &=& e^{2\phi}|x|^{2\alpha}K_g,\\
\frac {\partial \phi}{\partial n} &=& e^{\phi}|x|^{\alpha}h_g.
\end{array}
\right.
\end{equation*}
By standard elliptic regularity we conclude that $\phi\in W^{2,p}_{loc}(D^+_r\cup L_r)$ for some $p>1$ if $\alpha\geq 0$ and
 if the curvature $K_g$ and $h_g$ of $M$ is regular enough. Therefore, by Sobolev embedding, $\phi\in C^{1}_{loc}(D^+_r\cup L_r)$.
If we denote $V(x)=e^{\phi}$,  $W_1(x)=e^{2\phi}|x|^{2\alpha}K_g$ and $W_2(x)=e^{\phi}|x|^{\alpha}h_g$, then  $0< a\leq V(x)\leq b$,  $W_1(x)$ is in $L^p(D^+_r)$ and $W_2(x)$ is in $L^p(L_r)$ for all $p>1$ if the curvature $K_g$ and $h_g$ of $M$ is regular enough. Therefore,  the equations (\ref{eq-5}) can be rewritten as:
\begin{equation*}
\left\{
\begin{array}{rcll}
-\Delta u(x) &=& 2V^2(x)|x|^{2\alpha}e^{2u(x)}-V(x)|x|^{\alpha}e^{u(x)}|\Psi|^2-W_1(x),  \quad &\text {in } D^+_{r}, \\
\slashiii{D}\Psi &=& -V(x)|x|^{\alpha}e^{u(x)}\Psi, \quad &\text {in } D^+_{r}, \\
\frac {\partial u}{\partial n}& = & cV(x)|x|^{\alpha}e^u-W_2,\quad & \text { on } L_r, \\
B(\Psi) &=& 0,\quad & \text { on } L_r.
\end{array}
\right.
\end{equation*}
Furthermore, let $w(x)$ satisfy
\begin{equation*}
\left\{
\begin{array}{rcll}
-\Delta w(x) &=& -W_1(x),  \quad &\text {in } D^+_{r},\\
\frac {\partial u}{\partial n}& = & -W_2(x), \quad & \text { on } L_r, \\
w(x) &=& 0,   \quad  &\text{ on } S^+_r.
\end{array}
\right.
\end{equation*}
It is easy to see that $w(x)$ is in $C^{2}(D^+_r)\cap C^{1}(D^+_r\cup L_r)$. Setting $v(x)=u(x)-w(x)$, then $(v,\Psi)$ satisfies
\begin{equation*}
\left\{
\begin{array}{rcll}
-\Delta v(x) &=& 2V^2(x)|x|^{2\alpha}e^{2v(x)}-V(x)|x|^{\alpha}e^{v(x)}|\Psi|^2,  \quad &\text {in } D^+_{r}, \\
\slashiii{D}\Psi &=& -V(x)|x|^{\alpha}e^{v(x)}\Psi, \quad &\text {in } D^+_{r}, \\
\frac {\partial v}{\partial n}& = & cV(x)|x|^{\alpha}e^{v(x)},\quad & \text { on } L_r, \\
B(\Psi) &=& 0,\quad & \text { on } L_r.
\end{array}
\right.
\end{equation*}
Here $\alpha \geq 0$, $V(x)$ is in $C^{1}_{loc}(D^+_r\cup L_r)$ and satisfies $0< a\leq V(x)\leq b$. Thus we get the local system (\ref{Eq-B}) of the boundary problem  (\ref{Eq-NB}).

\
\

As the interior case, we can also  define $(u,\Psi)$ be a weak solution of (\ref{Eq-B}) if $u\in
W^{1,2}(D^+_r)$ and $\Psi \in W^{1,\frac 43}_B(\Gamma (\Sigma D^+_r))$
satisfy
\begin{eqnarray*}
\int_{D^+_r} \nabla u\nabla\phi dx&=& \int_{D^+_r} (2V^2(x)|x|^{2\alpha}e^{2u(x)}-V(x)|x|^{\alpha}e^{u(x)}|\Psi|^2)\phi dx+\int_{ L_r}(cV(x)|x|^{\alpha}e^{v(x)})\phi d\sigma \\
\int_{D^+_r} \langle \Psi,\slashiii{D}\xi\rangle
dx &=& -\int_{D^+_r} V(x)|x|^{\alpha}e^{v(x)}\langle \Psi,\xi\rangle dx
\end{eqnarray*}
for any  $\phi\in C_0^{\infty}({D^+_r\cup L_r})$ and any spinor
$\xi\in C^\infty_0(\Gamma (\Sigma (D^+_r\cup L_r)))\cap W^{1,\frac 43}_B(\Gamma (\Sigma D^+_r))$. Here

$$
W^{1,\frac 43}_B(\Gamma (\Sigma D^+_r))=\{\psi| \psi\in W^{1,\frac
43}(\Gamma (\Sigma D^+_r)), B\psi|_{L_r}=0 \}.
$$

\
\

For weak solutions of (\ref{Eq-B}) we also have the following regularity result.
\begin{prop}\label{prop-b} Let $(u,\Psi)$ be a weak solution  of (\ref{Eq-B})
with the energy condition
\begin{equation}\label{Eq-BC}
\int_{D^+_r}|x|^{2\alpha}e^{2u}+|\Psi|^4 dv+\int_{L_r}|x|^{\alpha}e^u d\sigma
<\infty.
\end{equation} Then $u\in C^{2}(D^+_r)\cap C^{1}(D^+_r\cup L_r)$
and
 $\Psi \in C^{2}(\Gamma(\Sigma D^+_r ))\cap C^{1}(\Gamma(\Sigma (D^+_r\cup L_r)))$.
\end{prop}

Note that when $\alpha=0$ this proposition has been proved in \cite{JWZZ2}. When $\alpha>0$, to get the $L^1$ integral of $u^+$, we need a trick which was introduced in \cite{BT} and also was used in  \cite{JZZ3}. That is, by using the fact that for some $t>0$
$$
\int_{D^+_r}\frac{1}{|x|^{2t\alpha}}dx\leq C,
$$
we can choose $s=\frac t{t+1}\in (0,1)$ when $\alpha>0$ and $s=1$ when $\alpha=0$ such that
$$
2s\int_{D^+_r}u^+dx\leq \int_{D^+_r}e^{2su}dx\leq (\int_{D^+_r}|x|^{2\alpha}e^{2u}dx)^s
(\int_{D^+_r}|x|^{-2t\alpha}dx)^{1-s}<\infty.
$$
Once we get the $L^1$ integral of $u^+$, we can get the conclusion of Proposition \ref{prop-b} by use the same argument in \cite{JWZZ2}. We omit the proof here.

\
\

We call $(u,\psi)$ a \emph{regular} solution to \eqref{Eq-B} if $u\in C^{2}(D^+_r)\cap C^{1}(D^+_r\cup L_r)$ and
$\Psi \in C^{2}(\Gamma(\Sigma D^+_r))\cap C^{1}(\Gamma(\Sigma (D^+_r\cup L_r)))$.

\
\

Next we consider the convergence of a sequence of regular solutions to
(\ref{Eq-B})  under a smallness condition for the  energy. We assume that $(u_n, \Psi_n)$ satisfy that
\begin{equation}\label{Eq-BN}
\left\{
\begin{array}{rcll}
-\Delta u_n(x) &=& 2V^2(x)|x|^{2\alpha}e^{2u_n(x)}-V(x)|x|^{\alpha}e^{u_n(x)}|\Psi_n|^2,  \quad &\text {in } D^+_{r}, \\
\slashiii{D}\Psi_n &=& -V(x)|x|^{\alpha}e^{u_n(x)}\Psi_n, \quad &\text {in } D^+_{r}, \\
\frac {\partial u_n}{\partial n}& = & cV(x)|x|^{\alpha}e^{u_n(x)},\quad & \text { on } L_r, \\
B(\Psi_n) &=& 0,\quad & \text { on } L_r,
\end{array}
\right.
\end{equation}
with the energy condition
\begin{equation}\label{condition}\int_{D^+_r}|x|^{2\alpha}e^{2u_n}+|\Psi_n|^4 dv+\int_{L_r}|x|^{\alpha}e^{u_n} d\sigma
<C \end{equation}
for some constant $C>0$.
First, we study the small energy regularity, i.e. when the energy $\int_{D^+_r}|x|^{2\alpha}e^{2u_n}dx$ and $\int_{L_r}|x|^{\alpha}e^{u_n}dx$ are  small enough,
$u_n$ will be uniformly bounded from above. Our Lemma is:

\begin{lm}\label{lmbm4} For $\varepsilon_1<\pi$, and $\varepsilon_2<\pi$. If a sequence of regular solutions $(u_n,\Psi_n)$ to (\ref{Eq-BN}) with
 $$\int_{D^+_r}2V^2(x)|x|^{2\alpha}e^{2u_n}dx<\varepsilon_1, \quad |c|\int_{L_r}V(x)|x|^\alpha e^{u_n}d\sigma<\varepsilon_2, \quad
\int_{D^+_r}|\Psi_n|^4dx<C$$ for some fixed constant $C>0$, we have that $||u^+_n||_{L^\infty(\overline{D}^+_{\frac r4})}$ and
$||\Psi_n||_{L^\infty(\overline{D}^+_{\frac r8})}$ are uniformly
bounded.
\end{lm}
\begin{proof} As the same situation as in Proposition \ref{prop-b}, we can no longer
use the inequality $2\int {u^+_n}<\int {e^{2u_n}}$  to get  the uniform bound of the $L^1$-integral of $u^+_n$ when  $\alpha>0$. But notice that there exists a constant $t>0$ such that
$$
\int_{D^+_r}\frac{1}{|x|^{2t\alpha}}dx\leq C.
$$
Setting $s=\frac t{t+1}\in (0,1)$,  then we  obtain
$$
2s\int_{D^+_r}u^+_ndx\leq \int_{D^+_r}e^{2su_n}dx\leq (\int_{D^+_r}|x|^{2\alpha}e^{2u_n}dx)^s
(\int_{D^+_r}|x|^{-2t\alpha}dx)^{1-s}<C.
$$
Then by a similar argument as in the proof of Lemma 3.5 in \cite{JWZZ2} we can prove this Lemma.
\end{proof}

\

When the energy $\int_{D^+_r}2V^2(x)|x|^{2\alpha}e^{2u_n}+\int_{L_r}V(x)|x|^{\alpha}e^{u_n}ds $ is large, in general, blow-up phenomenon may occur, i.e., Theorem \ref{mainthm} holds.

\
\

\begin{rem}\label{rem3.4}Let $v_n=u_n+\alpha\log|x|$, then $(v_n, \Psi_n)$ satisfies
 \begin{equation*}
\left\{
\begin{array}{rcll}
-\Delta v_n(x) &=& 2V^2(x)e^{2v_n(x)}-V(x)e^{v_n(x)}|\Psi_n|^2,     \quad  &\text {in } D^+_r, \\
\slashiii{D}\Psi_n &=& -V(x)e^{v_n(x)}\Psi_n, \quad &\text { in } D^+_r,\\
\frac {\partial v_n}{\partial n}& = & cV(x)e^{v_n(x)}+ \pi \alpha \delta_{p=0},\quad & \text { on } L_r, \\
B\Psi_n &=& 0,\quad & \text { on } L_r,
\end{array}
\right.
\end{equation*}
with the energy condition
\begin{equation*}
\int_{D^+_r}e^{2v_n}+|\Psi_n|^4dx+\int_{L_r}e^{v_n}ds<C.
\end{equation*}
Then, by using similar arguments as in \cite{BT}, the two blow-up sets of $u_n$ and $v_n$ are the same.
To show this conclution, it is sufficient to show the point $x=0$
is a blow-up point for $u_n$ if and only if it is a blow-up point for $v_n$. In fact, if $0$ is the only blow-up point for $v_n$
in a small neighbourhood $D^+_{\delta_0}\cup L_{\delta_0}$, that is, for any $\delta\in (0,\delta_0)$, $\exists C_\delta>0$, such that
\begin{equation}\label{tt}\max_{\overline{D^+_{\delta_0}\setminus D^+_\delta}}v_n\leq C_\delta,\quad \text { and }
 \max_{\overline{D^+_{\delta_0}}}v_n\rightarrow +\infty,\end{equation}
then, it is easy to see that $0$ is also the only blow-up point for $u_n$
in a small neighbourhood $D^+_{\delta_0}\cup L_{\delta_0}$, that is, for any $\delta\in (0,\delta_0)$, $\exists C_\delta>0$, such that
\begin{equation}\label{tt1}\max_{\overline{D^+_{\delta_0}\setminus D^+_\delta}}u_n\leq C_\delta,\quad \text { and }
 \max_{\overline{D^+_{\delta_0}}}u_n\rightarrow +\infty.\end{equation}
In converse, we assume that  $0$ is the only blow-up point for $u_n$ in a small neighbourhood $D^+_{\delta_0}\cup L_{\delta_0}$
such that (\ref{tt1}) is holds. We argue by contradiction and suppose that there exists a uniform constant $C$, such that $v_n(x)\leq C$
for any $x\in \overline{D}^+_{\delta_0}$. First, we can obtain that there exists a uniform constant $C$,
such that $|\Psi_n|^2(x)\leq C$ for any $x\in \overline{D}^+_{\frac{\delta_0}2}$. For this purpose, we
extend $(v_n,\Psi_n)$ to the lower half disk $D^-_{r}$. Assume $\bar{x}
$ is the reflection point of $x$ about $\partial \R^2_+$, and define
\begin{eqnarray*}
v_n(\bar{x}) :=& v_n(x),  \quad \bar{x} \in D^{-}_{r},\\
\Psi_n(\bar{x})  :=& ie_1\cdot \Psi_n(x),  \quad \bar{x}\in D^{-}_{r},\\
A_n(x):=& \left \{\begin{matrix} e^{v_n(x)},\quad x\in D^{+}_{r},\\
e^{v_n(\bar{x})}, \quad x\in D^-_{r}.\end{matrix} \right.
\end{eqnarray*}
Then $\Psi_n$ satisfies
$$
\slashiii{D}\Psi_n=-A_n(x)\Psi_n,\quad \text { in } D_r. $$
Since $A_n(x)$ is uniformly bounded in $ L^{\infty}(D_{\delta_0})$ and $\int_{D_{\delta_0}}|\Psi_n|^4dx<C$, we
have $\Psi_n$ is uniformly bounded in $ W^{1,\frac 43}(\Gamma(\Sigma D_{\frac {\delta_0}2}))$ and in
particular $\Psi_n$ is uniformly bounded $ C^\gamma(\Gamma( \Sigma \overline{D}^+_{\frac
{\delta_0}2}))$ for some $0<\gamma<1$. Further, since
$$
f_n(x):=2V_n^2(x)|x|^{2\alpha}e^{2u_n(x)}-V_n(x)|x|^{\alpha}e^{u_n(x)}|\Psi_n|^2 =2V_n^2(x)e^{2v_n(x)}-V_n(x)e^{v_n(x)}|\Psi_n|^2
$$
and
$$
g_n:=-V_n(x)|x|^{\alpha}e^{u_n(x)}\Psi_n =- V_n(x)e^{v_n(x)}\Psi_n
$$
are uniformly bounded in  $\overline{D}^+_{\frac{\delta_0}2}$. Then
by Harnack type inequality of Neumann boundary problem  (see Lemma A.2 in \cite{JWZZ2}),
it follows that $\inf_{\overline{D}^+_{\frac{\delta_0}2}}u_n\rightarrow +\infty$. Thus we get a contradiction
since the blow-up set of $u_n$ is finite.
\end{rem}

\section{Removability of Local Sigularities}

The Pohozaev indenty is closely related to the removability of singularities.  In this section, we shall first establish the Pohozaev identiy for regular solutions to (\ref{Eq-B}). Then for solutions defined on a domain with isolated singularity, we define a constant which is called the Pohozaev constant. The most important is that a necessary and sufficient condition for the removability of local singularities is the vanishing of Pohazaev constant.

\begin{prop}\label{prop-poho}(Pohozaev indenty)
Let $(u,\Psi)$ be a regular solution of (\ref{Eq-B}), that is $(u,\Psi)$ satisfies
\begin{equation*}
\left\{
\begin{array}{rcll}
-\Delta u(x) &=& 2V^2(x)|x|^{2\alpha}e^{2u(x)}-V(x)|x|^{\alpha}e^{u(x)}|\Psi|^2,  \quad &\text { in } D^+_{R}, \\
\slashiii{D}\Psi &=& -V(x)|x|^{\alpha}e^{u(x)}\Psi, \quad &\text { in } D^+_{R}, \\
\frac {\partial u}{\partial n}& = & cV(x)|x|^{\alpha}e^{u(x)},\quad & \text { on } L_R, \\
B\Psi &=& 0,\quad & \text { on } L_R.
\end{array}
\right.
\end{equation*}
Then we have the following Pohozaev identity
\begin{eqnarray}\label{poho-1}
&&R\int_{S^+_R} |\frac {\partial u}{\partial
\nu}|^2-\frac
12|\nabla u|^2d\sigma \nonumber\\
&=& (1+\alpha)\{\int_{D^+_R}(2V^2(x)|x|^{2\alpha}e^{2u}-V(x)|x|^{\alpha}e^u|\Psi|^2)dv+
\int_{L_R}cV(x)|x|^{\alpha}e^{u}ds \} \nonumber\\
& & -R\int_{S^+_R}V^2(x)|x|^{2\alpha}e^{2u}d\sigma +\int_{L_R}c\frac {\partial V(s,0)}{\partial s}|s|^{\alpha}se^{u(s,0)}ds-cV(s,0)|s|^{\alpha}se^{u(s,0)}|^{s=R}_{s=-R} \nonumber\\
& &+\int_{D^+_R}x\cdot \nabla (V^2(x))|x|^{2\alpha}e^{2u}dv -\int_{D^+_R}x\cdot \nabla
V(x)|x|^{\alpha}e^u|\psi|^2dv \nonumber\\
& & + \frac 14\int_{S^+_R}\la\frac {\partial
\Psi}{\partial \nu},(x+\bar{x})\cdot\Psi\ra d\sigma+\frac
14\int_{S^+_R}\la (x+\bar{x})\cdot\Psi, \frac
{\partial \Psi}{\partial \nu}\ra d\sigma,
\end{eqnarray}
where $\nu$ is the outward normal vector to $S^+_R$, and $\bar{x}$ is the reflection point of $x$ about $\partial \R^2_+$.
\end{prop}
\begin{proof} The case of  $\alpha=0$ and $V\equiv1$ has already been treated in \cite{JZZ1}. The calculation of the Pohozaev identity is standard. Since in the sequel we will need to calculate the Pohozaev identity for different equations, for reader's convenience, we give the detailed proof for this general case.

First, we multiply the first equation by $x\cdot \nabla u$ and integrate over $D^+_R$ to obtain
\[-\int_{D^+_R}\triangle u x\cdot \nabla
udv=\int_{D^+_R}2V^2(x)|x|^{2\alpha}e^{2u}x\cdot \nabla udv-\int_{D^+_R}V(x)|x|^\alpha e^u|\Psi|^2
x\cdot \nabla udv .\]

It follows from direct computations that
\begin{eqnarray*}
& & \int_{D^+_R}\triangle u x\cdot \nabla udv\\
&=& R\int_{S^+_R} |\frac {\partial u}{\partial
\nu}|^2-\frac 12|\nabla u|^2d\sigma+\int_{L_R}\frac{\partial u}{\partial n}(x\cdot \nabla u)ds\\
&=& R\int_{S^+_R} |\frac {\partial u}{\partial
\nu}|^2-\frac 12|\nabla u|^2d\sigma+\int_{L_R}cV(x)|x|^{\alpha}e^u(x\cdot \nabla u)ds\\
&=& R\int_{S^+_R} |\frac {\partial u}{\partial
\nu}|^2-\frac 12|\nabla u|^2d\sigma-(\alpha+1)\int_{L_R}cV(x)|x|^{\alpha}e^{u}ds\\
& & -\int_{L_R}c\frac {\partial V(s,0)}{\partial s}|s|^{\alpha}se^{u(s,0)}ds+cV(s,0)|s|^{\alpha}se^{u(s,0)}|^{s=R}_{s=-R},
\end{eqnarray*}

\begin{eqnarray*}
& & \int_{D^+_R}2V^{2}(x)|x|^{2\alpha}e^{2u}x\cdot \nabla udv\\
& = & R\int_{S^+_R}V^2(x)|x|^{2\alpha}e^{2u}d\sigma-(2+2\alpha)\int_{D^+_R}V^2(x)|x|^{2\alpha}e^{2u}dv
-\int_{D^+_R}x\cdot \nabla (V^2(x))|x|^{2\alpha}e^{2u}dv,
\end{eqnarray*}

and
\begin{eqnarray*}
& & \int_{D^+_R}V(x)|x|^{\alpha}e^u|\Psi|^2
x\cdot \nabla udv\\
& = & R\int_{S^+_R}V(x)|x|^{\alpha}e^u|\Psi|^2d\sigma-\int_{D^+_R}|x|^{\alpha}e^ux\cdot \nabla
(V(x)|\Psi|^2)dv-(2+\alpha)\int_{D^+_R}V(x)|x|^{\alpha}e^u|\Psi|^2dv.
\end{eqnarray*}

Therefore we have
\begin{eqnarray}\label{4.2}
&&R\int_{S^+_R} |\frac {\partial u}{\partial
\nu}|^2-\frac
12|\nabla u|^2d\sigma \nonumber\\
&=& (1+\alpha)\int_{D^+_R}2V^2(x)|x|^{2\alpha}e^{2u}dv-(2+\alpha)\int_{D^+_R}V(x)|x|^{\alpha}e^u|\Psi|^2dv+
(\alpha+1)\int_{L_R}cV(x)|x|^{\alpha}e^{u}ds \nonumber\\
& & -R\int_{S^+_R}V^2(x)|x|^{2\alpha}e^{2u}d\sigma+R\int_{S^+_R}V(x)|x|^{\alpha}e^u|\Psi|^2d\sigma \nonumber\\
& &  +\int_{L_R}c\frac {\partial V(s,0)}{\partial s}|s|^{\alpha}se^{u(s,0)}ds-cV(s,0)|s|^{\alpha}se^{u(s,0)}|^{s=R}_{s=-R} \nonumber\\
& &+\int_{D^+_R}x\cdot \nabla (V^2(x))|x|^{2\alpha}e^{2u}dv -\int_{D^+_R}|x|^{\alpha}e^ux\cdot \nabla
(V(x)|\Psi|^2)dv
\end{eqnarray}

\
\

On the other hand, for $x\in \R^2_+$, we denote $x=x_1e_1+x_2e_2$ under the local orthonormal basis $\{e_1,e_2\}$ on $\R^2_+$.  Using the Clifford multiplication relation
\[
e_i\cdot e_j+e_j\cdot e_i=-2\delta _{ij},\text{ for }1\leq i,j\leq 2
\] and
\[
\left\langle \psi ,\varphi \right\rangle =\left\langle e_i\cdot \psi
,e_i\cdot \varphi \right\rangle
\]
for any spinors $\psi ,\varphi \in \Gamma (\Sigma M)$. We know that
\begin{equation}\label{4.3}
\la\psi,e_i\cdot\psi\ra +\la e_i\cdot\psi,\psi\ra=0
\end{equation}
for any $i=1,2$. Using the chirality boundary condition of $\Psi$, we
extend $(u,\Psi)$ to the lower half disk $D^-_{R}$. Assume
$\bar{x} $ is the reflection point of $x$ about $\partial \R^2_+$,
and define
\begin{equation}
\label{reflec1} u(\bar{x}) := u(x), \quad \bar{x} \in D^{-}_{R},
\end{equation}
\begin{equation}
\label{reflec2} \Psi(\bar{x})  := ie_1 \cdot \Psi(x),   \quad \bar{x}\in D^{-}_{R}.
\end{equation}
Then it follows from the argument in Lemma 3.4 of \cite{JWZZ2} that we
obtain
$$
\slashiii{D}\psi=-A(x)\psi \quad \text { in } D_R. $$ Here
\begin{equation*}
A(x)=\left \{\begin{matrix} V(x)|x|^{\alpha}e^{u(x)},\quad x\in D^{+}_{R},\\
V(\bar{x})|\bar{x}|^{\alpha}e^{u(\bar{x})}, \quad x\in D^-_{R}.\end{matrix} \right.
\end{equation*}
Using the Schr\"{o}dinger-Lichnerowicz formula
$\slashiii{D}^2=-\triangle +\frac 12 K_g$, we have
\begin{equation}\label{4.4}
-\triangle \Psi=-dA(x)\cdot\psi+A^2(x)\Psi \quad \text{ in } D_{R}.
\end{equation}
Then we multiply (\ref{4.4}) by $x\cdot\Psi$ (where $\cdot$ denotes
the Clifford multiplication) and integrate over $D_R$ to obtain
\[\int_{D_R}\la\triangle\Psi,x\cdot\Psi\ra dv=\int_{B_R}\la dA(x)\cdot\Psi,x\cdot\Psi\ra
dv-\int_{D_R}A^2(x)\la\Psi,x\cdot \Psi\ra dv,\]

and
\[\int_{D_R}\la x\cdot \Psi,\triangle\Psi\ra dv=\int_{D_R}
\la x\cdot \Psi,dA(x)\cdot\Psi\ra dv-\int_{D_R}A^2(x)\la x\cdot
\Psi, \Psi\ra dv .\]

\noindent On the other hand, by partial integration,
\begin{eqnarray*}
& & \int_{D_R}\la\triangle\Psi,x\cdot\Psi\ra dv \\
&=&\int_{D_R}div\la\nabla \Psi,x\cdot\Psi\ra
dv-\int_{D_R}\sum^2_{\alpha=1}\la
\nabla_{e_\alpha}\Psi,e_\alpha\cdot\Psi\ra dv-\int_{D_R}\la\nabla
\Psi,x\cdot\nabla \Psi\ra\\
&=&\int_{\partial D_R}\la\frac {\partial \Psi}{\partial
\nu},x\cdot\Psi\ra d\sigma+\int_{D_R}\la\slashiii{D}\Psi,\Psi\ra
dv-\int_{D_R}\la\nabla
\Psi,x\cdot\nabla \Psi\ra\\
&=& \int_{\partial D_R}\la\frac {\partial \Psi}{\partial
\nu},x\cdot\Psi\ra d\sigma-\int_{D_R} A(x)|\Psi|^2
dv-\int_{D_R}\la\nabla \Psi,x\cdot\nabla \Psi\ra,\\
&=&  \int_{ \partial D^+_R\cap \R^2_+}\la\frac {\partial
\Psi}{\partial \nu},(x+\bar{x})\cdot\Psi\ra d\sigma-2\int_{D^+_R}
V(x)|x|^{\alpha}e^u|\Psi|^2 dv-\int_{D_R}\la\nabla \Psi,x\cdot\nabla \Psi\ra,
\end{eqnarray*}
and similarly
\[
\int_{D_R}\la x\cdot\Psi,\triangle \Psi\ra=\int_{\partial D^+_R\cap
\R^2_+}\la (x+\bar{x})\cdot \Psi, \frac {\partial \Psi}{\partial \nu}\ra
d\sigma-2\int_{D^+_R}V(x)|x|^{\alpha} e^u|\Psi|^2 dv-\int_{D_R}\la x\cdot\nabla
\Psi,\nabla \Psi\ra.
\]

\noindent Furthermore we also have
\begin{eqnarray*}
&& \int_{D_R}\la dA(x)\cdot \Psi,x\cdot \Psi\ra dv +\int_{D_R}\la
x\cdot\Psi, dA(x)\cdot \Psi\ra dv\\
&=&\int_{D_R}\sum^2_{\alpha,\beta=1}\la\nabla_{e_\alpha}A(x)e_\alpha\cdot\Psi,e_\beta\cdot\Psi\ra
 x_\beta dv+\int_{D_R}
\sum^2_{\alpha,\beta=1}\la e_{\beta}\cdot
\Psi,\nabla_{e_\alpha}A(x)e_\alpha\cdot\Psi\ra
 x_\beta dv\\
 &=& 2\int_{D_R}\sum^2_{\alpha=1}\la\nabla_{e_\alpha}A(x)e_\alpha\cdot\Psi,e_\alpha\cdot\Psi\ra
 x_\alpha dv\\
 &=& 2\int_{D_R}x\cdot\nabla (A(x))|\Psi|^2dv\\
 &=&-2\int_{D_R}A(x)x\cdot\nabla(|\Psi|^2)dv-4\int_{D_R}A(x)|\Psi|^2dv+2R\int_{\partial
 D_R}A(x)|\Psi|^2dv\\
 &=& -4\int_{D^+_R}V(x)|x|^{\alpha}e^u x\cdot\nabla(|\Psi|^2)dv-8\int_{D^+_R}V(x)|x|^{\alpha}e^u |\Psi|^2dv+4R\int_{\partial
 D^+_R\cap \R^2_+} V(x)|x|^{\alpha}e^u|\Psi|^2dv.
\end{eqnarray*}
Therefore we obtain
\begin{eqnarray}\label{4.5}
& & R\int_{\partial
D^+_R\cap \R^2_+}V(x)|x|^{\alpha}e^u|\Psi|^2d\sigma-\int_{D^+_R}V(x)|x|^{\alpha}e^ux\cdot\nabla(|\Psi|^2)dv\nonumber \\
&= &\frac 14\int_{\partial D^+_R\cap \R^2_+}\la\frac {\partial
\Psi}{\partial \nu},(x+\bar{x})\cdot\Psi\ra d\sigma+\frac
14\int_{\partial D^+_R\cap \R^2_+}\la (x+\bar{x})\cdot\Psi, \frac
{\partial \Psi}{\partial \nu}\ra d\sigma\nonumber \\
& & +\int_{D^+_R} V(x)|x|^{\alpha}e^u|\Psi|^2 dv.
\end{eqnarray}
Putting (\ref{4.2}) and (\ref{4.5}) together, we obtain our Pohozaev type
identity (\ref{poho-1}).
\end{proof}

Pohozaev type identity is shown to be closely related to the removablity of local singularities of solutions. For a solution of \eqref{Eq-L} and \eqref{Co-L}, we defined in \cite{JZZ3} the following Pohozaev constant:
\begin{df}[\cite{JZZ3}] \label{poho-const} Let $(u,\Psi)\in C^2(D_r\backslash\{0\})\times C^2(\Gamma(\Sigma (D_r\backslash\{0\})))$ be a solution of \eqref{Eq-L} and \eqref{Co-L}. For $0<R<r$, we define the {\em Pohozaev constant} with respect to the equations \eqref{Eq-L} with the constraint \eqref{Co-L} as follows:
 \begin{eqnarray*}
C(u,\Psi)&:=& R\int_{\partial D_R(0)} |\frac {\partial u}{\partial \nu}|^2-\frac 12|\nabla u|^2d\sigma \\
&-&(1+\alpha)\int_{D_R(0)}(2V^2(x)|x|^{2\alpha}e^{2u}-V(x)|x|^{\alpha}e^u|\Psi|^2)dx \\
& & +R\int_{\partial D_R(0)}V^2(x)|x|^{2\alpha}e^{2u}d\sigma-\frac 12\int_{\partial
D_R(0)}\la\frac {\partial \Psi}{\partial \nu}, x\cdot\Psi\ra +\la
x\cdot\Psi, \frac {\partial \Psi}{\partial \nu}\ra d\sigma\\
& & -\int_{D_R(0)}(|x|^{2\alpha}e^{2u}x\cdot\nabla (V^2(x))-|x|^{\alpha}e^u|\Psi|^2x\cdot \nabla
V(x))dx
\end{eqnarray*}
where $\nu$ is the outward normal vector of $\partial D_R(0)$.
\end{df}

It is clear that $C(u,\Psi)$ is independent of $R$ for $0<R<r$. Thus, the vanishing of the Pohozaev constant $C(u,\Psi)$ is equivalent to the {\em Pohozaev identity}
\begin{eqnarray}\label{poho}
&&R\int_{\partial D_R(0)} |\frac {\partial u}{\partial \nu}|^2-\frac 12|\nabla u|^2d\sigma \nn\\
&=&(1+\alpha)\int_{D_R(0)}(2V^2(x)|x|^{2\alpha}e^{2u}-V(x)|x|^{\alpha}e^u|\Psi|^2)dx \nn\\
& & -R\int_{\partial D_R(0)}V^2(x)|x|^{2\alpha}e^{2u}d\sigma+\frac 12\int_{\partial
D_R(0)}(\la\frac {\partial \Psi}{\partial \nu}, x\cdot\Psi\ra +\la
x\cdot\Psi, \frac {\partial \Psi}{\partial \nu}\ra ) d\sigma   \nn\\
& & +\int_{D_R(0)}(|x|^{2\alpha}e^{2u}x\cdot\nabla (V^2(x))-|x|^{\alpha}e^u|\Psi|^2x\cdot \nabla
V(x))dx
\end{eqnarray}
for  a solution $(u,\Psi)\in C^2(D_r(0))\times C^2(\Gamma(\Sigma D_r(0)))$  of (\ref{Eq-L}) and (\ref{Co-L}).

We also proved in \cite{JZZ3} that a local singularity is removable iff the Pohozaev identity \eqref{poho} holds, that is, iff the Pohozaev constant vanishes.
\begin{thm}[\cite{JZZ3}]\label{thm-sigu-move1}
Let $(u,\Psi)\in C^2(D_r\setminus\{0\})\times C^2(\Gamma(\Sigma (D_r\setminus \{0\})))$  be  a solution of (\ref{Eq-L}) and (\ref{Co-L}).
Then there is a constant $\gamma < 2\pi (1+\a)$ such that
\begin{eqnarray*}
u(x)=- \frac{\gamma}{2\pi}{\rm log} |x| + h, \quad {\rm near }\ 0,
\end{eqnarray*}
where $h$ is bounded near $0$. The Pohozaev constant  $C(u, \Psi) $ and $\gamma$ satisfy:
\begin{eqnarray*}
 C(u, \Psi)=\frac{\gamma^2}{4\pi}.
\end{eqnarray*}
In particular, $(u,\Psi)\in C^2(D_r)\times C^2(\Gamma(\Sigma D_r))$, i.e. the local singularity of $(u,\Psi)$ is removable,   iff $C(u,\Psi)=0$.
\end{thm}

For the singular boundary problem \eqref{Eq-B}, we can define the Pohozaev constant in a similar way:

\begin{df}

Let $(u,\Psi) \in C^{2}(D^+_r)\cap C^{1}(D^+_r\cup L_r\backslash\{0\})\times  C^{2}(\Gamma(\Sigma D^+_r))\cap C^{1}(\Gamma(\Sigma (D^+_r\cup L_r\backslash\{0\})))$ be a solution of \eqref{Eq-B} and \eqref{Eq-BC}. For $0<R<r$, we define the {\em Pohozaev constant} with respect to the equations \eqref{Eq-B} with the constraint \eqref{Eq-BC} as follows:

\begin{eqnarray*}
& & C_B(u,\Psi) :=  R\int_{\partial D^+_R\cap \R^2_+} |\frac {\partial u}{\partial
\nu}|^2-\frac
12|\nabla u|^2d\sigma \nonumber\\
& & -(1+\alpha)\int_{D^+_R}(2V^2(x)|x|^{2\alpha}e^{2u}-V(x)|x|^{\alpha}e^u|\Psi|^2)dv-
(\alpha+1)\int_{\partial D^+_R\cap \partial \R^2_+}cV(x)|x|^{\alpha}e^{u}ds \nonumber\\
& & +R\int_{\partial D^+_R\cap \R^2_+}V^2(x)|x|^{2\alpha}e^{2u}d\sigma -\int_{\partial D^+_R\cap
\partial \R^2_+}c\frac {\partial V(s,0)}{\partial s}|s|^{\alpha}se^uds+cV(s,0)|s|^{\alpha}se^u|^{s=R}_{s=-R} \nonumber\\
& &-\int_{D^+_R}x\cdot \nabla (V^2(x))|x|^{2\alpha}e^{2u}dv +\int_{D^+_R}x\cdot \nabla
V(x)|x|^{\alpha}e^u|\psi|^2dv \nonumber\\
& & - \frac 14\int_{\partial D^+_R\cap \R^2_+}\la\frac {\partial
\Psi}{\partial \nu},(x+\bar{x})\cdot\Psi\ra d\sigma-\frac
14\int_{\partial D^+_R\cap \R^2_+}\la (x+\bar{x})\cdot\Psi, \frac
{\partial \Psi}{\partial \nu}\ra d\sigma.
\end{eqnarray*}
\end{df}

The removability theorem of a local singularity at the boundary is following:

\begin{thm}(Removability of a local boundary singularity)\label{thm-sigu-move1-B}
Let $(u,\Psi) \in C^{2}(D^+_r)\cap C^{1}(D^+_r\cup L_r\backslash\{0\})\times  C^{2}(\Gamma(\Sigma D^+_r))\cap C^{1}(\Gamma(\Sigma (D^+_r\cup L_r\backslash\{0\})))$ be a solution of \eqref{Eq-B} and \eqref{Eq-BC}, then there is a constant $\gamma < \pi (1+\a)$ such that
\begin{eqnarray*}
u(x)=- \frac{\gamma}{2\pi}{\rm log} |x| + h, \quad {\rm near }\ 0,
\end{eqnarray*}
where $h$ is bounded near $0$. The Pohozaev constant  $C(u, \Psi) $ and $\gamma$ satisfy:
\begin{eqnarray*}
 C(u, \Psi)=\frac{\gamma^2}{2\pi}.
\end{eqnarray*}
In particular, $(u,\Psi) \in C^{2}(D^+_r)\cap C^{1}(D^+_r\cup L_r)\times  C^{2}(\Gamma(\Sigma D^+_r))\cap C^{1}(\Gamma(\Sigma (D^+_r\cup L_r)))$, i.e. the local singularity of $(u,\Psi)$ is removable,   iff $C(u,\Psi)=0$.
\end{thm}

To prove Theorem \ref{thm-sigu-move1-B}, we need to derive the decay of spinor part $\Psi$ near the singular point. For the case of $\alpha=0$ and $V(x)=1$, this is shown in \cite{JZZ1}.
By using similar arguments, we can also get the following lemma for the general case:

\begin{lm}\label{asy-phi}
There are  $0<\varepsilon_1 < 2\pi$ and
$0<\varepsilon_2<\pi$ such that $(v,\phi)$
satisfy
\begin{equation*}
\left\{
\begin{array}{rcll}
-\Delta v &=& 2V^2(x)|x|^{2\alpha}e^{2v}-V(x)|x|^{\alpha}e^v\left\langle \phi ,\phi
\right\rangle,&\qquad \text { in } B_{r_0}^+,
\\
\slashiii{D}\phi &=&\ds  -V(x)|x|^{\alpha}e^v\phi, &\qquad \text { in } B_{r_0}^+,\\
\frac{\partial v}{\partial n} &=& cV(x)|x|^{\alpha}e^v, &\qquad \text { on }
L_{r_0}\backslash\{0\},\\
B\phi &=& 0, &\qquad \text { on } L_{r_0}\backslash\{0\},
\end{array}
\right.
\end{equation*}
with energy conditions
\begin{equation*}
 \int_{B^+_{r_0}}|x|^{2\alpha}e^{2v}dx\leq \varepsilon_1<2\pi,~~~
\int_{B^+_{r_0}}|\phi|^4dx\leq C, ~~~
 |c|\int_{L_{r_0}}|x|^{\alpha}e^{v}ds\leq \varepsilon_2<\pi.
\end{equation*}
Then for any $x\in \overline {B}^+_{\frac {r_0}{2}}$ we have
\begin{equation}\label{asy-phi1}
|\phi(x)||x|^{\frac 12}+|\nabla\phi(x)||x|^{\frac 32}\leq
C(\int_{B^+_{2|x|}}|\phi|^4dx)^{\frac 14}.
\end{equation}
Furthermore, if we assume that $e^{2v}=O(\frac
{1}{|x|^{2(1+\alpha)-\varepsilon}})$, then,  for any $x\in
\overline{B}^+_{\frac {r_0}2}$, we have
\begin{equation}\label{asy-phi11}
|\phi(x)||x|^{\frac 12}+|\nabla\phi(x)||x|^{\frac 32}\leq
C|x|^{\frac {1}{4C}}(\int_{B^+_{r_0}}|\phi|^4dx)^{\frac 14},
\end{equation}
for some positive constant $C$. Here $\varepsilon$ is any
sufficiently small positive number.
\end{lm}

\noindent{\bf Proof of Theorem \ref{thm-sigu-move1-B}:}   By the conformal invariance, we assume without loss of generality that $
 \int_{B^+_{r}}|x|^{2\alpha}e^{2v}dx\leq \varepsilon_1$ and $|c|\int_{L_{r}}|x|^{\alpha}e^{v}ds\leq \varepsilon_2 $ where $\varepsilon_1$ and $\varepsilon_2$ are as in Lemma \ref{asy-phi}.
By standard potential analysis, it follows that there is a
constant $\gamma$ such that
\begin{equation*}
\lim_{|x|\rightarrow 0}\frac {u}{-\log|x|}=\frac{\gamma}{\pi}.
\end{equation*}
By $\int_{D^+_r}|x|^{2\alpha}e^{2u}dx<C$, we obtain that $\gamma\leq \pi(1+\alpha)$.
Furthermore, by using Lemma \ref{asy-phi} and by a similar argument as in the proof of Proposition 5.4 of
\cite{JWZZ1}, we can improve  this to the strict inequality $\gamma<\pi(1+\alpha)$. Next we set
$$
v(x)=-\frac 1{\pi}\int_{B^+_{r}}\log|x-y|(2V^2(y)|y|^{2\alpha}e^{2u}-V(y)|y|^{\alpha}e^u|\Psi|^2)dy-\frac 1{\pi}\int_{L_{r}}\log|x-y|(cV(y)|y|^{\alpha}e^{u})d\sigma
$$
and set $w=u-v$. Notice that $v$ satisfies that
\begin{equation*}
\left\{
\begin{array}{rlll}
-\Delta v &= & 2V^2(x)|x|^{2\alpha}e^{2u}-V(x)|x|^{\alpha}e^u|\Psi|^2, &\quad \text { in } D^+_r,\\
\frac {\partial v}{\partial n}&=& cV(x)|x|^{\alpha}e^u, &\quad \text { on } L_r,
\end{array}
\right.
\end{equation*}
and  $w$ satisfies that
\begin{equation*}
\left\{
\begin{array}{rlll}
-\Delta w &= & 0, &\quad \text { in } D^+_r,\\
\frac {\partial v}{\partial n}&=& 0, &\quad \text { on } L_r\backslash\{0\}.
\end{array}
\right.
\end{equation*}
We can check that
$$\lim_{|x|\rightarrow 0}\frac {v(x)}{-\log|x|}=0.$$
Since we can extend $w$ to $B_r\backslash \{0\}$ evenly to get a harmonic function $w$ in $B_r\backslash \{0\}$, then we obtain that
$$
\lim_{|x|\rightarrow 0}\frac {w(x)}{-\log|x|}=\lim_{|x|\rightarrow
0}\frac{u-v}{-\log|x|}=\frac{\gamma}{\pi}.
$$
Duo to  $w$ is harmonic in $B_1\backslash \{0\}$ we have
$$
w=-\frac{\gamma}{\pi}\log|x|+w_0
$$ with a smooth harmonic function $w_0$ in $B_r$. Therefore we have
$$
u=-\frac{\gamma}{\pi}\log|x|+v+w_0 \quad \text{ near } 0.
$$

To compute the Pohozaev constant of $(u,\Psi)$ we need the decay of the gradient of $u$ near the singular point.
We denote that
 $f_1(x):=2V^2(x)|x|^{2\alpha}e^{2u(x)}$,  $ f_2(x):=-V(x)|x|^{\alpha}e^{u(x)}|\Psi|^2(x)$ and $f_3(x):= cV(x)|x|^{\alpha}e^u $. Since each $f_i$ is $L^1$ integrable, we can obtain  $e^{|v(x)|}\in L^p(D^+_r)$ for any $p\geq 1$ and $e^{|v(x)|}\in L^p(L_r)$ for any $p\geq 1$. Since
$$
f_1(x)=|x|^{-\frac {2\gamma}{\pi}+2\alpha }(2V^2(x)e^{2w_0(x)+2v(x)}),
$$
$$
f_2(x)=-|x|^{-\frac {\gamma}{\pi}+\alpha-1}(V(x)e^{w_0(x)+v(x)}|x||\Psi|^2(x)),
$$
and
$$
f_3(x)=|x|^{-\frac {\gamma}{\pi}+\alpha}(cV(x)e^{w_0(x)+v(x)}),
$$
we set $s_1=\frac {2\gamma}{\pi}-2\alpha $ and  $s_2=\frac {\gamma}{\pi}-\alpha+1 $. Then $\max\{s_1, s_2\}=s_2<2$.
Since $|\Psi|\leq C|x|^{-\frac 12}$ near $0$ and $w_0(x)$ is smooth in $B_r$, we have by H\"older's inequality that $f_1\in L^{t}(D^+_r)$ for any $t\in (1,\frac 2{s_1})$ if $s_1>0$, and $f_1\in L^{t}(D^+_r)$ for any $t>1$ if $s_1\leq 0$. For $f_2$,  we  have $f_2\in L^t(D^+_r)$ for any $t\in (1,\frac 2{s_2})$ if $s_2>0$, and $f_2\in L^{t}(D^+_r)$ for any $t>1$ if $s_2\leq 0$. For $f_3$,  we  have $f_3\in L^t(L_r)$ for any $t\in (1,\frac 2{s_1})$ if $s_1>0$, and $f_3\in L^{t}(L_r)$ for any $t>1$ if $s_1\leq 0$. Putting all together and by standard elliptic theory, we have $v(x)$ is in $L^{\infty}(\overline{D}^+_r)$.
On the other hand, since $v(x)$ is in $L^{\infty}(\overline{D}^+_r)$, it follows from Lemma  \ref{asy-phi} that there exists a small $\delta_0>0$ such that
$$
|\Psi|\leq C|x|^{\delta_0-\frac 12}, \quad \text{ near } 0,
$$
and
$$
|\nabla \Psi|\leq C|x|^{\delta_0-\frac 32}, \quad \text{ near } 0.
$$
Next we estimate $\nabla v(x)$. If $s_1<0$ and $s_2<0$, then $v(x)$ is in $C^1(\overline{B}^+_r)$.
If $s_1>0$ or $s_2>0$, $\nabla v(x)$ will have a decay when $|x|\rightarrow 0$. Without loss of generality, we assume that $0<s_1<2$ and $0<s_2<2$.  For any $x\in D^+_r$ we hanve
\begin{eqnarray*}
& & |\nabla v(x)|
\leq  \frac{1}{\pi}\int_{D^+_r}\frac{1}{|x-y|}(|f_1(y)|+|f_2(y)|)dy+ \frac{1}{\pi}\int_{L_r}\frac{1}{|x-y|}|f_3(y)|dy\\
 & = & \frac{1}{\pi}\int_{\{|x-y|\geq\frac {|x|}2\}\cap {D^+_r}}\frac{1}{|x-y|}(|f_1(y)|+|f_2(y)|)dy+ \frac{1}{\pi}\int_{\{|x-y|\leq\frac {|x|}2\}\cap {D^+_r}}\frac{1}{|x-y|}(|f_1(y)|+|f_2(y)|)dy\\
& & +  \frac{1}{\pi}\int_{\{|x-y|\geq\frac {|x|}2\}\cap {L_r}}\frac{1}{|x-y|}|f_3(y)|dy+ \frac{1}{\pi}\int_{\{|x-y|\leq\frac {|x|}2\}\cap {L_r}}\frac{1}{|x-y|}|f_3(y)|dy\\
& = & I_1+I_2+I_3+I_4.
\end{eqnarray*}
Fix $t\in (1,\frac 2{s_2})$ and choose $0<\tau_1<1$ such that $\frac {\tau_1 t}{t-1}<2$. Hence, we have $0<\tau_1 <2-s_2$. Then by H\"{o}lder's inequality we obtain
\begin{eqnarray*}
I_1 \leq  (\int_{\{|x-y|\geq\frac {|x|}2\}\cap {D^+_r}}\frac 1{|x-y|^{\frac {\tau_1t}{t-1}}}dy)^{\frac {t-1}t}(\int_{\{|x-y|\geq\frac {|x|}2\}\cap {D^+_r}}\frac{1}{|x-y|^{(1-\tau_1)t}}(|f_1|+|f_2|)^tdy)^{\frac 1t} \leq  \frac {C}{|x|^{1-\tau_1}}.
\end{eqnarray*}
For $I_2$, since $y\in\{y||x-y|\leq \frac{|x|}2\}$ implies that $|y|\geq \frac{|x|}2$, we can get that
\begin{eqnarray*}
I_2 \leq  C \int_{\{|x-y|\leq\frac {|x|}2\}\cap {D^+_r}}\frac{1}{|x-y||y|^{s_2}}dy \nonumber \leq  C|x|^{1-s_2}.
\end{eqnarray*}
Similarly,  for $I_3$, we fix $t\in(1,\frac{2}{s_1})$ and choose $\tau_{2}>0$ such that $\frac{\tau_{2} t}{t-1}<1$, and hence we have $0<\tau_{2} <1-\frac{s_1}{2}$. By Holder's inequality we obtain,
\begin{eqnarray*}
 | I_{3}| &\leq & \frac {1}{\pi}(\int _{\{|x-y|\geq \frac{|x|}{2}\}\cap L_{r}}\frac{1}{|x-y|^{\frac{\tau_{2} t}{t-1}}}dy)^{\frac{t-1}{t}}(\int _{\{|x-y|\geq \frac{|x|}{2}\}\cap L_{r}}\frac{1}{|x-y|^{t(1-\tau_{2})}}|f_3(y)|^tdy)^{\frac 1t}\\
 & \leq &  \frac{C}{|x|^{1-\tau_{2}}}.
\end{eqnarray*}
For $I_{4}$ we have
\begin{eqnarray*}
|I_{4}|&\leq & C\int _{\{|x-y|\leq \frac{|x|}{2}\}\cap L_{r}}\frac{1}{|x-y|}\frac{1}{|y|^{\frac {s_1}2}}dy\\
& \leq &  \frac{C}{|x|^{\frac {s_1}2}}\int _{\{|x-y|\leq \frac{|x|}{2}\}\cap L_{r}}\frac{1}{|x-y|}dy\leq \frac{C}{|x|^{\tau_{3}}},
\end{eqnarray*} for some $\tau_{3}$ with $0<\tau_{3}<1$.
In conclusion, for all $x\in B_{r}^+(0)$ we have
\begin{equation}\label{2.15}
|\nabla v(x)|\leq \frac{C}{|x|^{1-\tau_{1}}}+\frac{C}{|x|^{1-\tau_{2}}}+\frac{C}{|x|^{\tau_{3}}}
\end{equation}
for suitable constants $0<\tau_{1} <2-s_2$, $0<\tau_{2} <1-\frac{s_1}2$ and $0<\tau_{3}<1$.

At this point we are ready to compute the Pohozaev constant $C(u, \Psi)$. We denote
$$\nabla u=-\frac{\gamma}{\pi}\frac x{|x|^2}+\nabla (w_0+v(x))=-\frac{\gamma}{\pi}\frac x{|x|^2}+\nabla \eta(x).$$
By \eqref{2.15}, we have
\begin{eqnarray*}
& & r\int_{S^+_{r}} (\frac{1}{2}| \nabla u|^{2}-| \frac{\partial u}{\partial \nu}|^{2})ds   \\
&=&r\int _{S^+_{r}}\frac{1}{2}[(\frac{\gamma}{\pi})^2\frac 1{|x|^2}-2\frac{\gamma}{\pi}\frac{x\cdot \nabla \eta }{|x|^2}+|\nabla \eta|^2]ds-r \int_{S_{r}^+}
(-\frac{\gamma}{\pi}\frac{1}{|x|}+\frac{x\cdot \nabla\eta}{|x|})^2ds \\
&=&r\int _{S^+_{r}}[-\frac{1}{2}(\frac{\gamma}{\pi})^2\frac{1}{|x|^2}+\frac{\gamma}{\pi}\frac{x\cdot \nabla \eta}{|x|^2}+\frac{1}{2}|\nabla\eta|^2-(\frac{x\cdot\nabla \eta}{|x|})^2]d\sigma \\
&=&-\frac{1}{2}(\frac{\gamma}{\pi})^2\pi+ \frac{\gamma}{\pi}r\int _{S_{r}^+}\frac{x\cdot\nabla\eta}{|x|^2}+\frac{r}{2}\int _{S_{r}^+}|\nabla\eta|^2-r\int _{S_{r}^+}(\frac{x\cdot \nabla\eta}{|x|})^2\\
&=& -\frac{\gamma^2}{2\pi}+o_r(1),
\end{eqnarray*} where $o_r(1)\rightarrow 0$ as $r\rightarrow 0$.
We also have
$$
(1+\alpha)\int_{D^+_r}2V^2(x)|x|^{2\alpha}e^{2u}-V(x)|x|^{\alpha}e^{u}|\Psi|^2 dx =o_r(1),
$$
and
$$
 r\int_{ S^+_r}V^2(x)|x|^{2\alpha}e^{2u}d\sigma =o_r(1),
$$
and
$$
\int_{D^+_r}(|x|^{2\alpha}e^{2u}x\cdot\nabla (V^2(x))-|x|^{\alpha}e^{u}|\Psi|^2x\cdot \nabla
V(x))dx=o_r(1),
$$
and
$$
(\alpha+1)\int_{S^+_r}cV(x)|x|^{\alpha}e^{u}d\sigma -\int_{L_r}c\frac {\partial V(s,0)}{\partial s}|s|^{\alpha}se^uds+cV(s,0)|s|^{\alpha}se^u|^{s=r}_{s=-r}=o_r(1),
$$
and
$$
\int_{S^+_r}\la\frac{\partial \Psi}{\partial \nu},(x+\bar{x})\cdot \nabla \Psi\ra d\sigma+\int_{S^+_r}\la (x+\bar{x})\cdot \nabla \Psi,\frac{\partial \Psi}{\partial \nu}\ra d\sigma=o_r(1).
$$
Putting all together and letting $r\rightarrow 0$, we
get
$$
C(u,\Psi)=\lim_{r\rightarrow 0}C(u,\Psi, r)=\frac {\gamma^2}{2\pi}.
$$
Since $C(u,\Psi)=0$ for $(u,\Psi)$, therefore we get $\gamma =0$. This implies that the local singularity of $(u,\Psi)$ is removable.
\qed

\

\section{Bubble Energy}

After a suitable rescaling at a boundary blow-up point, we will obtain
a bubble, i.e. an entire solution on the upper half-plane $\R_+^2$
with finite energy. In this section, we will investigate such entire
solutions. We will first show  the asymptotic behavior of an entire
solution and compute the bubble energy, and then  show that an entire solution can be conformally extended to a spherical cap, i.e., the singularity at infinity is removable.

The considered equations are
\begin{equation}\label{se}
\left\{
\begin{array}{rcll}
-\Delta u &=& \ds\vs 2|x|^{2\alpha}e^{2u}-|x|^{\alpha}e^u\left\langle \psi ,\psi
\right\rangle,& \qquad \text{ in } \R_+^2,
\\
\slashiii{D}\psi &=&\ds\vs  -|x|^\alpha e^u\psi, &\qquad \text{ in } \R_+^2,\\
\ds\vs\frac{\partial u}{\partial n} &=& c|x|^\alpha e^u, &\qquad \text { on }
\partial \R_+^2,\\
B\psi &=& 0, &\qquad \text { on }
\partial \R_+^2.
\end{array}
\right.
\end{equation}

\noindent The energy condition is
\begin{equation}\label{sec}
I(u,\psi)=\int_{\R_+^2}(|x|^{2\alpha}e^{2u}+|\psi|^4)dx+\int_{\partial \R_+^2}|x|^{\alpha}e^u
ds<\infty.
\end{equation}

\ \

First, let us notice that if $(u,\psi)$ is a weak solution of (\ref{se}) and (\ref{sec}) with
$u\in H^{1,2}_{loc}(\R_+^2)$ and $\psi\in W^{1,\frac
43}_{loc}(\Gamma(\Sigma\R_+^2))$, by using similar arguments as in the proof of Proposition \ref{prop-b}, we have $u^{+}\in L^{\infty}(\overline{\R}_+^2)$.
Consequently, it follows that $u \in C^{2}_{loc}(\R^2_+)\cap
C^{1}_{loc}(\overline{\R}^2_+)$ and $\psi\in C^{2}_{loc}(\Gamma(\Sigma \R^2_+))\cap
C^{1}_{loc}(\Gamma(\Sigma\overline{\R}^2_+))$.

\
\

We call $(u,\psi)$ a \emph{regular} solution of \eqref{se} and \eqref{sec}, if $u \in C^{2}_{loc}(\R^2_+)\cap
C^{1}_{loc}(\overline{\R}^2_+)$ and $\psi\in C^{2}_{loc}(\Gamma(\Sigma \R^2_+))\cap
C^{1}_{loc}(\Gamma(\Sigma\overline{\R}^2_+))$.

\
\

Next, we denote by $(v,\phi)$ the Kelvin transformation of $(u,\psi)$, i.e.
\begin{eqnarray*}
& & v(x)=u(\frac {x}{|x|^2})-2(1+\alpha)\ln |x|,\\
& & \phi (x)=|x|^{-1} \psi (\frac{x}{|x|^2}).
\end{eqnarray*}
Then $(v,\phi)$ satisfies
\begin{equation}\label{sek}
\left\{
\begin{array}{rcll}
-\Delta v &=& 2|x|^{2\alpha}e^{2v}-|x|^{\alpha}e^v\left\langle \phi ,\phi
\right\rangle,&\qquad \text { in } \R_+^2,
\\
\slashiii{D}\phi &=&\ds  -|x|^{\alpha}e^v\phi, &\qquad \text { in } \R_+^2,\\
\frac{\partial v}{\partial n} &=& c|x|^{\alpha}e^v, &\qquad \text { on }
\partial \R_+^2\backslash\{0\},\\
B\phi &=& 0, &\qquad \text { on }
\partial \R_+^2\backslash\{0\}.
\end{array}
\right.
\end{equation}
And, by change of variable, we can choose $r_0$ small enough such that $(v,\phi)$ satisfies
\begin{equation}\label{sek2}
 \int_{|x|\leq r_0}|x|^{2\alpha}e^{2v}dx\leq \varepsilon_1<2\pi,\quad
\int_{|x|\leq r_0}|\phi|^4dx\leq C,\quad
 |c|\int_{|s|\leq r_0}|x|^{\alpha}e^{v}ds\leq \varepsilon_2<\pi.
\end{equation}

Applying  Lemma \ref{asy-phi} to \eqref{sek} and \eqref{sek2},  and by the Kelvin transformation,  we obtain
the asymptotic estimate of the spinor $\psi(x)$
\begin{equation}\label{asy-psi1}
|\psi (x)|\leq C|x|^{-\frac 12-\delta_0}\qquad \text{for}\quad |x|
\quad \text{near}\quad \infty,
\end{equation}
and
\begin{equation}\label{asy-psi2}
|\nabla\psi (x)|\leq C|x|^{-\frac 32-\delta_0}\qquad \text{for}\quad |x|
\quad \text{near}\quad \infty,
\end{equation}
for some positive number $\delta_0$ provided that $e^{2v}=O(\frac
{1}{|x|^{2(1+\alpha)-\varepsilon}})$, where $\varepsilon$ is any small positive number.

\ \

Denote
$$d =\int_{\R_+^2}2|x|^{2\alpha}e^{2u}-|x|^{\alpha}e^u|\psi|^2dx+\int_{\partial
R_+^2}c|x|^{\alpha}e^uds,$$
and
$$\xi_0=\int_{\R_+^2}e^u\psi dx.$$

Next, we will show that $d=2(1+\alpha)\pi$ and $\xi_0$ is a well-defined constant spinor.

\begin{prop}\label{asy}
Let $(u,\psi)$ be a regular solution of (\ref{se}) and (\ref{sec}) and let $c$ be
a nonnegative constant. Then we have
\begin{equation}\label{ayu}
u(x)=-\frac {d}{\pi} \ln{|x|}+C+O(|x|^{-1}) \qquad
\text{for}\quad |x| \quad \text{near}\quad \infty,
\end{equation}

\begin{equation}\label{aypsi}
\psi (x)=-\frac {1}{2\pi}\frac{x}{|x|^2}
(I+ie_1)\cdot \xi_0+o(|x|^{-1})\qquad \text{for}\quad |x| \quad \text{near}\quad
\infty,
\end{equation}
where $\cdot$ is the Clifford multiplication, $C$ is a positive
universal constant, and $I$ is the identity. In particular we
have $d =2(1+\alpha)\pi$ and $\xi_0$ is well defined.
\end{prop}

\begin{proof} We shall apply standard potential analysis to prove this proposition. Similar arguments can be found in \cite{CL2,JWZ1, JWZ2} and the references therein. The essential facts used in this case are the Pohozaev identity and the decay estimate  for the spinor. For readers' convenience, we sketch the proof here.
\\

\noindent {\bf Step 1.} $
\lim_{|x|\rightarrow \infty}\frac{u(x)}{\ln|x|}=-\frac{d}{\pi}$ and $d>\pi(1+\alpha)$.

Let \begin{eqnarray*} w(x) &=&\frac{1}{2\pi}
\int_{\R^{2}_{+}}(\log |x-y|+\log
|\overline{x}-y|-2\log|y|)(2|y|^{2\alpha}e^{2u(y)}-|y|^{\alpha}e^{u(y)}|\psi(y)|^2)dy\\
& &+\frac{1}{2\pi} \int_{\partial \R^{2}_{+}}(\log |x-y|+\log
|\overline{x}-y|-2\log|y|)c|y|^{\alpha}e^{u(y)}dy.
\end{eqnarray*}
where $\bar{x}$ is the reflection point of $x$ about $\partial \R^2_+$. It
is easy to check that $w(x)$ satisfies

\begin{equation*}  \left\{
\begin{array}{rcll}
{\Delta} w &=& 2|x|^{2\alpha}e^{2u}-|x|^{\alpha}e^u|\psi|^2, &\qquad \text{in}
\quad \R^{2}_{+},\\
\frac {\partial w}{\partial n}&=& -c|x|^{\alpha}e^u,
&\qquad \text{on}\quad \partial \R^{2}_{+}.
\end{array}
\right. \end{equation*}
and
$$
\lim_{|x|\rightarrow \infty}\frac{w(x)}{\ln|x|}=\frac d \pi.
$$
Consider $v(x)=u+w$. Then $v(x)$ satisfies
\begin{equation*}\left\{
\begin{array}{rcl}
{\Delta} v &=& 0, \qquad \text{in}
\quad \R^{2}_{+},\\
\frac {\partial v}{\partial n}&=& 0, \qquad \text{on}\quad
\partial \R^{2}_{+}.
\end{array}
\right.
\end{equation*}

We extend $v(x)$ to $\R^2$ by even reflection such that $v(x)$ is
harmonic in $\R^2$. From Lemma 5.1  we know
$v(x)\leq C(1+\ln (|x|+1))$ for some positive constant $C$. Thus
$v(x)$ is a constant. This completes the proof of Step 1. Since $\int_{\R_+^2}|x|^{2\alpha}e^{2u}dx<\infty$, we get that $d \geq \pi(1+\alpha)$. Furthermore, similarly as in the case of the usual Liouville or super-Liouville equation, we can show that $d>\pi(1+\alpha)$.

\
\

\noindent {\bf Step 2.} The proof of (\ref{ayu}) and $d =2\pi(1+\alpha)$.

Notice that we have shown  $d>\pi(1+\alpha)$ in Step 2,  we then can improve the estimates of $e^{2u}$ to
$$
e^{2u}\leq C|x|^{-2(1+\alpha)-\varepsilon} \quad \text{ for } |x| \text { near } \infty.
$$
Therefore the asymptotic estimates \eqref{asy-psi1} and \eqref{asy-psi2} of the spinor $\psi(x)$ hold. By using the standard potential analysis we can obtain that
\begin{equation*}
u(x)=-\frac{d}{\pi}\ln{|x|}+C+O(|x|^{-1}) \qquad \text{for}\quad |x| \quad
\text{near}\quad \infty
\end{equation*}
for some constant $C>0$. Thus we get the proof of (\ref{ayu}).

Furthermore, we can show that $d=2\pi(1+\alpha)$. For sufficiently large $R>0$, the Pohozaev identity for the solution $(u,\psi)$ gives
\begin{eqnarray}\label{ph}
&& R\int_{ S^+_R} |\frac {\partial u}{\partial
\nu}|^2-\frac
12|\nabla u|^2d\sigma \nonumber\\
&=& (1+\alpha)\int_{D^+_R}2|x|^{2\alpha}e^{2u}-|x|^{\alpha}e^u|\Psi|^2dv+
(\alpha+1)\int_{L_R}c|x|^{\alpha}e^{u}ds \nonumber\\
& & -R\int_{S^+_R}|x|^{2\alpha}e^{2u}d\sigma -c|s|^{\alpha}se^u|^{s=R}_{s=-R} \nonumber\\
& & + \frac 14\int_{S^+_R}\la\frac {\partial
\Psi}{\partial \nu},(x+\bar{x})\cdot\Psi\ra d\sigma+\frac
14\int_{ S^+_R}\la (x+\bar{x})\cdot\Psi, \frac
{\partial \Psi}{\partial \nu}\ra d\sigma.
\end{eqnarray}
By the asymptotic estimates \eqref{asy-psi1},  \eqref{asy-psi2} and \eqref{ayu} of $(u,\psi)$ we have
$$
\lim_{R\rightarrow +\infty}R\int_{ S^+_R} |\frac {\partial u}{\partial
\nu}|^2-\frac
12|\nabla u|^2d\sigma=\frac{d^2}{2\pi},
$$
and
$$
\lim_{R\rightarrow +\infty}R\int_{S^+_R}|x|^{2\alpha}e^{2u}d\sigma +c|s|^{\alpha}se^u|^{s=R}_{s=-R}=0,
$$
and
$$
\lim_{R\rightarrow +\infty}\int_{S^+_R}\la\frac {\partial
\Psi}{\partial \nu},(x+\bar{x})\cdot\Psi\ra d\sigma+\int_{ S^+_R}\la (x+\bar{x})\cdot\Psi, \frac
{\partial \Psi}{\partial \nu}\ra d\sigma=0.
$$
Let $R\rightarrow +\infty$ in (\ref{ph}), we get that
$$
\frac{d^2}{2\pi}=(1+\alpha)d.
$$
It follows that $d=2\pi(1+\alpha)$.

\
\

\noindent {\bf Step 3.} The proof of (\ref{aypsi}).

Since  $d=2\pi(1+\alpha)$ by Step 2,  we can improve the estimate for $e^{2u}$ to
\begin{equation}\label{ayu6}
e^{2u}\leq C|x|^{-4(1+\pi)} \qquad \text{for} \quad |x|\quad
\text{near}\quad \infty.
\end{equation}
This implies that the constant spinor $\xi_0$ is well defined. By using the chirality boundary condition of spinor, we
extend $(u,\psi)$ to the lower half plane $\R^2_{+}$(see \eqref{reflec1} and \eqref{reflec2}) to get
$$
\slashiii{D}\psi=-A(x)\psi,\quad \text { in } \R^2.
$$
Here $A(x)$ is defined by
\begin{equation*}
A(x)=\left \{\begin{matrix} |x|^{\alpha}e^{u(x)},\quad x\in \R^{2}_{+},\\
|\bar{x}|^{\alpha}e^{u(\bar{x})}, \quad x\in \R^2_{-}.\end{matrix} \right.
\end{equation*}
Define
$$
\xi_1=\int_{\R^2}A(x)\psi dx.
$$
The constant spinor $\xi_1$ is also well defined. From the
asymptotic estimates (\ref{asy-psi1}) and  (\ref{ayu6}) and a
similar argument in \cite{JZZ3} we obtain
\begin{equation}\label{aypsii}
\psi (x)=-\frac {1}{2\pi}\frac{x}{|x|^2}
\cdot \xi_1+o(|x|^{-1})\qquad \text{for}\quad |x| \quad \text{near}\quad
\infty.
\end{equation}
Since
\begin{eqnarray*}
\xi_1 &=& \int_{\R_+^2}A(x)\psi dx+\int_{\R_-^2}A(x)\psi dx\\
&=& \int_{\R_+^2}|x|^\alpha e^u\psi dx+\int_{\R_-^2}|\bar{x}|^{\alpha}e^{u(\bar{x})}ie_1\cdot\psi(\bar{x}) dx\\
&=& \int_{\R_+^2}|x|^{\alpha}e^u\psi dx+\int_{\R_+^2}|y|^{\alpha}e^{u(y)}ie_1\cdot\psi(y) dy\\
&=& (I+ie_1)\cdot \int_{\R_+^2}|x|^{\alpha}e^u\psi dx\\
&=& (I+ie_1)\cdot \xi_0.
\end{eqnarray*}
Hence we obtain from (\ref{aypsii})
\begin{equation*}
\psi (x)=-\frac {1}{2\pi}\frac{x}{|x|^2}(I+ie_1)
\cdot \xi_0+o(|x|^{-1})\qquad \text{for}\quad |x| \quad \text{near}\quad
\infty.
\end{equation*}
Thus we finish the proof of Step 3 and we complete the proof of the Proposition.
\end{proof}

\
\

Proposition \ref{asy} indicates that the singularity at infinity of regular solutions for (\ref{se}) and (\ref{sec}) can be removed as in many other conformally invariant problems.

\
\

\begin{thm}\label{rgs}
Let $(u,\psi)$ be a regular solution of (\ref{se}) and (\ref{sec}).
Then $(u,\psi)$ extends conformally to a regular solution on a spherical cap $\mathbb{S}^2_{c^\prime}$, where $c^\prime $  is the geodesic curvature of $\partial \mathbb{S}^2_{c^\prime}$.
\end{thm}

\begin{proof}
Let $(v,\phi)$ be the Kelvin transformation of $(u,\psi)$ as before. Then
$(v,\phi)$ satisfies the system (\ref{sek}). To prove the theorem, by conformal invariance, it is
sufficient to show that $(v,\phi)$ is regular on $\overline{\R}_+^2$. Applying Proposition \ref{asy}, we get
\begin{equation}
v(x)=(\frac {d}{\pi}-2(1+\alpha)) \ln{|x|}+O(1) \qquad \text{for}\quad
|x| \quad \text{near}\quad 0.
\end{equation}
Since $\alpha =2\pi(1+\alpha)$, it follows that $v$ is bounded near the singularity $0$. Recall that $\phi$ is also bounded near $0$, we can apply elliptic theory to obtain that $(v,\phi)$ is regular on $\overline{\R}_+^2$.  \end{proof}

\

\section{Energy Identity for Spinors}

The energy identity for spinor part of solutions to the super-Liouville equations on closed Riemann
surfaces was derived in \cite{JWZZ1, JZZ3}. In this section, we shall
prove an analogue for the singular super-Liouville boundary problem, i.e. Theorem
\ref{engy-indt}. For harmonic maps in dimension two and J-holomorphic
curves as well as for solutions of certain nonlinear Dirac type equations, similar results are derived in \cite{DT, PW, Ye, Z2} and the references therein.

To prove Theorem \ref{engy-indt}, we shall derive the local estimate for the spinor part on an upper half annulus. Since we can extend $(u,\Psi)$ to the lower half disk $D_r^-$ by the chirality boundary condition of $\Psi$,  the proof of this local estimate can be established by using the result  of Lemma 3.4 of \cite{JWZZ1}. Here we just state the Lemma and omit the proof.

\begin{lm}\label{main-lamm}
Let $(u,\Psi)$ satisfies (\ref{Eq-B}) and
 $$\int_{D^+_r}|x|^{2\alpha}e^{2u}+|\Psi|^4dx+\int_{L_r}|x|^{\alpha}e^{u}ds<C.$$
For $0<r_1<2r_1<\frac {r_2}2<r_2<r$, consider the annulus $A_{r_1,r_2}=\{x\in \R^2|r_1\leq
|x|\leq r_2\}$ and the upper half annulus $A^+_{r_1,r_2}=A_{r_1,r_2}\cap \R_+^2$. Then we have
\begin{eqnarray}\label{inqu}
&&(\int_{A^+_{2r_1,\frac {r_2}2}}|D\Psi|^{\frac 43})^{\frac
34}+(\int_{A^+_{2r_1,\frac {r_2}2}}|\Psi|^4)^{\frac 14}\\
&\leq & C_0 (\int_{A^+_{r_1,r_2}}|x|^{2\alpha}e^{2u})^{\frac
12}(\int_{A^+_{r_1,r_2}}|\Psi|^4)^{\frac
14}+C(\int_{A^+_{r_1,2r_1}}|\Psi|^4)^{\frac 14}+C(\int_{A^+_{\frac
{r_2}2, r_2}}|\Psi|^4)^{\frac 14}\nonumber
\end{eqnarray} for a positive constant $C_0$ and some universal
positive constant $C$.
\end{lm}

\
\

\noindent{\bf Proof of Theorem \ref{engy-indt}}.  We will follow
closely the argument for the energy identity of harmonic maps, see
\cite{DT}, or for super-Liouville equations, see
\cite{JWZZ1, JZZ1, JZZ3}.  Since the blow-up set $\Sigma_1$ is finite, we can
find small disk $D^+_{\delta_i}(x_i)$, which is centered at each blow-up point
$x_i$, such that $D^+_{\delta_i}(x_i)\cap D^+_{\delta_j}(x_j)=\emptyset$ for
$i\neq j, i,j=1,2,\cdots, P$, and on $(D^+_r\cup \L_r)\backslash
\bigcup_{i=1}^{P}(D^+_{\delta_i}(x_i)\cup L_{\delta_i}(x_i))$, $\Psi_n$  converges strongly to
$\Psi$ in $L^4$. So, we need to prove that there are
$(u^{i,k},\xi^{i,k})$, which are solutions of (\ref{b1}),
$i=1,2,\cdots , I; k=1,2,\cdots, K_i$, such that
\begin{equation}\label{e1}
\lim_{\delta_i\rightarrow 0}\lim_{n\rightarrow
\infty}\int_{D^+_{\delta_i}(x_i)}|\Psi_n|^4dv=\sum_{k=1}^{L_i}\int_{S^2}|\xi^{i,k}|^4dv,
\text { for } i=1,2,\cdots, I;
\end{equation}
or, we need to prove that there are $(u^{j,l},\xi^{j,l})$, which are
solutions of (\ref{b2}), $j=1,2,\cdots , J; l=1,2,\cdots, L_j$, such
that
\begin{equation}\label{e2}
\lim_{\delta_j\rightarrow 0}\lim_{n\rightarrow
\infty}\int_{D^+_{\delta_j}(x_i)}|\Psi_n|^4dv=\sum_{l=1}^{L_j}\int_{S^2_{c'}}|\xi^{j,l}|^4dv,
\text { for } j=1,2,\cdots, J;
\end{equation}

\
\

When $p\in (D^+_r)^o$, from \cite{JZZ3}, we know that (\ref{e1})
holds. So, without loss of generality, we assume that $p\in
L_r$ and there is only one bubble at each blow-up point $p$. Furthermore, we may assume that $p=0$. The case of  $p\neq 0$ can be handled in an analogous way and in fact this case is simpler, as $|x|^{\alpha}$ is a smooth function near $p$.   Then what we
need to prove is that there exists a bubble $(u,\xi)$ as (\ref{b1}),
such that
\begin{equation}\label{2.4}
\lim_{\delta\rightarrow 0}\lim_{n\rightarrow
\infty}\int_{D^+_{\delta}}|\Psi_n|^4dv=\int_{S^2}|\xi|^4dv,
\end{equation}
or there exists a bubble $(u, \xi)$ as (\ref{b2}) such that  such
that
\begin{equation}\label{2.41}
\lim_{\delta\rightarrow 0}\lim_{n\rightarrow
\infty}\int_{D^+_{\delta}}|\Psi_n|^4dv=\int_{S^2_{c'}}|\xi|^4dv.
\end{equation}

\
\

Next we rescale  functions $(u_n,\Psi_n)$ at the blow-up point $p=0$ and then try to get the bubble of $(u_n,\Psi_n)$. To this purpose, we let $x_n\in \overline{D}^+_{\delta}$ such that
$u_n(x_n)=\max_{\overline{D}^+_{\delta}}u_n(x)$. Write $x_n=(s_n,t_n)$. It is clear that $x_n\rightarrow p$ and
$u_n(x_n)\rightarrow +\infty$. Define  $\lambda_n=e^{-\frac{u_n(x_n)}{\alpha+1}}$. We know $\lambda_n$, $|x_n|$ and $t_n$ converge to $0$ as $n\rightarrow 0$, but their rates of converging to $0$ may be different.  Next we will distinguish three cases.

\
\

\noindent{\bf Case I.} $\frac{|x_n|}{\lambda_n}=O(1)$ as $n\rightarrow +\infty$.

In this case,
we define the rescaling functions
\begin{equation*}
\left\{
\begin{array}{rcl}
\widetilde{u}_n(x)&=&u_n(\lambda_nx)+(1+\alpha)\ln {\lambda_n}\\
\widetilde{\Psi}_n(x)&=&\lambda^{\frac
12}_n\Psi_n(\lambda_nx)
\end{array}
\right.
\end{equation*}
for any $x \in \overline {D}^+_{\frac{\delta}{2\lambda_n}
}$. Then $(\widetilde{u}_n(x),\widetilde{\Psi}_n(x))$ satisfies
\begin{equation*}
\left\{
\begin{array}{rcll}
-\triangle \widetilde{u}_n(x)&=& 2V^2(\lambda_n x)|x|^{2\alpha}e^{2\widetilde{u}_n(x)}-V(\lambda_n x)|x|^\alpha
e^{\widetilde{u}_n(x)}|\widetilde{\Psi}_n(x)|^2, & \text{ in } D^+_{\frac{\delta}{2\lambda_n}
},\\
\slashiii
{D}\widetilde{\Psi}_n(x)&=&-V(\lambda_n x)|x|^{\alpha}e^{\widetilde{u}_n(x)}\widetilde{\Psi}_n(x),
& \text{ in } D^+_{\frac{\delta}{2\lambda_n}
},\\
\frac{\partial \widetilde{u}_n(x)}{\partial
n}&=&cV(\lambda_n x)|x|^\alpha e^{\widetilde{u}_n(x)}, & \text{ on } L_{\frac{\delta}{2\lambda_n}
},\\
B\widetilde{\Psi}_n(x)&=& 0, & \text{ on }
L_{\frac{\delta}{2\lambda_n}
},
\end{array}
\right.
\end{equation*}
with the energy condition
$$
\int_{D^+_{\frac{\delta}{2\lambda_n}
}}|x|^{2\alpha}e^{2\widetilde{u}_n(x)}+|\widetilde{\Psi}_n(x)|^4dv+\int_{L_{\frac{\delta}{2\lambda_n}
}}|x|^\alpha e^{\widetilde{u}_n(x)}d\sigma
<C.
$$
We know that
$$ \max_{\bar{D}^+_{\frac{\delta}{2\lambda_n}}}\widetilde{u}_n(x)=
\widetilde{u}_n(\frac{x_n}{\lambda_n})=u_n(x_n)+(\alpha_n+1)\ln \lambda_n=0.$$
Notice that the maximum point of $\widetilde {u}_n(x)$, i.e. $\frac{x_n}{\lambda_n}$, is bounded, namely $|\frac{x_n}{t_n}|\leq C$. So by taking a subsequence, we can assume that
$\frac{x_n}{t_n}\rightarrow x_0\in \bar{\R}^2_+$  with $|x_0|\leq C$. Therefore it follows from Theorem \ref{mainthm} that, by passing to a subsequence,  $(\widetilde{u}_n,\widetilde{\Psi}_n)$ converges in  $C^{2}_{loc}(\R^2_+)\cap C^{1}_{loc}(\bar{\R}^2_+)\times C^{2}_{loc}(\Gamma(\Sigma {\R}^2_+))\cap C^{1}_{loc}(\Gamma(\Sigma {\bar{\R}}^2_+))$ to some $(\widetilde
u,\widetilde \Psi)$ satisfying
\begin{equation}
\left\{
\begin{array}{rcll}
-\Delta \widetilde{u} &=& 2V^2(0)|x|^{2\alpha}e^{2\widetilde{u}}-V(0)|x|^{\alpha}e^{\widetilde u}|
\widetilde \Psi|^2, &\quad \text { in } \R^2_+,\\
\slashiii{D}\widetilde{\Psi} &=&
-V(0)|x|^{\alpha}e^{\widetilde{u}}\widetilde{\Psi},  &\quad \text { in } \R^2_+,\\
\frac{\partial \widetilde{u}}{\partial
n}&=&cV(0)|x|^\alpha e^{\widetilde{u}}, & \quad \text{ on }  \partial \R^2_+,\\
B\widetilde{\Psi}&=& 0, & \quad \text{ on } \partial \R^2_+
\end{array}
\right.   \label{eq-13-1}
\end{equation}
with the energy condition
$\int_{\R^2_+}(|x|^{2\alpha}e^{2\widetilde{u}}+|\widetilde{\Psi}|^4)dx +\int_{\partial \R^2_+}
|x|^\alpha e^{\widetilde{u}}d\sigma< \infty$.  By Proposition \ref{asy}, there holds
\[\int_{\R^2_+}(2V^2(0)|x|^{2\alpha}e^{2\widetilde{u}}-V(0)|x|^{\alpha}e^{\widetilde {u}}|\widetilde{\Psi}|^2)dx+\int_{\partial \R^2_+}
cV(0)|x|^\alpha e^{\widetilde{u}}d\sigma=2\pi(1+\alpha).\]
By the removability of a global singularity (Theorem \ref{rgs}), we get a bubbling solution  on $S^2_{c'}$.

\
\

\noindent{\bf Case II.} $\frac{|x_n|}{\lambda_n}\rightarrow +\infty$ as $n\rightarrow +\infty$.

In this case, we must have
\begin{equation}\label{ttt}
\overline{u}_n(y_n):=u_n(x_n)+(\alpha+1)\ln|x_n|=(\alpha+1)\ln|x_n|-(\alpha+1)\ln\lambda_n\rightarrow +\infty.
\end{equation}
Therefore we can rescale twice to get the bubble. First, we defince the rescaling functions
\begin{equation*}
\left\{
\begin{array}{rcl}
\overline{u}_n(x)&=&u_n(|x_n|x)+(\alpha+1)\ln|x_n|\\
\overline{\Psi}_n(x)&=&|x_n|^{\frac{1}{2}}\Psi_n(|x_n|x)
\end{array}
\right.
\end{equation*}
for any $x\in \overline{D}^+_{\frac{\delta}{2|x_n|}}$. Then $(\overline{u}_n(x),\overline{\Psi}_n(x))$ satisfies that
\begin{equation*}
\left\{
\begin{array}{rcll}
-\triangle \overline{u}_n(x)&=& 2V^2(|x_n|x)|x|^{2\alpha}e^{2\overline{u}_n(x)}-V(|x_n|x)|x|^\alpha e^{\overline{u}_n(x)}|\overline{\Psi}_n(x)|^2, & \text{ in } D^+_{\frac{\delta}{2|x_n|}},\\
\slashiii{D}\overline{\Psi}_n(x)&=&-V(|x_n|x)|x|^\alpha e^{\overline{u}_n(x)}\overline{\Psi}_n(x),
& \text{ in } D^+_{\frac{\delta}{2|x_n|}},\\
\frac{\partial \overline{u}_n(x)}{\partial
\nu}&=&cV(|x_n|x)|x|^\alpha e^{\overline{u}_n(x)}, & \text{ on } L_{\frac{\delta}{2|x_n|}},\\
B\overline{\Psi}_n(x)&=& 0, & \text{ on } L_{\frac{\delta}{2|x_n|}}.
\end{array}
\right.
\end{equation*}
Set that $y_n=\frac{x_n}{|x_n|}$. We assume that $y_0=\lim_{n\rightarrow \infty}\frac{x_n}{|x_n|}$. By (\ref{ttt}), we know $y_0$ is a blow-up point of $(\overline{u}_n, \overline{\Psi}_n)$. We can set $\delta_n=e^{-\overline{u}_n(y_n)}$, and $\rho_n=\frac{e^{-u_n(x_n)}}{|x_n|^\alpha}=\lambda_n(\frac{\lambda_n}{|x_n|})^\alpha $. It is clear that $\delta_n\rightarrow 0$, $\rho_n\rightarrow 0$ and $\frac{|x_n|}{\rho_n}\rightarrow +\infty$ as $n\rightarrow \infty$.
We define the rescaling functions
\begin{equation*}
\left\{
\begin{array}{rcl}
\widetilde{u}_n(x)&=&\overline{u}_n(\delta_nx+y_n)+\ln\delta_n=u_n(x_n+\rho_nx)-u_n(x_n)\\
\widetilde{\Psi}_n(x)&=&=\delta_n^{\frac 12}\overline{\Psi}_n(\delta_nx+y_n)=\rho_n^{\frac
12}\Psi_n(x_n+\rho_n x)
\end{array}
\right.
\end{equation*}
for any $x$ such that $y_n+\delta_n x \in \overline{D}^+_{R}(y_n)$ with any $R>1$.  By a direct computation, we have
 $$\Omega_n=\{x\in \R^2 | y_n+\delta_n x \in \overline{D}^+_{R}(y_n)\}=\{x\in \R^2 | x_n+\rho_n x \in \overline{D}^+_{R|x_n|}(x_n)\}.$$
We set $L_n=\partial \Omega_n \cap \{x\in \R^2 | t= -\frac {t_n}{\rho_n} \}$. Then $(\widetilde{u}_n(x),\widetilde{\Psi}_n(x))$ satisfies
\begin{equation*}
\left\{
\begin{array}{rcll}
-\triangle \widetilde{u}_n(x)&=& 2V^2(x_n+\rho_nx)|\frac {x_n}{|x_n|}+\frac{\rho_n}{|x_n|}x|^{2\alpha}e^{2\widetilde{u}_n(x)}\\
& & -V(x_n+\rho_nx)|\frac {x_n}{|x_n|}+\frac{\rho_n}{|x_n|}x|^\alpha
e^{\widetilde{u}_n(x)}|\widetilde{\Psi}_n(x)|^2, & \text{ in } \Omega_n
,\\
\slashiii
{D}\widetilde{\Psi}_n(x)&=&-V(x_n+\rho_nx)|\frac {x_n}{|x_n|}+\frac{\rho_n}{|x_n|}x|^\alpha e^{\widetilde{u}_n(x)}\widetilde{\Psi}_n(x),
& \text{ in } \Omega_n
,\\
\frac{\partial \widetilde{u}_n(x)}{\partial
n}&=&cV(x_n+\rho_nx)|\frac {x_n}{|x_n|}+\frac{\rho_n}{|x_n|}x|^\alpha e^{\widetilde{u}_n(x)}, & \text{ on } L_n
,\\
B\widetilde{\Psi}_n(x)&=& 0, & \text{ on } L_n,
\end{array}
\right.
\end{equation*}
with the energy condition
$$
\int_{\Omega_n}|\frac {x_n}{|x_n|}+\frac{\rho_n}{|x_n|}x|^{2\alpha}e^{2\widetilde{u}_n(x)}+|\widetilde{\Psi}_n(x)|^4dv+\int_{L_n}|\frac {x_n}{|x_n|}+\frac{\rho_n}{|x_n|}x|^\alpha e^{\widetilde{u}_n(x)}d\sigma
<C.
$$
It is clear that
$$ \widetilde{u}_n(x)\leq\max_{\Omega_n}\widetilde{u}_n(x)=
\widetilde{u}_n(0)=0 . $$

Now we proceed by distinguishing two subcases.

\

\noindent{\bf Case II.1}   $\frac {t_n}{\rho_n}\rightarrow +\infty$ as $n\rightarrow \infty$.

\

Notice that $|\frac {x_n}{|x_n|}+\frac{\rho_n}{|x_n|}x|\rightarrow 1 $ as $n\rightarrow \infty$ in $C^0_{loc}(\R^2)$. It follows from Theorem \ref{mainthm} that, by passing to a subsequence,  $(\widetilde{u}_n,\widetilde{\Psi}_n)$ converges in  $C^{2}_{loc}(\R^2)\times C^{2}_{loc}(\Gamma(\Sigma {\R}^2))$ to some $(\widetilde
u,\widetilde \Psi)$ satisfying
\begin{equation}
\left\{
\begin{array}{rcll}
-\Delta \widetilde{u} &=& 2V^2(0)e^{2\widetilde{u}}-V(0)e^{\widetilde u}|
\widetilde \Psi|^2, &\quad \text { in } \R^2,\\
\slashiii{D}\widetilde{\Psi} &=&
-V(0)e^{\widetilde{u}}\widetilde{\Psi},  &\quad \text { in } \R^2,\\
\end{array}
\right.   \label{eq-13-2}
\end{equation}
with the energy condition
$\int_{\R^2}e^{2\widetilde{u}}+|\widetilde{\Psi}|^4dx <\infty$.  By Proposition 6.4 in \cite{JWZ1}, there holds
\[\int_{\R^2}(2V^2(0)e^{2\widetilde{u}}-V(0)e^{\widetilde {u}}|\widetilde{\Psi}|^2)dx=4\pi.\]
By the removability of a global singularity (Theorem 6.5 in \cite{JWZ1}), we get a bubbling solution  on $S^2$.

\

\noindent{\bf Case II.2}   $\frac {t_n}{\rho_n}\rightarrow \Lambda$ as $n\rightarrow \infty$.

\

Simiar in the Case II.1, we have from Theorem \ref{mainthm} that, by passing to a subsequence,  $(\widetilde{u}_n,\widetilde{\Psi}_n)$ converges in  $C^{2}_{loc}(\R^2_{-\Lambda})\cap C^{1}_{loc}(\bar{\R}^2_{-\Lambda})\times C^{2}_{loc}(\Gamma(\Sigma {\R}^2_{-\Lambda}))\cap C^{1}_{loc}(\Sigma\bar{\R}^2_{-\Lambda})$ to some $(\widetilde
u,\widetilde \Psi)$ satisfying
\begin{equation}
\left\{
\begin{array}{rcll}
-\Delta \widetilde{u} &=& 2V^2(0)e^{2\widetilde{u}}-V(0)e^{\widetilde u}|
\widetilde \Psi|^2, &\quad \text { in } \R^2_{-\Lambda},\\
\slashiii{D}\widetilde{\Psi} &=&
-V(0)e^{\widetilde{u}}\widetilde{\Psi},  &\quad \text { in } \R^2_{-\Lambda},\\
\frac{\partial \widetilde{u}}{\partial
n}&=&cV(0) e^{\widetilde{u}}, &\quad  \text{ on } \partial \R^2_{-\Lambda}
,\\
B\widetilde{\Psi} &=& 0, & \quad \text{ on } \partial \R^2_{-\Lambda},
\end{array}
\right.   \label{eq-13-3}
\end{equation}
with the energy condition
$\int_{\R^2_{-\Lambda}}e^{2\widetilde{u}}+|\widetilde{\Psi}|^4dx+\int_{\partial \R^2_{-\Lambda}}e^{\widetilde{u}}d\sigma <\infty$.  By Proposition 6.4 in \cite{JWZ1}, there holds
\[\int_{\R^2_{-\Lambda}}(2V^2(0)e^{2\widetilde{u}}-V(0)e^{\widetilde {u}}|\widetilde{\Psi}|^2)dx+\int_{\partial \R^2_{-\Lambda}}cV(0)e^{\widetilde{u}}d\sigma=2\pi.\]
By the removability of a global singularity (Theorem 6.5 in \cite{JZZ1}), we get a bubbling solution  on $S^2_{c'}$.

\
\

It is well know, in order to prove (\ref{2.4}) or (\ref{2.41}),  we need to prove that there is no any energy of $\Psi_n$ in the neck domain, i.e.
\begin{equation}\label{nev}
\lim_{\delta\rightarrow 0}\lim_{R\rightarrow +\infty}\lim_{n\rightarrow
\infty}\int_{A^+_{\delta,R,n}}|\Psi_n|^4dv=0,
\end{equation}
where $A^+_{\delta,R,n}$ is the neck domain which is defined latter. To this purpose, we shall proceed separately for Case I, Case II.1 and Case II.2.

\
\

For {\bf Case I}, we define the neck domain is
$$ A^+_{\delta, R, n}=\{x\in \R^2_+|\lambda_n R\leq |x|\leq
\delta\}.$$

We have two claims.

\

\noindent{\bf Claim 1} For any
$\varepsilon>0$, there is an $N>1$ such that for any $n\geq N$, we
have
$$
\int_{D^+_r\setminus
D^+_{e^{-1}r}}(|x|^{2\alpha}e^{2u_n}+|\Psi_n|^4)+\int_{\partial
(D^+_r\setminus D^+_{e^{-1}r})\cap
\partial \R^2_+}|x|^{\alpha}e^{u_n}<\varepsilon; \quad
\forall r\in [e\lambda_n R, \delta].
$$

\

To prove this claim, we note  two facts. The first fact is: for any
$T>0$, there exists some $N(T)$ such that for any $n\geq N(T)$, we
have
\begin{equation}\label{2.6}
\int_{D^+_\delta \setminus D^+_{\delta
e^{-T}}}(|x|^{2\alpha}e^{2u_n}+|\Psi_n|^4)+\int_{\partial( D^+_\delta
\setminus D^+_{\delta e^{-T}})\cap
\partial \R^2_+}|x|^{\alpha}e^{u_n}<\varepsilon.
\end{equation}
Actually, since $(u_n,\Psi_n)$ has no blow-up point in
$\overline{D}^+_{\delta}\backslash \{p\}$, then $|\Psi_n|$ is
uniformly bounded in $\overline{D^+_\delta\backslash D^+_{\delta
e^{-T}}}$ , and $u_n$ will either be uniformly bounded in
$\overline{D^+_\delta\backslash D^+_{\delta e^{-T}}}$ or uniformly
tend to $-\infty$ in $\overline{D^+_\delta\backslash D^+_{\delta
e^{-T}}}$. So if $u_n$ uniformly tends to $-\infty$ in
$\overline{D_\delta\backslash D_{\delta e^{-T}}}$, it is clear that,
for any given $T>0$, we have an $N(T)$ big enough such that when
$n\geq N(T)$
$$
\int_{D^+_\delta\backslash D^+_{\delta
e^{-T}}}(|x|^{2\alpha}e^{2u_n}+\int_{\partial (D^+_\delta\backslash
D^+_{\delta e^{-T}})\cap \partial
\R^2_+}|x|^{\alpha}e^{u_n}<\frac{\varepsilon}{2}.
$$
Moreover, since $\Psi_n$ converges to $\Psi$ in
$L_{loc}^4((D^+_r\cap L_r)\setminus \Sigma_1)$ and hence
$$
\int_{D^+_\delta\backslash D^+_{\delta
e^{-T}}}|\Psi_n|^{4} \rightarrow \int_{D^+_\delta\backslash
D^+_{\delta e^{-T}}}|\Psi|^{4}.
$$
For any small $\varepsilon>0$, we may choose $\delta>0$ small enough
such that  $\int_{D^+_\delta}|\Psi|^{4}<\frac \varepsilon 4$, then
for any given $T>0$, we have an $N(T)$ big enough such that when
$n\geq N(T)$
$$
\int_{D^+_\delta\backslash D^+_{\delta
e^{-T}}}|\Psi_n|^{4}<\frac{\varepsilon}{2}.
$$
Consequently, we get (\ref{2.6}).

\

If $(u_n, \Psi_n)$ is uniformly bounded in
$\overline{D^+_\delta\backslash D^+_{\delta e^{-T}}}$, then we know
$(u_n, \Psi_n)$ converges to a weak solution $(u,\Psi)$ strongly
on compact sets of $D^+_{\delta}\setminus\{p\}$. Therefore, we can
also  choose $\delta>0$ small enough such that, for any given $T>0$,
there exists an $N(T)$ big enough, when $n\geq N(T)$, (\ref{2.6})
holds.

\
\

The second fact is: For any small $\varepsilon>0$, and $T>0$, we may
choose an $N(T)$ such that when $n\geq N(T)$

\begin{eqnarray*}
& & \int_{D^+_{\lambda_nRe^T}\setminus
D^+_{\lambda_nR}}(|x|^{2\alpha}e^{2u_n}+|\Psi_n|^4)+\int_{\partial
(D^+_{\lambda_nRe^T}\setminus D^+_{\lambda_nR})\cap
\partial \R^2_+}|x|^{\alpha}e^{u_n}\\
&=& \int_{D^+_{Re^T}\setminus
D^+_{R}}(|x|^{2\alpha}e^{2\widetilde{u}_n}+|\widetilde{\Psi}_n|^4)+\int_{\partial
(D^+_{Re^T}\setminus D^+_{R})\cap
\partial \R^2_+}|x|^{\alpha}e^{\widetilde{u}_n}\\
& < & \varepsilon,
\end{eqnarray*}
if $R$ is big enough.

\
\

Now we can prove the claim. We argue  by contradiction by using the
above two facts.  If there exists $\varepsilon_0>0$ and a sequence
${r_n}$, $r_n\in [e\lambda_nR,\delta]$, such that
$$
\int_{D^+_{r_n}\setminus
D^+_{e^{-1}r_n}}(|x|^{2\alpha}e^{2u_n}+|\Psi_n|^{4})+\int_{\partial
(D^+_{r_n}\setminus D^+_{e^{-1}r_n})\cap
\partial \R^2_+}|x|^{\alpha}e^{u_n}\geq \varepsilon_0.
$$
Then, by the above two facts, we know that $\frac
{\delta}{r_n}\rightarrow +\infty$ and
$\frac{\lambda_nR}{r_n}\rightarrow 0$, in particular, $r_n\rightarrow
0$ as $n\rightarrow +\infty$. Rescaling again, we set
\begin{equation*}
\left\{
\begin{array}{rcl}
v_n(x)&=& u_n(r_nx)+(1+\alpha)\ln r_n, \\
\varphi_n(x)&=& r_n^{\frac 12}\Psi(r_nx)\\
\end{array}
\right.
\end{equation*}
for any  $x\in D^+_{\frac {\delta}{r_n}}\setminus
D^+_{\frac {\lambda_n R}{r_n}}$.

It is clear that
\begin{equation}\label{3.1}
\int_{D^+_1\setminus D^+_{e^{-1}}}(|x|^{2\alpha}e^{2v_n}+|\varphi_n|^4)+\int_{\partial
(D^+_1\setminus D^+_{e^{-1}})\cap
\partial \R^2_+}|x|^{\alpha}e^{v_n}\geq
\varepsilon_0.
\end{equation}

And  $(v_n,\varphi_n)$ satisfies for any $R>0$

\begin{equation*}
\left\{
\begin{array}{rcll}
-\triangle v_n(x) &=& 2V^2(r_nx)|x|^{2\alpha}e^{2v_n(x)}-V(r_nx)|x|^{\alpha}e^{v_n(x)}
|\varphi_n(x)|^2, &\quad \text{ in }
(D^+_{\frac{\delta}{r_n}}\setminus D^+_{\frac{\lambda_nR}{r_n}}), \\
\slashiii{D}\varphi_n(x) &=& -V(r_nx)|x|^{\alpha}e^{v_n(x)}\varphi_n(x), &
\quad \text{ in } (D^+_{\frac{\delta}{r_n}}\setminus D^+_{\frac{\lambda_nR}{r_n}}),\\
\frac{\partial v_n(x)}{\partial n}&=& cV(r_nx)|x|^{\alpha}e^{v_n(x)}, &\quad
\text{ on } \partial (D^+_{\frac{\delta}{r_n}}\setminus D^+_{\frac{\lambda_nR}{r_n}})\cap
\partial \R^2_+,\\
B\varphi_n(x)&=&0, &\quad \text{ on } \partial
(D^+_{\frac{\delta}{r_n}}\setminus D^+_{\frac{\lambda_nR}{r_n}})\cap
\partial \R^2_+.
\end{array}
\right.
\end{equation*}
According to Theorem \ref{mainthm}, there are three possible cases:

 (1). There exists some $q\in Q_n=(D^+_{\frac{\delta}{r_n}}\setminus
 D^+_{\frac{\lambda_nR}{r_n}})$ and energy
concentration occurs near the point $q$, namely along some subsequence we
have $$\lim_{n\rightarrow \infty} \int_{D_r(q)\cap Q_n
}(|x|^{2\alpha}e^{2v_n}+|\varphi_n|^{4})+\int_{ D_r(q)\cap
\partial Q_n\cap \{t=0\}}|x|^{\alpha}e^{v_n}\geq \varepsilon_0>0$$ for any small $r>0$. In such
a case, we still obtain the second ``bubble" by the rescaling
argument. Thus we get a contradiction.

 (2). For any $R>0$, there is no blow-up point in $D^+_R\setminus D^+_{\frac 1R}$ and $v_n\rightarrow -\infty$ uniformly in
$\overline{D^+_R\setminus D^+_{\frac 1R}}$. Then, it is clear that $\varphi_n$ converges
to a spinor $\varphi$ in $ L_{loc}^4(\overline{\R_+^2}\setminus\{0\})$
which satisfies
\begin{equation*}
\left\{
\begin{array}{rcll}
\slashiii{D}\varphi &=& 0, & \text { in } \R^2_+,\\
B \varphi &=& 0, &\text { on } \partial \R^2_+\setminus\{0\}.
\end{array}
\right.
\end{equation*}
 We translate $\varphi $ to be a harmonic spinor on $\R_+^2
 \setminus\{0\}$ satisfying the corresponding chirality boundary
 condition and then extend it as in \eqref{reflec2} to a harmonic
 spinor $\overline{\varphi}$ on $\R^2\setminus\{0\}$ with bounded
 energy, i.e., $||\overline{\varphi}||_{L^4(\R^2)} < \infty$. As discussed in \cite{JWZZ1}, $\overline{\varphi}$ conformally extends to a harmonic spinor on $S^2$. By the well known fact that
there is no nontrivial harmonic spinor on $S^2$, we have that
$\overline{\varphi} \equiv 0$ and hence $\varphi_n$ converges to
$0$ in $L_{loc}^4(\R^2_+\setminus\{0\})$. This will contradict (\ref{3.1})

(3). For any $R>0$, there is no blow-up point in $(D^+_R\setminus D^+_{\frac 1R})$  and $(v_n, \varphi_n)$ is uniformly bounded in
$(D^+_R\setminus D^+_{\frac 1R})$. In such a case $(v_n,
\varphi_n)$ will converge to $(v,\varphi)$ strongly on
$(D^+_R\setminus D^+_{\frac 1R})$ and $(v,\varphi)$ satisfying
\begin{equation*}
\left\{
\begin{array}{rcll}
-\triangle v &=& 2V^2(0)|x|^{2\alpha}e^{2v}-V(0)|x|^{\alpha}e^{v}|\varphi|^2, & \quad \text{ in }\R^2_+,\\
\slashiii{D}\varphi &=& -V(0)|x|^{\alpha}e^{v}\varphi, & \quad \text{ in
}\R^2_+,\\
\frac{\partial v}{\partial n}&=& cV(0)|x|^{\alpha}e^{v}, &\quad \text{ on }
 \partial \R^2_+\setminus\{0\},\\
B\varphi &=& 0, &\quad \text{ on }  \partial \R^2_+\setminus\{0\}
\end{array}
\right.
\end{equation*}
with finite energy. It is clear that $(v,\varphi)$ is
regular.

 Next we need to remove the singularities of $(v,\varphi)$ and then obtain the second bubble  of the system. Concequently we get a contradiction. To this purpose, let us use the Pohozaev identity of $(u_n, \Psi_n)$ in $D^+_\delta$, it follows for any $\rho$ with $r_n\rho<\delta$
\begin{eqnarray*}
&& r_n\rho\int_{ S^+_{r_n\rho}} |\frac {\partial u_n}{\partial
\nu}|^2-\frac
12|\nabla u_n|^2d\sigma \nonumber\\
&=& (1+\alpha)\int_{D^+_{r_n\rho}}2V^2(x)|x|^{2\alpha}e^{2u_n}-V(x)|x|^{\alpha}e^{u_n}|\Psi_n|^2dv+
(\alpha+1)\int_{L_{r_n\rho}}cV(x)|x|^{\alpha}e^{u_n}ds \nonumber\\
& & -r_n\rho\int_{S^+_{r_n\rho}}V^2(x)|x|^{2\alpha}e^{2u_n}d\sigma +\int_{ L_{r_n \rho}}c\frac {\partial V(s,0)}{\partial s}|s|^{\alpha}se^{u_n}ds-cV((s,0))|s|^{\alpha}se^{u_n}|^{s=r_n\rho}_{s=-r_n\rho} \nonumber\\
& & +\int_{D^+_{r_n\rho}}x\cdot \nabla (V^2(x))|x|^{2\alpha}e^{2u_n}dv -\int_{D^+_{r_n\rho}}x\cdot \nabla
V(x)|x|^{\alpha}e^{u_n}|\Psi_n|^2dv \nonumber\\
& & + \frac 14\int_{S^+_{r_n\rho}}\la\frac {\partial
\Psi_n}{\partial \nu},(x+\bar{x})\cdot\Psi_n\ra d\sigma+\frac
14\int_{ S^+_{r_n\rho}}\la (x+\bar{x})\cdot\Psi_n, \frac
{\partial \Psi_n}{\partial \nu}\ra d\sigma.
\end{eqnarray*}
Hence for rescaling functions $(v_n, \varphi_n)$ we have
\begin{eqnarray*}
&& \rho\int_{ S^+_{\rho}} |\frac {\partial v_n}{\partial
\nu}|^2-\frac
12|\nabla v_n|^2d\sigma \nonumber\\
&=& (1+\alpha)\int_{D^+_{\rho}}2V^2(r_nx)|x|^{2\alpha}e^{2v_n}-V(r_nx)|x|^{\alpha}e^{v_n}|\varphi_n|^2dv+
(\alpha+1)\int_{L_{\rho}}cV(r_nx)|x|^{\alpha}e^{v_n}ds \nonumber\\
& & -\rho\int_{S^+_{\rho}}V^2(r_nx)|x|^{2\alpha}e^{2v_n}d\sigma +\int_{ L_{\rho}}c\frac {\partial V((r_ns,0))}{\partial s}|s|^{\alpha}se^{v_n}ds-cV((r_ns,0))|s|^{\alpha}se^{v_n}|^{s=\rho}_{s=-\rho} \nonumber\\
& & +\int_{D^+_{\rho}}x\cdot (\nabla V^2)(r_nx)|x|^{2\alpha}e^{2v_n}dv -\int_{D^+_{\rho}}x\cdot (\nabla
V)(r_nx)|x|^{\alpha}e^{v_n}|\varphi_n|^2dv \nonumber\\
& & + \frac 14\int_{S^+_{\rho}}\la\frac {\partial
\varphi_n}{\partial \nu},(x+\bar{x})\cdot\varphi_n\ra d\sigma+\frac
14\int_{ S^+_{\rho}}\la (x+\bar{x})\cdot\varphi_n, \frac
{\partial \varphi_n}{\partial \nu}\ra d\sigma.
\end{eqnarray*}
This implies that the associated Pohozaev constant of $(v_n,\varphi_n)$ satisties
\begin{eqnarray*}
& & C(v_n, \varphi_n)  =  C(v_n, \varphi_n,\rho)\\
& = & \rho\int_{ S^+_{\rho}} |\frac {\partial v_n}{\partial
\nu}|^2-\frac
12|\nabla v_n|^2d\sigma \nonumber\\
& & -(1+\alpha)\int_{D^+_{\rho}}2V^2(r_nx)|x|^{2\alpha}e^{2v_n}-V(r_nx)|x|^{\alpha}e^{v_n}|\varphi_n|^2dv-
(\alpha+1)\int_{L_{\rho}}cV(r_nx)|x|^{\alpha}e^{v_n}ds \nonumber\\
& & +\rho\int_{S^+_{\rho}}V^2(r_nx)|x|^{2\alpha}e^{2v_n}d\sigma -\int_{ L_{\rho}}c\frac {\partial V((r_ns,0))}{\partial s}|s|^{\alpha}se^{v_n}ds+cV((r_ns,0))|s|^{\alpha}se^{v_n}|^{s=\rho}_{s=-\rho} \nonumber\\
& & -\int_{D^+_{\rho}}x\cdot (\nabla V^2)(r_nx)|x|^{2\alpha}e^{2v_n}dv +\int_{D^+_{\rho}}x\cdot (\nabla
V)(r_nx)|x|^{\alpha}e^{v_n}|\varphi_n|^2dv \nonumber\\
& & - \frac 14\int_{S^+_{\rho}}\la\frac {\partial
\varphi_n}{\partial \nu},(x+\bar{x})\cdot\varphi_n\ra d\sigma-\frac
14\int_{ S^+_{\rho}}\la (x+\bar{x})\cdot\varphi_n, \frac
{\partial \varphi_n}{\partial \nu}\ra d\sigma\\
& = & 0.
\end{eqnarray*}
Since, for any $\rho>0$, $\int_{D^+_\rho}|x|^{2\alpha}e^{2v_n}+|\varphi_n|^4dv+\int_{L_\rho}|x|^{\alpha}e^{v_n}ds<C$, it is easy to check that
$$
\lim_{\rho\rightarrow 0}\lim_{n\rightarrow \infty}\int_{D^+_{\rho}}x\cdot (\nabla V^2)(r_nx)|x|^{2\alpha}e^{2v_n}dv +\int_{D^+_{\rho}}x\cdot (\nabla
V)(r_nx)|x|^{\alpha}e^{v_n}|\varphi_n|^2dv=0,
$$
and
$$
\lim_{\rho\rightarrow 0}\lim_{n\rightarrow \infty}\int_{ L_{\rho}}c\frac {\partial V((r_ns,0))}{\partial s}|s|^{\alpha}se^{v_n}ds=0.
$$
This implies that
\begin{eqnarray*}
0 & = & \lim_{\rho\rightarrow 0}\lim_{n\rightarrow \infty} C(v_n, \varphi_n,\rho)\\
& = & \lim_{\rho\rightarrow 0}C(v,\varphi,\rho)-(1+\alpha)\lim_{r\rightarrow 0}\lim_{n\rightarrow \infty}\int_{D^+_{r}}2V^2(r_nx)|x|^{2\alpha}e^{2v_n}-V(r_nx)|x|^{\alpha}e^{v_n}|\varphi_n|^2dv\\
& & -(1+\alpha)\lim_{r\rightarrow 0}\lim_{n\rightarrow \infty}\int_{L_{r}}cV(r_nx)|x|^{\alpha}e^{v_n}ds\\
& = & C(v,\varphi)-(1+\alpha)\beta.
\end{eqnarray*}
Here
$$
\beta=\lim_{r\rightarrow 0}\lim_{n\rightarrow \infty}[\int_{D^+_{r}}2V^2(r_nx)|x|^{2\alpha}e^{2v_n}-V(r_nx)|x|^{\alpha}e^{v_n}|\varphi_n|^2dv+\int_{L_{r}}cV(r_nx)|x|^{\alpha}e^{v_n}ds],
$$
and $C(v,\varphi)=C(v,\varphi,\rho)$ is the Pohozaev constant of $(v,\varphi)$, i.e.
\begin{eqnarray*}
C(v, \varphi)  & = & \rho\int_{ S^+_{\rho}} |\frac {\partial v}{\partial
\nu}|^2-\frac
12|\nabla v|^2d\sigma \nonumber\\
& & -(1+\alpha)[\int_{D^+_{\rho}}2V^2(0)|x|^{2\alpha}e^{2v}-V(0)|x|^{\alpha}e^{v}|\varphi|^2dv+
\int_{L_{\rho}}cV(0)|x|^{\alpha}e^{v}ds ]\nonumber\\
& & +\rho\int_{S^+_{\rho}}V^2(0)|x|^{2\alpha}e^{2v}d\sigma +cV(0)|s|^{\alpha}se^{v}|^{s=\rho}_{s=-\rho} \nonumber\\
& & - \frac 14\int_{S^+_{\rho}}\la\frac {\partial
\varphi}{\partial \nu},(x+\bar{x})\cdot\varphi\ra d\sigma-\frac
14\int_{ S^+_{\rho}}\la (x+\bar{x})\cdot\varphi, \frac
{\partial \varphi}{\partial \nu}\ra d\sigma.\\
\end{eqnarray*}
On the other hand, we use the fact that $(v_n,\varphi_n)$ converges to $(v,\varphi)$ in $C^2_{loc}(\R^2_+)\cap C^1_{loc}(\overline{\R}^2_+\backslash \{0\})\times C^2_{loc}(\Gamma(\Sigma\R^2_+))\cap C^1_{loc}(\Gamma(\Sigma\overline{\R}^2_+\backslash \{0\}))$ again to get
\begin{eqnarray*}
& & \int_{D^+_{\rho}}2V^2(r_nx)|x|^{2\alpha}e^{2v_n}-V(r_nx)|x|^{\alpha}e^{v_n}|\varphi_n|^2dv+\int_{L_{\rho}}cV(r_nx)|x|^{\alpha}e^{v_n}ds\\
& \rightarrow &  \int_{D^+_{\rho}}2V^2(0)|x|^{2\alpha}e^{2v}-V(0)|x|^{\alpha}e^{v}|\varphi|^2dv+\int_{L_{\rho}}cV(0)|x|^{\alpha}e^{v}ds+\beta
\end{eqnarray*}
as $n\rightarrow \infty$. By using Green's representation formula for $u_n$ in $D^+_\rho$ and then take $n\rightarrow \infty$, we have
$$
v(x)=\frac{\beta}{\pi}\ln\frac{1}{|x|}+\phi(x)+\gamma(x),
$$  where
\begin{eqnarray*}
\phi(x)&=&\frac{1}{\pi}\int_{D_\rho^+}\ln\frac{1}{|x-y|}(2V^2(0)|y|^{2\alpha}e^{2v(y)}-V(0)|y|^{\alpha}e^{v(y)}|\varphi|^2(y))dy\\
& & +\frac{1}{\pi}\int_{L_r}\ln\frac{1}{|x-y|}(cV(0)|y|^{\alpha}e^{v(y)})dy,
\end{eqnarray*}
and
$$
\gamma(x)=\frac{1}{\pi}\int_{S^+_\rho}\ln\frac{1}{|x-y|}\frac{\partial v}{\partial \nu}+\frac{(x-y)\cdot\nu}{|x-y|^2}v(y)dy.
$$
It is clear that $\gamma(x)$ is in $C^1(\overline{D^+_\rho})$ and $\phi$ satisfies
\begin{equation*}
\left\{
\begin{array}{rcll}
-\triangle \phi &=& 2V^2(0)|x|^{2\alpha}e^{2v}-V(0)|x|^{\alpha}e^{v}|\varphi|^2, & \quad \text{ in }D_\rho^+,\\
\frac{\partial \phi}{\partial \nu}&=& cV(0)|x|^{\alpha}e^{v}, &\quad \text{ on }
 L_\rho.\\
\end{array}
\right.
\end{equation*}
By similar arguments as the proof of  Propostion \ref{thm-sigu-move1-B}, we can obtain that
$$
C(v,\varphi)=\frac{\beta^2}{2\pi},
$$
This implies that
$$
(1+\alpha)\beta=\frac{\beta^2}{2\pi}.
$$
Noticing that $\int_{D^+_\rho}|x|^{2\alpha}e^{2v}dx<\infty$, we have $\beta\leq (1+\alpha)\pi$. Therefore we obtain that $\beta=0$, i.e. $C(v,\varphi)=0$, and the singularity at $0$ of $(v,\varphi)$ is removed by Propostion \ref{thm-sigu-move1-B}. Forthermore, the singularity at $\infty$ of $(v,\varphi)$ is also removed by Theorem \ref{rgs}. Thus  we get another bubble on $S^2_{c'}$, and  we get a contradiction to the
assumption that $m=1$. Concequently we complete the proof of the claim 1.

\
\

\noindent {\bf Claim 2}  We can separate $A^+_{\delta, R, n}$ into
finitely many parts
$$
A^+_{\delta, R, n}=\bigcup_{k=1}^{N_k}A^+_k$$ such that on each part
$$
\int_{A^+_k}|x|^{2\alpha}e^{2u_n}\leq \frac 1{4\Lambda^2},\quad k=1,2,\cdots ,N_k.
$$
where  $N_k\leq N_0$ for $N_0$  is a uniform integer for all
$n$ large enough, $A^+_k=D^+_{r^{k-1}}\setminus D^+_{r^k}$, $r^0=\delta, r^{N_k}=\lambda_nR
$, $r^k<r^{k-1}$ for $k=1,2,\cdots, N_k$, and $C_0$ is a constant as in Lemma \ref{main-lamm}.

\
\

The proof of this claim is very similar to those in \cite{JWZZ1, JZZ1, Z1} and the argument is now standard, so we omit it.
\
\

Now we apply {\bf Claim 1} and {\bf Claim 2} to prove (\ref{nev}). Let
$\varepsilon>0$ be small, and let $\delta$ be small enough, and let
$R$ and $ n$ be big enough. We apply Lemma \ref{main-lamm} to each
part $A_k^+$ to obtain
\begin{eqnarray*}
(\int_{A^+_l}|\Psi_n|^4)^{\frac 14} &\leq & C_0
(\int_{D^+_{er^{l-1}}\setminus D^+_{e^{-1}r^l}}|x|^{2\alpha}e^{2u_n})^{\frac
12}(\int_{D^+_{er^{l-1}}\setminus D^+_{e^{-1}r^l}}|\Psi_n|^4)^{\frac
14}\\
&+&C(\int_{D^+_{er^{l-1}}\setminus D^+_{r^{l-1}}}|\Psi_n|^4)^{\frac
14}+C(\int_{D^+_{r^{l}}\setminus D^+_{e^{-1}r^l}}|\Psi_n|^4)^{\frac 14}\\
&\leq & C_0 ((\int_{A^+_l}|x|^{2\alpha}e^{2u_n})^{\frac 12}+\varepsilon^{\frac
12}+\varepsilon^{\frac 12})((\int_{A^+_l}|\Psi_n|^4)^{\frac
14}+\varepsilon^{\frac 14}+\varepsilon^{\frac 14})+C\varepsilon^{\frac 14}\\
&\leq &C_0 (\int_{A^+_l}|x|^{2\alpha}e^{2u_n})^{\frac
12}(\int_{A^+_l}|\Psi_n|^4)^{\frac 14}+C(\varepsilon^{\frac
14}+\varepsilon^{\frac 12}+\varepsilon^{\frac 34})\\
& \leq & \frac 12 (\int_{A^+_l}|\Psi_n|^4)^{\frac
14}+C(\varepsilon^{\frac 14}+\varepsilon^{\frac
12}+\varepsilon^{\frac 34}).
\end{eqnarray*}
Therefore we have
\begin{equation*}(\int_{A^+_l}|\Psi_n|^4)^{\frac
14}\leq C(\varepsilon^{\frac 14}+\varepsilon^{\frac
12}+\varepsilon^{\frac 34}).
\end{equation*}
Since  $\varepsilon $ is small, we may assume $\varepsilon\leq 1$.
Then we get
\begin{equation}\label{2.1}
(\int_{A^+_l}|\Psi_n|^4)^{\frac 14}\leq C\varepsilon^{\frac 14}.
\end{equation}
With similar arguments, and using (\ref{2.1}), we have
\begin{equation}\label{2.2}
(\int_{A^+_l}|\nabla\psi_n|^{\frac 43})^{\frac 34}\leq
C\varepsilon^{\frac 14}.
\end{equation}
Summing up (\ref{2.1}) and (\ref{2.2}) on $A^+_l$ we get

\begin{equation}\label{2.3}
\int_{A^+_{\delta,
R,n}}|\Psi_n|^4+\int_{A^+_{\delta,R,n}}|\nabla\psi_n|^{\frac
43}=\sum_{l=1}^{N_0}\int_{A^+_l}|\Psi_n|^4+|\nabla\psi_n|^{\frac
43}\leq C\varepsilon^{\frac 13}.
\end{equation}
Thus we have shown (\ref{nev}) in the first case.

\
\

For {\bf Case II}, according the blow-up process, we define the neck domain is
$$ A^+_{S, R, n}=\{x\in \R^2_+|\rho_n R\leq |x-x_n|\leq |x_n|S \}.$$

Notice that
\begin{eqnarray*}
& & \int_{D^+_{\delta}}|\Psi_n|^4dv=\int_{D^+_{\frac{\delta}{|x_n|}}}|\overline{\Psi}_n|^4dv\\
& = & \int_{D^+_{\frac{\delta}{|x_n|}}\backslash D^+_{R_1}(y_n)}|\overline{\Psi}_n|^4dv+\int_{D^+_{R_1}(y_n)\backslash D^+_{\delta_n R_2}(y_n) }|\overline{\Psi}_n|^4dv+\int_{D^+_{\delta_n R_2}(y_n)}|\overline{\Psi}_n|^4dv\\
& = & \int_{D^+_{\frac{\delta}{|x_n|}}\backslash D^+_{R_1}(y_n)}|\overline{\Psi}_n|^4dv+\int_{D^+_{|x_n|R_1}(x_n)\backslash D^+_{|x_n|\delta_n R_2}(x_n) }|\Psi_n|^4dv+\int_{D^+_{\delta_n R_2}(y_n)}|\overline{\Psi}_n|^4dv.
\end{eqnarray*}
Duo to the assumption that $(u_n,\Psi_n)$ has only one bubble at the blow-up point $p=0$, $(\overline{u}_n,\overline{\Psi}_n)$ also has only one bubble at its blow-up point $y_0$. Therefore, we have
$$
\lim_{\delta\rightarrow 0}\lim_{R_1\rightarrow \infty}\lim_{n\rightarrow \infty}\int_{D^+_{\frac{\delta}{|x_n|}}\backslash D^+_{R_1}(y_n)}|\overline{\Psi}_n|^4dv=0.
$$
While $D^+_{\delta_n R_2}(y_n)$ is a bubble domain, we know to prove (\ref{nev}) it is sufficient to prove that
\begin{equation}\label{nev1}
\lim_{S\rightarrow \infty}\lim_{R\rightarrow \infty}\lim_{n\rightarrow \infty}\int_{A^+_{S, R, n}}|{\Psi}_n|^4dv=0.
\end{equation}

To prove (\ref{nev1}), by using the similar argument as the case 1, we have the following facts:

\

\noindent{\bf Fact II.1}:  For any small $\varepsilon>0$, and $T>0$, there exists some $N(T)$ such that for any $n\geq N(T)$ we
have
\begin{equation*}
\int_{D^+_{|x_n|S}(x_n) \setminus D^+_{|x_n|Se^{-T}}(x_n)}(|x|^{2\alpha}e^{2u_n}+|\Psi_n|^4)+\int_{\partial( D^+_{|x_n|S}(x_n) \setminus D^+_{|x_n|Se^{-T}}(x_n))\cap
\partial \R^2_+}|x|^{\alpha}e^{u_n}<\varepsilon,
\end{equation*}
for sufficiently large $S$.

\

\noindent {\bf Fact II.2}: For any small $\varepsilon>0$, and $T>0$, we may
choose an $N(T)$ such that when $n\geq N(T)$
\begin{eqnarray*}
& & \int_{D^+_{\rho_nRe^T}(x_n)\setminus
D^+_{\rho_nR}(x_n)}(|x|^{2\alpha}e^{2u_n}+|\Psi_n|^4)+\int_{\partial
(D^+_{\rho_nRe^T}(x_n)\setminus D^+_{\rho_nR}(x_n))\cap
\partial \R^2_+}|x|^{\alpha}e^{u_n}\\
&=& \int_{(D_{Re^T}\setminus
D_{R})\cap \{t>-\frac{t_n}{\rho_n}\}}(|\frac{x_n}{|x_n|}+\frac{\rho_n}{|x_n|}x|^{2\alpha}e^{2\widetilde{u}_n}+|\widetilde{\Psi}_n|^4)+\int_{(D_{Re^T}\setminus D_{R})\cap
\{t=-\frac{t_n}{\rho_n}\}}|\frac{x_n}{|x_n|}+\frac{\rho_n}{|x_n|}x|^{\alpha}e^{\widetilde{u}_n}\\
& < & \varepsilon,
\end{eqnarray*}
if $R$ is large enough.

\

Buy using the above two facts, we need to prove the following claim:

\

\noindent{\bf Claim II.1} For any
$\varepsilon>0$, there is an $N>1$ such that for any $n\geq N$, we
have
$$
\int_{D^+_r(x_n)\setminus
D^+_{e^{-1}r}(x_n)}(|x|^{2\alpha}e^{2u_n}+|\Psi_n|^4)+\int_{\partial
(D^+_r(x_n)\setminus D^+_{e^{-1}r}(x_n))\cap
\partial \R^2_+}|x|^{\alpha}e^{u_n}<\varepsilon; \quad
\forall r\in [e\rho_n R, |x_n|S].
$$

\

\noindent{\it Proof of Claim II.1} We assume  by a contradiction that there exists $\varepsilon_0>0$ and a sequence
${r_n}$, $r_n\in [e\rho_n R, |x_n|S]$, such that
$$
\int_{D^+_{r_n}(x_n)\setminus
D^+_{e^{-1}r_n}(x_n)}(|x|^{2\alpha}e^{2u_n}+|\Psi_n|^{4})+\int_{\partial
(D^+_{r_n}(x_n)\setminus D^+_{e^{-1}r_n}(x_n))\cap
\partial \R^2_+}|x|^{\alpha}e^{u_n}\geq \varepsilon_0.
$$
Then, by Fact II.1 and Fact II.2,  we know that $\frac
{|x_n|S}{r_n}\rightarrow +\infty$ and
$\frac{\rho_n R}{r_n}\rightarrow 0$, in particular, $r_n\rightarrow
0$ as $n\rightarrow +\infty$. We assume that $\Lambda =\lim_{n\rightarrow \infty}\frac{t_n}{r_n}$. Here $\Lambda$ is either a nonnegative real number or $+\infty$. Next we proceed by distinguishing two cases.

\

\noindent{\bf Case II.1} $\Lambda>0$.

\

In this case, we note that $D_{r_n\rho}(x_n)$ is in $\R^2_+$ when $n$ is sufficient small and $0<\rho<\Lambda$. We define the rescaling functions again
\begin{equation*}
\left\{
\begin{array}{rcl}
v_n(x)&=& u_n(r_nx+x_n)+\ln (r_n|x_n|^{\alpha}), \\
\varphi_n(x)&=& r_n^{\frac 12}\Psi(r_nx+x_n)\\
\end{array}
\right.
\end{equation*}
for any  $r_nx+x_n\in D^+_{|x_n|S}(x_n)\setminus
D^+_{\rho_n R}(x_n)$.
Then $(v_n(x), \varphi_n(x))$ satisfies that
\begin{equation}\label{3-1}
\int_{(D_1\setminus D_{e^{-1}})\cap \{t>-\frac{t_n}{r_n}\}}(|\frac{x_n}{|x_n|}+\frac{r_n}{|x_n|}x|^{2\alpha}e^{2v_n}+|\varphi_n|^4)+\int_{(D_1\setminus D_{e^{-1}})\cap
\{t=-\frac{t_n}{r_n}\}}|\frac{x_n}{|x_n|}+\frac{r_n}{|x_n|}x|^{\alpha}e^{v_n}\geq
\varepsilon_0.
\end{equation}
Note that $(v_n,\varphi_n)$ satisfies for any $R>0$ and $S>0$
\begin{equation*}
\left\{
\begin{array}{rcll}
-\triangle v_n(x) &=& 2V^2(r_nx+x_n)|\frac{x_n}{|x_n|}+\frac{r_n}{|x_n|}x|^{2\alpha}e^{2v_n(x)}\\
& &-V(r_nx+x_n)|\frac{x_n}{|x_n|}+\frac{r_n}{|x_n|}x|^{\alpha}e^{v_n(x)}
|\varphi_n(x)|^2, & \text{ in }
(D_{\frac{|x_n|S}{r_n}}\setminus D_{\frac{\rho_nR}{r_n}})\cap \{t>-\frac{t_n}{r_n}\}, \\
\slashiii{D}\varphi_n(x) &=& -V(r_nx+x_n)|\frac{x_n}{|x_n|}+\frac{r_n}{|x_n|}x|^{\alpha}e^{v_n(x)}\varphi_n(x), &
 \text{ in } (D_{\frac{|x_n|S}{r_n}}\setminus D_{\frac{\rho_nR}{r_n}})\cap \{t>-\frac{t_n}{r_n}\},\\
\frac{\partial v_n(x)}{\partial n}&=& cV(r_nx+x_n)|\frac{x_n}{|x_n|}+\frac{r_n}{|x_n|}x|^{\alpha}e^{v_n(x)}, &
\text{ on } (D_{\frac{|x_n|S}{r_n}}\setminus D_{\frac{\rho_nR}{r_n}})\cap \{t=-\frac{t_n}{r_n}\},\\
B\varphi_n(x)&=&0, & \text{ on }
(D_{\frac{|x_n|S}{r_n}}\setminus D_{\frac{\rho_nR}{r_n}})\cap \{t=-\frac{t_n}{r_n}\}.
\end{array}
\right.
\end{equation*}

According to Theorem \ref{mainthm}, there are three possible cases. Similar to the Case I, we can rule out the first and the second possible cases. If the third case happens, then there is no blow-up point in $(D_R\setminus D_{\frac 1R})\cap \{t\geq -b\} $ for any  $R>0$ and  any $b<\Lambda$. Furthermore $(v_n, \varphi_n)$  will converge to $(v,\varphi)$ strongly on $(D_R\setminus D_{\frac 1R})\cap \{t\geq -b\}$. If $\Lambda >0$, then $(v,\varphi)$ satisfies
\begin{equation}\label{8-1}
\left\{
\begin{array}{rcll}
-\triangle v &=& 2V^2(0)e^{2v}-V(0)e^{v}|\varphi|^2, & \quad \text{ in }\R^2_\Lambda\setminus\{0\},\\
\slashiii{D}\varphi &=& -V(0)e^{v}\varphi, & \quad \text{ in
}\R^2_\Lambda\setminus\{0\},\\
\frac{\partial v}{\partial n}&=& cV(0)e^{v}, &\quad \text{ on }
 \partial \R^2_\Lambda(\text { in the case of } \Lambda < +\infty), \\
B\varphi &=& 0, &\quad \text{ on }  \partial \R^2_\Lambda(\text { in the case of  } \Lambda < +\infty)
\end{array}
\right.
\end{equation}
with finite energy.

Since $D_{r_n\rho}(x_n)$ contains completely in $\R^2_+$ when $n$ is sufficient small and $0<\rho<\Lambda$, we know that  the origin is acturally an interior singular point of $(v,\varphi)$ to  (\ref{8-1}). Then this local singular can be removed by using the similar arguments in the case II of \cite{JZZ3}. After removing the local sigularity $0$, we can remove  the singularity  at  $\infty$ of $(v,\varphi)$ to (\ref{8-1}) by Theorem \ref{rgs}. Thus  we get another bubble on $S^2_{c'}$, and  we get a contradiction to the assumption that $m=1$. Concequently we complete the proof of the claim II.1.

\

\noindent{\bf Case II.2} $\Lambda=0$.

\

In this case, noticing that $x_n=(s_n,t_n)$ and   $\lim_{n\rightarrow \infty}\frac
{|x_n|}{r_n}=+\infty$, we have   $\lim_{n\rightarrow \infty}\frac{|s_n|}{t_n}=+\infty$ and $\lim_{n\rightarrow \infty}\frac{|s_n|}{r_n}=+\infty$. We set $x'_n=(s_n, 0)$. Then we define the rescaling functions in this case
\begin{equation*}
\left\{
\begin{array}{rcl}
v_n(x)&=& u_n(r_nx+x'_n)+\ln (r_n|s_n|^{\alpha}), \\
\varphi_n(x)&=& r_n^{\frac 12}\Psi(r_nx+x'_n)\\
\end{array}
\right.
\end{equation*}
for any  $r_nx+x'_n\in D^+_{|x_n|S}(x'_n)\setminus
D^+_{\rho_n R}(x'_n)$.
Since that
\begin{eqnarray*}
& & \int_{D^+_{\frac 32 r_n}(x'_n)\setminus D^+_{\frac {1}{2}e^{-1}r_n}(x'_n)}(|x|^{2\alpha}e^{2u_n}+|\Psi_n|^{4})+\int_{\partial
(D^+_{\frac 12 r_n}(x'_n)\setminus D^+_{\frac 32 e^{-1}r_n}(x'_n))\cap
\partial \R^2_+}|x|^{\alpha}e^{u_n}\nonumber \\
& \geq & \int_{D^+_{r_n}(x_n)\setminus
D^+_{e^{-1}r_n}(x_n)}(|x|^{2\alpha}e^{2u_n}+|\Psi_n|^{4})+\int_{\partial
(D^+_{r_n}(x_n)\setminus D^+_{e^{-1}r_n}(x_n))\cap
\partial \R^2_+}|x|^{\alpha}e^{u_n}\nonumber\\
& \geq &  \varepsilon_0,
\end{eqnarray*}
we have that $(v_n(x), \varphi_n(x))$ satisfies that
\begin{equation}\label{8-3}
\int_{D^+_{\frac 32}\setminus D^+_{\frac { e^{-1}}2}}(|\frac{x'_n}{|s_n|}+\frac{r_n}{|s_n|}x|^{2\alpha}e^{2v_n}+|\varphi_n|^4)+\int_{\partial(D^+_{\frac{3}{2}}\setminus D^+_{\frac {e^{-1}}2})\cap
\{t=0\}}|\frac{x'_n}{|s_n|}+\frac{r_n}{|s_n|}x|^{\alpha}e^{v_n}\geq
\varepsilon_0.
\end{equation}
Note that $(v_n,\varphi_n)$ satisfies for any $R>0$ and $S>0$
\begin{equation*}
\left\{
\begin{array}{rcll}
-\triangle v_n(x) &=& 2V^2(r_nx+x'_n)|\frac{x'_n}{|s_n|}+\frac{r_n}{|s_n|}x|^{2\alpha}e^{2v_n(x)}\\
& &-V(r_nx+x'_n)|\frac{x'_n}{|s_n|}+\frac{r_n}{|s_n|}x|^{\alpha}e^{v_n(x)}
|\varphi_n(x)|^2, & \text{ in }
(D^+_{\frac{|x_n|S}{r_n}}\setminus D^+_{\frac{\rho_nR}{r_n}}), \\
\slashiii{D}\varphi_n(x) &=& -V(r_nx+x'_n)|\frac{x'_n}{|s_n|}+\frac{r_n}{|s_n|}x|^{\alpha}e^{v_n(x)}\varphi_n(x), &
 \text{ in } (D^+_{\frac{|x_n|S}{r_n}}\setminus D^+_{\frac{\rho_nR}{r_n}}),\\
\frac{\partial v_n(x)}{\partial n}&=& cV(r_nx+x'_n)|\frac{x'_n}{|s_n|}+\frac{r_n}{|s_n|}x|^{\alpha}e^{v_n(x)}, &
\text{ on } \partial (D^+_{\frac{|x_n|S}{r_n}}\setminus D^+_{\frac{\rho_nR}{r_n}})\cap \partial \R^2_+,\\
B\varphi_n(x)&=&0, & \text{ on }
\partial (D^+_{\frac{|x_n|S}{r_n}}\setminus D^+_{\frac{\rho_nR}{r_n}})\cap \partial \R^2_+.
\end{array}
\right.
\end{equation*}
According to Theorem \ref{mainthm}, there are three possible cases. From (\ref{8-3}), we can rule out the first and the second possible cases by using the simiar arguments of Case I.  Next we assume that the third case happens, i.e. there is no blow-up point in $\overline{D^+_R\setminus D^+_{\frac 1R}} $ for any  $R>0$.  Furthermore $(v_n, \varphi_n)$  will converge to $(v,\varphi)$ strongly on $\overline{D^+_R\setminus D^+_{\frac 1R}} $, and $(v,\varphi)$ satisfies
\begin{equation}
\left\{
\begin{array}{rcll}
-\triangle v &=& 2V^2(0)e^{2v}-V(0)e^{v}|\varphi|^2, & \quad \text{ in }\R^2_+,\\
\slashiii{D}\varphi &=& -V(0)e^{v}\varphi, & \quad \text{ in
}\R^2_+,\\
\frac{\partial v}{\partial n}&=& cV(0)e^{v}, &\quad \text{ on }
 \partial \R^2_+\setminus\{0\}, \\
B\varphi &=& 0, &\quad \text{ on }  \partial \R^2_+\setminus\{0\}
\end{array}
\right.
\end{equation}\label{8-2}
with finite energy.

Next we will remove two singular points at $0$  and at $\infty$, and concequently we get the second bubble of the considered system. Thus we get a contradiction. To this purpose, let us computate the Pohozaev constant of $(v, \varphi)$. Let start with the Pohozaev identity of $(u_n, \Phi_n)$.  We multiply all terms in \eqref{Eq-B} by $(x-x'_n)\cdot \nabla u_n$ and integrate over  $D^+_{r_n\rho}(x'_n)$. It follows for any sufficient small  $\rho$ that
\begin{eqnarray*}
&& r_n\rho\int_{S^+_{r_n\rho}(x'_n)} |\frac {\partial u_n}{\partial
\nu}|^2-\frac
12|\nabla u_n|^2d\sigma \nonumber\\
&=& \int_{D^+_{r_n\rho}(x'_n)}(2V^2(x)|x|^{2\alpha}e^{2u_n}-V(x)|x|^{\alpha}e^{u_n}|\Psi_n|^2)dv+
\int_{L_{r_n\rho}(x'_n)}cV(x)|x|^{\alpha}e^{u_n}ds \nonumber\\
& & -r_n\rho\int_{S^+_{r_n\rho}(x'_n)}V^2(x)|x|^{2\alpha}e^{2u_n}d\sigma +\int_{L_{r_n\rho}(x'_n)}c\frac {\partial (V(s,0)|s|^{\alpha})}{\partial s}(s-s_n)e^{u_n}ds\nonumber\\
& & -cV(s,0)|s|^{\alpha}(s-s_n)e^{u_n}|^{s=x'_n+r_n\rho}_{s=x'_n-r_n\rho} \nonumber\\
& &+\int_{D^+_{r_n\rho}(x'_n)}(x-x'_n)\cdot \nabla (V^2(x)|x|^{2\alpha})e^{2u_n}dv -\int_{D^+_{r_n\rho}(x'_n)}(x-x'_n)\cdot \nabla
(V(x)|x|^{\alpha})e^{u_n}|\Psi_n|^2dv \nonumber\\
& & + \frac 14\int_{S^+_{r_n\rho}(x'_n)}\la\frac {\partial
\Psi}{\partial \nu},(x+\bar{x}-2x'_n)\cdot\Psi\ra d\sigma+\frac
14\int_{S^+_{r_n\rho}(x'_n)}\la (x+\bar{x}-2x'_n)\cdot\Psi, \frac
{\partial \Psi}{\partial \nu}\ra d\sigma
\end{eqnarray*}
Hence for rescaling functions $(v_n, \varphi_n)$ we have
\begin{eqnarray}\label{8-4}
&& \rho\int_{ S^+_{\rho}} |\frac {\partial v_n}{\partial
\nu}|^2-\frac
12|\nabla v_n|^2d\sigma \nonumber\\
&=& \int_{D^+_{\rho}}2V^2(r_nx+x'_n)|\frac {x'_n}{|s_n|}+\frac{r_n}{|s_n|}x|^{2\alpha}e^{2v_n}-V(r_nx+x'_n)|\frac {x'_n}{|s_n|}+\frac{r_n}{|s_n|}x|^{\alpha}e^{v_n}|\varphi_n|^2dv\nonumber\\
& & +\int_{L_{\rho}}cV(r_nx+x'_n)|\frac {x'_n}{|s_n|}+\frac{r_n}{|s_n|}x|^{\alpha}e^{v_n}ds \nonumber\\
& & -\rho\int_{S^+_{\rho}}V^2(r_nx+x'_n)|\frac {x'_n}{|s_n|}+\frac{r_n}{|s_n|}x|^{2\alpha}e^{2v_n}d\sigma +\int_{ L_{\rho}}c\frac {\partial( V((r_ns+s_n,0))|\frac {r_n}{|s_n|}s+\frac {s_n}{|s_n|}|^{\alpha})}{\partial s}se^{v_n}ds\nonumber\\
& & -cV((r_ns+s_n,0))|\frac {r_n}{|s_n|}s+\frac {s_n}{|s_n|}|^{\alpha}se^{v_n}|^{s=\rho}_{s=-\rho} \nonumber\\
& & +\int_{D^+_{\rho}}x\cdot \nabla (V^2(r_nx+x'_n)|\frac {x'_n}{|s_n|}+\frac{r_n}{|s_n|}x|^{2\alpha}) e^{2v_n}dv -\int_{D^+_{\rho}}x\cdot \nabla
(V(r_nx+x'_n)|\frac {x'_n}{|s_n|}+\frac{r_n}{|s_n|}x|^{\alpha})e^{v_n}|\varphi_n|^2dv \nonumber\\
& & + \frac 14\int_{S^+_{\rho}}\la\frac {\partial
\varphi_n}{\partial \nu},(x+\bar{x})\cdot\varphi_n\ra d\sigma+\frac
14\int_{ S^+_{\rho}}\la (x+\bar{x})\cdot\varphi_n, \frac
{\partial \varphi_n}{\partial \nu}\ra d\sigma.
\end{eqnarray}
Since  the associated Pohozaev constant of $(v_n,\varphi_n)$ is
\begin{eqnarray*}
& & C(v_n, \varphi_n)  =  C(v_n, \varphi_n,\rho)\\
& = & \rho\int_{ S^+_{\rho}} |\frac {\partial v_n}{\partial
\nu}|^2-\frac
12|\nabla v_n|^2d\sigma \nonumber\\
&& -\int_{D^+_{\rho}}2V^2(r_nx+x'_n)|\frac {x'_n}{|s_n|}+\frac{r_n}{|s_n|}x|^{2\alpha}e^{2v_n}-V(r_nx+x'_n)|\frac {x'_n}{|s_n|}+\frac{r_n}{|s_n|}x|^{\alpha}e^{v_n}|\varphi_n|^2dv\\
& & -\int_{L_{\rho}}cV(r_nx+x'_n)|\frac {x'_n}{|s_n|}+\frac{r_n}{|s_n|}x|^{\alpha}e^{v_n}ds \nonumber\\
& & +\rho\int_{S^+_{\rho}}V^2(r_nx+x'_n)|\frac {x'_n}{|s_n|}+\frac{r_n}{|s_n|}x|^{2\alpha}e^{2v_n}d\sigma -\int_{ L_{\rho}}c\frac {\partial( V((r_ns+s_n,0))|\frac {r_n}{|s_n|}s+\frac {s_n}{|s_n|}|^{\alpha})}{\partial s}se^{v_n}ds\nonumber\\
& & +cV((r_ns+s_n,0))|\frac {r_n}{|s_n|}s+\frac {s_n}{|s_n|}|^{\alpha}se^{v_n}|^{s=\rho}_{s=-\rho} \nonumber\\
& & -\int_{D^+_{\rho}}x\cdot \nabla (V^2(r_nx+x'_n)|\frac {x'_n}{|s_n|}+\frac{r_n}{|s_n|}x|^{2\alpha}) e^{2v_n}dv +\int_{D^+_{\rho}}x\cdot \nabla
(V(r_nx+x'_n)|\frac {x'_n}{|s_n|}+\frac{r_n}{|s_n|}x|^{\alpha})e^{v_n}|\varphi_n|^2dv \nonumber\\
& & - \frac 14\int_{S^+_{\rho}}\la\frac {\partial
\varphi_n}{\partial \nu},(x+\bar{x})\cdot\varphi_n\ra d\sigma-\frac
14\int_{ S^+_{\rho}}\la (x+\bar{x})\cdot\varphi_n, \frac
{\partial \varphi_n}{\partial \nu}\ra d\sigma,
\end{eqnarray*}
we have from (\ref{8-4}) that
$$
C(v_n, \varphi_n)  =  C(v_n, \varphi_n,\rho)=0.
$$
Since that $|\frac {x'_n}{|s_n|}+\frac{r_n}{|s_n|}x|^{2\alpha}$ is a smooth function in $\overline {D^+_\rho}$, by the energy condition,
 $$\int_{D^+_\rho}|\frac {x'_n}{|s_n|}+\frac{r_n}{|s_n|}x|^{2\alpha}e^{2v_n}+|\varphi_n|^4dv+\int_{L_\rho}|\frac {x'_n}{|s_n|}+\frac{r_n}{|s_n|}x|^{\alpha}e^{v_n}ds<C,$$
we can easily to check that
$$
\lim_{\rho\rightarrow 0}\lim_{n\rightarrow \infty}\int_{D^+_{\rho}}x\cdot \nabla (V^2(r_nx+x'_n)|\frac {x'_n}{|s_n|}+\frac{r_n}{|s_n|}x|^{2\alpha}) e^{2v_n}dv +\int_{D^+_{\rho}}x\cdot \nabla
(V(r_nx+x'_n)|\frac {x'_n}{|s_n|}+\frac{r_n}{|s_n|}x|^{\alpha})e^{v_n}|\varphi_n|^2dv =0,
$$
and
$$
\lim_{\rho\rightarrow 0}\lim_{n\rightarrow \infty}\int_{ L_{\rho}}c\frac {\partial( V((r_ns+s_n,0))|\frac {r_n}{|s_n|}s+\frac {s_n}{|s_n|}|^{\alpha})}{\partial s}se^{v_n}ds=0.
$$
This implies that
\begin{eqnarray*}
0 & = & \lim_{\rho\rightarrow 0}\lim_{n\rightarrow \infty} C(v_n, \varphi_n,\rho) =  \lim_{\rho\rightarrow 0}C(v,\varphi,\rho)\\
& & -\lim_{r\rightarrow 0}\lim_{n\rightarrow \infty}\int_{D^+_{r}}2V^2(r_nx+x'_n)|\frac {x'_n}{|s_n|}+\frac{r_n}{|s_n|}x|^{2\alpha}e^{2v_n}-V(r_nx+x'_n)|\frac {x'_n}{|s_n|}+\frac{r_n}{|s_n|}x|^{\alpha}e^{v_n}|\varphi_n|^2dv\\
& & -\lim_{r\rightarrow 0}\lim_{n\rightarrow \infty}\int_{L_{r}}cV(r_nx+x'_n)|\frac {x'_n}{|s_n|}+\frac{r_n}{|s_n|}x|^{\alpha}e^{v_n}ds\\
& = & C(v,\varphi)-\beta.
\end{eqnarray*}
Here
\begin{eqnarray*}
\beta & = & \lim_{r\rightarrow 0}\lim_{n\rightarrow \infty}[\int_{D^+_{r}}2V^2(r_nx+x'_n)|\frac {x'_n}{|s_n|}+\frac{r_n}{|s_n|}x|^{2\alpha}e^{2v_n}-V(r_nx+x'_n)|\frac {x'_n}{|s_n|}+\frac{r_n}{|s_n|}x|^{\alpha}e^{v_n}|\varphi_n|^2dv]\\
& & +\lim_{r\rightarrow 0}\lim_{n\rightarrow \infty}\int_{L_{r}}cV(r_nx+x'_n)|\frac {x'_n}{|s_n|}+\frac{r_n}{|s_n|}x|^{\alpha}e^{v_n}ds.
\end{eqnarray*}
and $C(v,\varphi)=C(v,\varphi,\rho)$ is the Pohozaev constant of $(v,\varphi)$, i.e.
\begin{eqnarray*}
C(v, \varphi)  & = & \rho\int_{ S^+_{\rho}} |\frac {\partial v}{\partial
\nu}|^2-\frac
12|\nabla v|^2d\sigma \nonumber\\
& & -\int_{D^+_{\rho}}2V^2(0)e^{2v}-V(0)e^{v}|\varphi|^2dv+
\int_{L_{\rho}}cV(0)e^{v}ds \nonumber\\
& & +\rho \int_{S^+_{\rho}}V^2(0)e^{2v}d\sigma +cV((0,0))se^{v}|^{s=\rho}_{s=-\rho} \nonumber\\
& & - \frac 14\int_{S^+_{\rho}}\la\frac {\partial
\varphi}{\partial \nu},(x+\bar{x})\cdot\varphi\ra d\sigma-\frac
14\int_{ S^+_{\rho}}\la (x+\bar{x})\cdot\varphi, \frac
{\partial \varphi}{\partial \nu}\ra d\sigma.\\
\end{eqnarray*}
On the other hand, we use the fact that $(v_n,\varphi_n)$ converges to $(v,\varphi)$ in $C^2_{loc}(\R^2_+)\cap C^1_{loc}(\overline{\R}^2_+\backslash \{0\})\times C^2_{loc}(\Gamma(\Sigma\R^2_+))\cap C^1_{loc}(\Gamma(\Sigma\overline{\R}^2_+\backslash \{0\}))$ again to get
\begin{eqnarray*}
& & \int_{D^+_{r}}2V^2(r_nx+x'_n)|\frac {x'_n}{|s_n|}+\frac{r_n}{|s_n|}x|^{2\alpha}e^{2v_n}-V(r_nx+x'_n)|\frac {x'_n}{|s_n|}+\frac{r_n}{|s_n|}x|^{\alpha}e^{v_n}|\varphi_n|^2dv\\
&& +\int_{L_{r}}cV(r_nx+x'_n)|\frac {x'_n}{|s_n|}+\frac{r_n}{|s_n|}x|^{\alpha}e^{v_n}ds\\
& \rightarrow &  \int_{D^+_{\rho}}2V^2(0)e^{2v}-V(0)e^{v}|\varphi|^2dv+\int_{L_{\rho}}cV(0)e^{v}ds+\beta
\end{eqnarray*}
as $n\rightarrow \infty$. By using Green's representation formula for $u_n$ in $D^+_\rho$ and then take $n\rightarrow \infty$, we have
$$
v(x)=\frac{\beta}{\pi}\ln\frac{1}{|x|}+\phi(x)+\gamma(x),
$$  with $\phi$ being a bounded term and $\gamma(x)$ being a regular term. Consenquently, we can obtain that
$$
C(v,\varphi)=\frac{\beta^2}{2\pi}.
$$
This implies that
$$
\beta=\frac{\beta^2}{2\pi}.
$$
Noticing that $\int_{D^+_\rho}e^{2v}dx<\infty$, we have $\beta\leq \pi$. Therefore we obtain that $\beta=0$, i.e. $C(v,\varphi)=0$, and the singularity at $0$ of $(v,\varphi)$ is removed by Propostion \ref{thm-sigu-move1-B}. Forthermore, the singularity at $\infty$ of $(v,\varphi)$ is also removed by Theorem \ref{rgs}. Thus  we get another bubble on $S^2_{c'}$, and  we get a contradiction to the
assumption that $m=1$. Concequently we complete the proof of {\bf Claim II.1}.

\
\

Next , similarly to {\bf Case I}. we can prove the following:\\

\noindent {\bf Claim II.2}  We can separate $A^+_{S, R, n}(x_n)$ into
finitely many parts
$$
A^+_{S, R, n}=\bigcup_{k=1}^{N_k}A^+_k$$ such that on each part
$$
\int_{A^+_k}|x|^{2\alpha}e^{2u_n}\leq \frac 1{4\Lambda^2},\quad k=1,2,\cdots ,N_k.
$$
where  $N_k\leq N_0$ for $N_0$  is a uniform integer for all
$n$ large enough, $A^+_k=D^+_{r^{k-1}}(x_n)\setminus D^+_{r^k}(x_n)$, $r^0=\delta, r^{N_k}=\lambda_nR
$, $r^k<r^{k-1}$ for $k=1,2,\cdots, N_k$, and $C_0$ is a constant as in Lemma \ref{main-lamm}.\\

Then,  by using {\bf Claim II.1} and {\bf Claim II.2} we can complete the proof of the result.  \qed

 \
 \

\section{Blow-up Behavior}\label{blow}

In this section, we will show that  $u_n\rightarrow -\infty$ uniformly on compact subset of $(D^+_r\cup L_r)\setminus \Sigma_1$ in means of the energy identity for spinors. Thus we rule out the possibility that $u_n$ is uniformly bounded in $L^{\infty}_{loc}((D^+_r\cup L_r)\setminus\Sigma_1)$ in Theorem \ref{mainthm}. The following  is the proof of Theorem \ref{mainthm-1}.

\

\noindent{\bf Proof of Theorem \ref{mainthm-1}:}
We prove the results by contradiction. Assume that the conclusion of the theorem is false.
Then by Theorem \ref{mainthm}, $u_n$ is uniformly bounded in $L_{loc}^{\infty}((D^+_r\cup L_r)\setminus\Sigma_1)$. Thus we know that $(u_n, \Psi_n)$
converges in $C^2$ on any compact subset of $(D^+_r\cup L_r)\setminus\Sigma_1$ to $(u,\Psi)$, which satisfies that
\begin{equation}\label{Eq-B-1}
\left\{
\begin{array}{rcll}
-\Delta u(x) &=& 2u^2(x)|x|^{2\alpha}e^{2u(x)}-V(x)|x|^{\alpha}e^{u(x)}|\Psi|^2,  \quad &\text {in } D^+_{r}\setminus \Sigma_1, \\
\slashiii{D}\Psi &=& -V(x)|x|^{\alpha}e^{u(x)}\Psi, \quad &\text {in } D^+_{r}\setminus \Sigma_1, \\
\frac {\partial u}{\partial n}& = & cV(x)|x|^{\alpha}e^{u(x)},\quad & \text { on } L_r\setminus \Sigma_1, \\
B(\Psi) &=& 0,\quad & \text { on } L_r\setminus \Sigma_1.
\end{array}
\right.
\end{equation}
with bounded energy
$$\int_{D^+_r}(|x|^{2\alpha}e^{2u}+|\Psi|^4)dx +\int_{L_r}|x|^\alpha e^u ds< +\infty.$$

Since the blow-up set $\Sigma_1$ is not empty, we can take a point $p \in \Sigma_1$. Choose a small $\delta_0>0$ such that
$p$ is the only point of $\Sigma _1$ in $\overline{D}_{2\delta_0} (p)\cap (D^+_r\cup L_r)=\{p\}$. If $p$ is the interior point of $D^+_r$, then we can choose $\delta_0$ sufficiently small such that $D_{2\delta_0}(p)\subset (D^+_r\cup L_r)$. Hence by Theorem 1.3 in \cite{JZZ3} we can get a contradiction.

Next we assume that $p$ is on $L_r$. Without loss of generality, we assume that $p=0$. The case of $p\neq 0$ can be dealed with in an analogous way.

We shall first show that the limit $(u,\Psi)$ is regular at the isolated singularity $p=0$, i.e.  $u\in C^{2}(D^+_r)\cap C^{1}(D^+_r\cup L_r)$
and
 $\Psi \in C^{2}(\Gamma(\Sigma D^+_r ))\cap C^{1}(\Gamma(\Sigma (D^+_r\cup L_r)))$ for some small $r>0$. To this end, we shall using Theorem \ref{thm-sigu-move1-B} for removability of a local singularity to remove the singularity. We know that the Phohozaev constant, denote $C_B(u, \Psi)$,  of $(u, \Psi)$ at $p=0$ is
\begin{eqnarray*}
C_B(u,\Psi) &:=& C_B(u,\Psi,\rho) = \rho\int_{S^+_\rho} |\frac {\partial u}{\partial
\nu}|^2-\frac
12|\nabla u|^2d\sigma \nonumber\\
& & -(1+\alpha)\int_{D^+_\rho}(2V^2(x)|x|^{2\alpha}e^{2u}-V(x)|x|^{\alpha}e^u|\Psi|^2)dv-
(\alpha+1)\int_{L_\rho}cV(x)|x|^{\alpha}e^{u}ds \nonumber\\
& & +\rho\int_{S^+_\rho}V^2(x)|x|^{2\alpha}e^{2u}d\sigma -\int_{L_\rho}c\frac {\partial V(s,0)}{\partial s}|s|^{\alpha}se^uds+cV(s,0)|s|^{\alpha}se^u|^{s=\rho}_{s=-\rho} \nonumber\\
& &-\int_{D^+_\rho}x\cdot \nabla (V^2(x))|x|^{2\alpha}e^{2u}dv +\int_{D^+_\rho}x\cdot \nabla
V(x)|x|^{\alpha}e^u|\psi|^2dv \nonumber\\
& & - \frac 14\int_{S^+_\rho}\la\frac {\partial
\Psi}{\partial \nu},(x+\bar{x})\cdot\Psi\ra d\sigma-\frac
14\int_{S^+_\rho}\la (x+\bar{x})\cdot\Psi, \frac
{\partial \Psi}{\partial \nu}\ra d\sigma
\end{eqnarray*} for any $0<\rho<\delta_0$.
On the other hand, since $(u_n,\Psi_n)$ are the regular solution, the Pohozaev constant $C_B(u_n, \Psi_n)=C_B(u_n, \Psi_n,\rho)$ satisfies

\begin{eqnarray*}
0 = C(u_n, \Psi_n) &=& C(u_n, \Psi_n,\rho)\\
&=& \rho\int_{ S^+_{\rho}} |\frac {\partial u_n}{\partial
\nu}|^2-\frac
12|\nabla u_n|^2d\sigma \nonumber\\
&&- (1+\alpha)\int_{D^+_{\rho}}2V^2(x)|x|^{2\alpha}e^{2u_n}-V(x)|x|^{\alpha}e^{u_n}|\Psi_n|^2dv-
(\alpha+1)\int_{L_{\rho}}cV(x)|x|^{\alpha}e^{u_n}ds \nonumber\\
& & +\rho\int_{S^+_{\rho}}V^2(x)|x|^{2\alpha}e^{2u_n}d\sigma -\int_{ L_{\rho}}c\frac {\partial V(s,0)}{\partial s}|s|^{\alpha}se^{u_n}ds+cV((s,0))|s|^{\alpha}se^{u_n}|^{s=\rho}_{s=-\rho} \nonumber\\
& & -\int_{D^+_{\rho}}x\cdot \nabla (V^2(x))|x|^{2\alpha}e^{2u_n}dv +\int_{D^+_{\rho}}x\cdot \nabla
V(x)|x|^{\alpha}e^{u_n}|\Psi_n|^2dv \nonumber\\
& & - \frac 14\int_{S^+_{\rho}}\la\frac {\partial
\Psi_n}{\partial \nu},(x+\bar{x})\cdot\Psi_n\ra d\sigma-\frac
14\int_{ S^+_{\rho}}\la (x+\bar{x})\cdot\Psi_n, \frac
{\partial \Psi_n}{\partial \nu}\ra d\sigma.
\end{eqnarray*}
Let $n\rightarrow \infty$ and $\rho\rightarrow 0$, by using that $(u_n,\Psi_n)$ converges to $(u,\Psi)$ regularily on any compact subset of $\overline{D}^+_{2\delta_0} \setminus \{0\}$ and that the energy  condition (\ref{condition}), to get
\begin{eqnarray*}
0 &=& \lim_{\rho\rightarrow 0}\lim_{n\rightarrow \infty}C(u_n,\Psi_n, \rho)=\lim_{\rho\rightarrow 0}C(u,\Psi, \rho) \\
& & -(1+\alpha)\lim_{\delta\rightarrow 0}\lim_{n\rightarrow \infty}\{\int_{D^+_\delta}(2V^2(x)|x|^{2\alpha}e^{2u_n}-V(x)|x|^{\alpha}e^{u_n}|\Psi_n|^2)dx+\int_{L_{\delta}}cV(x)|x|^{\alpha}e^{u_n}ds\}\\
&=&C(u,\Psi)-(1+\alpha)\beta,
\end{eqnarray*}
where
$$\beta=\lim_{\delta\rightarrow 0}\lim_{n\rightarrow \infty}\{\int_{D^+_\delta}(2V^2(x)|x|^{2\alpha}e^{2u_n}-V(x)|x|^{\alpha}e^{u_n}|\Psi_n|^2)dx+\int_{L_{\delta}}cV(x)|x|^{\alpha}e^{u_n}ds\}.$$
Moreover, we can also assume  that
\begin{eqnarray*}
&& (2V^2(x)|x|^{2\alpha}e^{2u_n}-V(x)|x|^{\alpha}e^{u_n}|\Psi_n|^2 )dx+ cV(x)|x|^{\alpha}e^{u_n}ds\\
&  \rightarrow & \nu =(2V^2(x)|x|^{2\alpha}e^{2u}-V(x)|x|^{\alpha}e^{u}|\Psi|^2)dx+ cV(x)|x|^{\alpha}e^{u}ds+\beta\delta_{p=0}
\end{eqnarray*}in the sense of distributions in $D^+_\delta\cup L_\delta$ for any small $\delta>0$. Then, applying similar arguments as in the proof of the local singularity removability in Claim I.1, Theorem \ref{engy-indt}, we can show that  $C_B(u,\Psi)=0$, $\beta = 0$ and hence $(u, \Psi)$ is a regular solution of \eqref{Eq-B} on $D^+_{2\delta_0}$ with bounded energy
$$\int_{D^+_{2\delta_0}}(|x|^{2\alpha} e^{2u}+|\Psi|^4)dx+ \int_{L_{2\delta}}|x|^{\alpha}e^{u}ds< +\infty.$$
Hence, we can choose some small $\delta_1 \in (0, \delta_0)$ such that for any $\delta \in (0, \delta_1)$,
\begin{equation} \label{less1/10}
\int_{D^+_{\delta}}(2V^2(x)|x|^{2\alpha}e^{2u}-V(x)|x|^{\alpha}e^{u}|\Psi|^2)dx +\int_{L_\delta} cV(x)|x|^{\alpha}e^{u}ds< {\rm min} \{ \frac{1+\a}{10},  \frac{1}{10}\}.
\end{equation}

Next, as  in the proof of Theorem \ref{engy-indt}, we rescale $(u_n,\Psi_n)$ near $p=0$. We let $x_n\in \overline{D}^+_{\delta_1}$ such that
$u_n(x_n)=\max_{\overline{D}^+_{\delta_1}}u_n(x)$. Write $x_n=(s_n,t_n)$.  Then $x_n\rightarrow p=0$ and
$u_n(x_n)\rightarrow +\infty$. Define  $\lambda_n=e^{-\frac{u_n(x_n)}{\alpha+1}}$. It is clear that $\lambda_n$, $|x_n|$ and $t_n$ converge to $0$ as $n\rightarrow 0$. we will proceed by distinguishing the following  three cases:

\

\noindent{\bf Case I.} $\frac{|x_n|}{\lambda_n}=O(1)$ as $n\rightarrow +\infty$.

\

In this case, the rescaling functions are
\begin{equation*}
\left\{
\begin{array}{rcl}
\widetilde{u}_n(x)&=&u_n(\lambda_nx)+(1+\alpha)\ln {\lambda_n}\\
\widetilde{\Psi}_n(x)&=&\lambda^{\frac
12}_n\Psi_n(\lambda_nx)
\end{array}
\right.
\end{equation*}
for any $x \in \overline {D}^+_{\frac{\delta_1}{2\lambda_n}}$. Moreover, by passing to a subsequence,
  $(\widetilde{u}_n,\widetilde{\Psi}_n)$ converges in  $C^{2}_{loc}(\R^2_+)\cap C^{1}_{loc}(\bar{\R}^2_+)\times C^{2}_{loc}(\Gamma(\Sigma {\R}^2_+))\cap C^{1}_{loc}(\Gamma(\Sigma {\bar{\R}}^2_+))$ to some $(\widetilde
u,\widetilde \Psi)$ satisfying
\begin{equation*}
\left\{
\begin{array}{rcll}
-\Delta \widetilde{u} &=& 2V^2(0)|x|^{2\alpha}e^{2\widetilde{u}}-V(0)|x|^{\alpha}e^{\widetilde u}|
\widetilde \Psi|^2, &\quad \text { in } \R^2_+,\\
\slashiii{D}\widetilde{\Psi} &=&
-V(0)|x|^{\alpha}e^{\widetilde{u}}\widetilde{\Psi},  &\quad \text { in } \R^2_+,\\
\frac{\partial \widetilde{u}}{\partial
n}&=&cV(0)|x|^\alpha e^{\widetilde{u}}, & \quad \text{ on }  \partial \R^2_+,\\
B\widetilde{\Psi}&=& 0, & \quad \text{ on } \partial \R^2_+
\end{array}
\right.
\end{equation*}
and
\[\int_{\R^2_+}(2V^2(0)|x|^{2\alpha}e^{2\widetilde{u}}-V(0)|x|^{\alpha}e^{\widetilde {u}}|\widetilde{\Psi}|^2)dx+\int_{\partial \R^2_+}
cV(0)|x|^\alpha e^{\widetilde{u}}d\sigma=2\pi(1+\alpha).\]
Then for $\delta\in(0, \delta_1)$ small enough, $R>0$ large enough  and $n$ large enough, we have
\begin{eqnarray}
& & \int_{D^+_{\delta}}(2V^2(x)|x|^{2\alpha}e^{2u_n}-V(x)|x|^{\alpha}e^{u_n}|\Psi_n|^2)dx+\int_{L_{\delta}}cV(x)|x|^{\alpha}e^{u_n}ds\nonumber\\
&=&\int_{D_{\lambda_n R}}(2V^2(x)|x|^{2\alpha}e^{2u_n}-V(x)|x|^{\alpha}e^{u_n}|\Psi_n|^2)dx+\int_{L_{\lambda_n R}}cV(x)|x|^{\alpha}e^{u_n}ds\nonumber\\
& & +\int_{D^+_\delta\backslash D^+_{\lambda_n R}}(2V^2(x)|x|^{2\alpha}e^{2u_n}-V(x)|x|^{\alpha}e^{u_n}|\Psi_n|^2)dx+\int_{L_{\delta}\setminus L_{\lambda_n R}}cV(x)|x|^{\alpha}e^{u_n}ds\nonumber\\
&\geq & \int_{D^+_{R}}(2V^2(\lambda_nx)|x|^{2\alpha}e^{2\widetilde{u}_n}-V(\lambda_nx)|x|^{\alpha}e^{\widetilde{u}_n}|\widetilde{\Psi}_n|^2)+\int_{L_{R}}cV(\lambda_n x)|x|^{\alpha}e^{\widetilde{u}_n}ds\nonumber\\
& & -\int_{D^+_\delta\backslash D^+_{t_nR}}V(x)|x|^{\alpha}e^{u_n}|\Psi_n|^2 \nonumber\\
& \geq &  2\pi(1+\alpha) - \frac{1+\a}{10}.  \label{almost2pi}
\end{eqnarray}
Here in the last step, the fact from Theorem \ref{engy-indt} that the neck energy of the spinor field $\Psi_n$ is converging to zero. We remark that in the above estimate, if there are multiple bubbles then we need to decompose
$D^+_\delta\backslash D^+_{\lambda_nR}$ further into bubble domains and neck domains and then apply the no neck energy result in Theorem \ref{engy-indt}
to each of these neck domains.

On the other hand, we fix some $\delta\in(0, \delta_1)$ small such that \eqref{almost2pi} holds and then let $n\rightarrow \infty$
to conclude that
\begin{eqnarray*}
2\pi(1+\alpha) - \frac{1+\a}{10} &\leq & \int_{D^+_{\delta}}(2V^2(x)|x|^{2\alpha}e^{2u_n}-V(x)|x|^{\alpha}e^{u_n}|\Psi_n|^2)dx+\int_{L_{\delta}}cV(x)|x|^{\alpha}e^{u_n}ds\\
&  =& -\int_{D^+_{\delta}}\Delta u_n= -\int_{\partial B_{\delta}}\frac{\partial u_n}{\partial n}\\
& \rightarrow &   -\int_{\partial D^+_{\delta}}\frac{\partial u}{\partial n} =  -\int_{D^+_{\delta}}\Delta u\\
& =& \int_{D^+_{\delta}}(2V^2(x)|x|^{2\alpha}e^{2u}-V(x)|x|^{\alpha}e^{u}|\Psi|^2)dx +\int_{L_{\delta}}cV(x)|x|^{\alpha}e^{u}ds < \frac{1+\a}{10}
\end{eqnarray*}
Here in the last step, we have used \eqref{less1/10}. Thus we get a contradiction and finish the proof of the Theorem in this case.

\
\

\noindent{\bf Case II.} $\frac{|x_n|}{\lambda_n}\rightarrow +\infty$ as $n\rightarrow +\infty$.

\
\

In this case, as in the arguments in Theorem \ref{engy-indt},  we can  rescale twice to get the bubble. First, we define the rescaling functions
\begin{equation*}
\left\{
\begin{array}{rcl}
\bar{u}_n(x)&=&u_n(|x_n|x)+(\alpha+1)\ln {|x_n|}\\
\bar{\Psi}_n(x)&=&|x_n|^{\frac
12}\Psi_n(|x_n|x)
\end{array}
\right.
\end{equation*}
for any $x\in \overline{D^+}_{\frac{\delta_1}{2|x_n|}}$. Set  $y_n:=\frac{x_n}{|x_n|}$. Due to $\bar{u}_n(y_n)\rightarrow +\infty$, we set that $\delta_n=e^{-\bar{u}_n(y_n)}$ and define the rescaling function
\begin{equation*}
\left\{
\begin{array}{rcl}
\widetilde{u}_n(x)&=&\bar{u}_n(\delta_nx+y_n)+\ln {\delta_n}\\
\widetilde{\Psi}_n(x)&=&\delta_n^{\frac 12}\bar{\Psi}_n(\delta_nx+y_n)
\end{array}
\right.
\end{equation*}
for any $\delta_nx+y_n\in \overline{D^+}_{\frac{\delta_1}{2|x_n|}}$.
Denote that $\rho_n=\frac{e^{-u_n(x_n)}}{|x_n|^\alpha}=\lambda_n(\frac{\lambda_n}{|x_n|})^\alpha $ and $x_n=(s_n, t_n)$.

\

\noindent{\bf Case II.1}   $\frac {t_n}{\rho_n}\rightarrow +\infty$ as $n\rightarrow \infty$.

\

Then, by passing to a subsequence,  $(\widetilde{u}_n,\widetilde{\Psi}_n)$ converges in  $C^{2}_{loc}(\R^2)\times C^{2}_{loc}(\Gamma(\Sigma {\R}^2))$ to some $(\widetilde
u,\widetilde \Psi)$ satisfying
\begin{equation*}
\left\{
\begin{array}{rcll}
-\Delta \widetilde{u} &=& 2V^2(0)e^{2\widetilde{u}}-V(0)e^{\widetilde u}|
\widetilde \Psi|^2, &\quad \text { in } \R^2,\\
\slashiii{D}\widetilde{\Psi} &=&
-V(0)e^{\widetilde{u}}\widetilde{\Psi},  &\quad \text { in } \R^2,\\
\end{array}
\right.
\end{equation*}
with the bubble energy
\[\int_{\R^2}(2V^2(0)e^{2\widetilde{u}}-V(0)e^{\widetilde {u}}|\widetilde{\Psi}|^2)dx=4\pi.\]

Therefore, for  $\delta\in(0, \delta_1)$ small enough, $S, R>0$ large enough  and $n$ large enough,
the fact that the neck energy of the spinor field $\Psi_n$
is converging to zero, we have
\begin{eqnarray*}
 &&  \int_{D^+_{\delta}}(2V^2(x)|x|^{2\alpha}e^{2u_n}-V(x)|x|^{\alpha}e^{u_n}|\Psi_n|^2)dx+\int_{L_{\delta}}cV(x)|x|^{\alpha}e^{u_n}ds\nonumber\\
&=&\int_{D^+_{\frac{\delta}{|x_n|}}}(2V^2(|x_n|x)|x|^{2\alpha}e^{2\bar{u}_n}-V(|x_n|x)|x|^{\alpha}e^{\bar{u}_n}|\bar{\Psi}_n|^2)dx+\int_{L_{\frac{\delta}{|x_n|}}}cV(|x_n|x)|x|^{\alpha}e^{\bar{u}_n}ds \nonumber\\
&=&\int_{D^+_{\frac{\delta}{|x_n|}}\setminus D^+_S(y_n)}(2V^2(|x_n|x)|x|^{2\alpha}e^{2\bar{u}_n}-V(|x_n|x)|x|^{\alpha}e^{\bar{u}_n}|\bar{\Psi}_n|^2)dx \nonumber\\
&&+\int_{ D^+_S(y_n)\setminus D^+_{\frac{\rho_n}{|x_n|}R}(y_n)}(2V^2(|x_n|x)|x|^{2\alpha}e^{2\bar{u}_n}-V(|x_n|x)|x|^{\alpha}e^{\bar{u}_n}|\bar{\Psi}_n|^2)dx \nonumber\\
&&+\int_{D^+_{\frac{\rho_n}{|x_n|}R}(y_n)}(2V^2(|x_n|x)|x|^{2\alpha}e^{2\bar{u}_n}-V(|x_n|x)|x|^{\alpha}e^{\bar{u}_n}|\bar{\Psi}_n|^2)dx +\int_{L_{\frac{\rho_n R}{|x_n|}}(y_n)}cV(|x_n|x)|x|^{\alpha}e^{\bar{u}_n}ds\nonumber\\
&& +\int_{L_S(y_n)\setminus L_{\frac{\rho_n R}{|x_n|}}(y_n)}cV(|x_n|x)|x|^{\alpha}e^{\bar{u}_n}ds+\int_{L_{\frac{\delta}{|x_n|}}\setminus L_{S}(y_n)}cV(|x_n|x)|x|^{\alpha}e^{\bar{u}_n}ds\nonumber\\
&\geq & \int_{D_R \cap \{t>-\frac{t_n}{\rho_n}\}}(2V^2(x_n+\rho_nx)|\frac{x_n}{|x_n|}+\frac{\rho_n}{|x_n|}x|^{2\alpha}e^{2\widetilde{u}_n(x)}
-V(x_n+\rho_nx)|\frac{x_n}{|x_n|}+\frac{\rho_n}{|x_n|}x|^{\alpha}
e^{\widetilde{u}_n(x)}|\widetilde{\Psi}_n|^2)dx \nonumber\\
& & + \int_{D_R \cap \{t=-\frac{t_n}{\rho_n}\}}(cV(x_n+\rho_nx)|\frac{x_n}{|x_n|}+\frac{\rho_n}{|x_n|}x|^{\alpha}e^{\widetilde{u}_n(x)}\nonumber\\
& & -\int_{D^+_{|x_n|S}(x_n)\backslash D^+_{\rho_n R}(x_n)}V(x)|x|^{\alpha}e^{u_n}|\Psi_n|^2 - \int_{D^+_{\frac{\delta}{|x_n|}}\setminus D^+_S(y_n)}V(t_nx)|x|^{\alpha}e^{\bar{u}_n}|\bar{\Psi}_n|^2 \nonumber\\
& \geq &  4\pi - \frac{1}{10}.
\end{eqnarray*}
Then, applying similar arguments as in {\bf Case I}, we get a contradiction, and finish the proof of the Theorem in this case.


\

\noindent{\bf Case II.2}   $\frac {t_n}{\rho_n}\rightarrow \Lambda$ as $n\rightarrow \infty$.

 \

Then, by passing to a subsequence,  $(\widetilde{u}_n,\widetilde{\Psi}_n)$ converges in  $C^{2}_{loc}(\R^2_{-\Lambda})\cap C^{1}_{loc}(\bar{\R}^2_{-\Lambda})\times C^{2}_{loc}(\Gamma(\Sigma {\R}^2_{-\Lambda}))\cap C^{1}_{loc}(\Sigma\bar{\R}^2_{-\Lambda})$ to some $(\widetilde
u,\widetilde \Psi)$ satisfying
\begin{equation*}
\left\{
\begin{array}{rcll}
-\Delta \widetilde{u} &=& 2V^2(0)e^{2\widetilde{u}}-V(0)e^{\widetilde u}|
\widetilde \Psi|^2, &\quad \text { in } \R^2_{-\Lambda},\\
\slashiii{D}\widetilde{\Psi} &=&
-V(0)e^{\widetilde{u}}\widetilde{\Psi},  &\quad \text { in } \R^2_{-\Lambda},\\
\frac{\partial \widetilde{u}}{\partial
n}&=&cV(0) e^{\widetilde{u}}, &\quad  \text{ on } \partial \R^2_{-\Lambda}
,\\
B\widetilde{\Psi} &=& 0, & \quad \text{ on } \partial \R^2_{-\Lambda},
\end{array}
\right.
\end{equation*}
with the bubble energy
\[\int_{\R^2_{-\Lambda}}(2V^2(0)e^{2\widetilde{u}}-V(0)e^{\widetilde {u}}|\widetilde{\Psi}|^2)dx+\int_{\partial \R^2_{-\Lambda}}cV(0)e^{\widetilde{u}}d\sigma=2\pi.\]
Then, applying similar arguments as in {\bf Case II.1},  we can get a contradiction, and finish the proof of the Theorem.
\hfill{$\square$}

\
\

\section{Blow-up Value}\label{value}

By means of Theorem \ref{mainthm-1}, we can further compute the blow-up value at the blow-up point $p$, which is defined as
\[m(p)=\lim_{\delta\rightarrow 0}\lim_{n\rightarrow \infty}\{\int_{D^+_\delta(p)}(2V^2(x)|x|^{2\alpha}e^{2u_n}-V(x)|x|^{\alpha}e^{u_n}|\Psi_n|^2)dx+\int_{L_{\delta}(p)}cV(x)|x|^{\alpha}e^{u_n}ds\}.\]
We know from Theorem \ref{mainthm-1} that $m(p)> 0$. Now we shall determine the precise value of $m(p)$ under a boundary condition.

\

\noindent{\bf Proof of Theorem \ref{BV}:}
When $p\notin L_{\delta_0}(p)$, It is clear that we can choose $\delta_0$ sufficiently small such that $
\overline{D_{\delta_0}^+(p)}=\overline{D_{\delta_0}(p)}$. Then we have $m(p)=4\pi$ according to the arguments in \cite{JZZ3}. Next we assume that  $p\in L_{\delta_0}(p)$. Without loss of generality, we assume $p=0$. The case of $p\neq 0$ can be handled analogously.

By using the boundary condition, it follows that
$$0\leq u_n-\min_{ S^+_{\delta_0}}u_n\leq C$$
on $S^+_{\delta_0}$. Let $w_n$ be the unique solution of the following problem
\begin{equation*}
\left\{
\begin{array}{rlll}
-\Delta w_n &=& 0, &\text{in } D^+_{\delta_0},\\
\frac{\partial w_n}{\partial n} &=& 0, &\text{on }L_{\delta_0},\\
w_n &=& u_n-\min_{S^+_{\delta_0}}u_n, &\text{on }S^+_{\delta_0}.
\end{array}
\right.
\end{equation*}
It follows from the maximum principle and the Hopf Lemma that  $w_n$ is uniformly bounded
in $\overline {D^+_{\delta_0}}$,  and consequently $w_n$ is $C^2(D^+_{\delta_0})\cap C^1(D^+_{\delta_0}\cup L_{\delta_0})$. Now we set that $v_n=u_n-\min_{S^+_{\delta_0}}u_n-w_n$. Then $v_n$ satisfies that
\begin{equation*}
\left\{
\begin{array}{lcl}
-\Delta v_n=2V^2(x)|x|^{2\alpha}e^{2u_n}-V(x)|x|^{\alpha}e^{u_n}|\Psi_n|^2, &\text{in } D^+_{\delta_0},\\
\frac{\partial v_n}{\partial n}=cV(x)|x|^{\alpha}e^{u_n}, &\text{on }L_{\delta_0},\\
v_n=0, &\text{on }S^+_{\delta_0},
\end{array}
\right.
\end{equation*}
with the energy condition
\begin{equation}\label{8.3}
\int_{D^+_{\delta_0}}(2V^2(x)|x|^{2\alpha}e^{2u_n}-V(x)|x|^{\alpha}e^{u_n}|\Psi_n|^2)dx+\int_{L_{\delta_0}}cV(x)|x|^{\alpha}e^{u_n}ds\leq C.
\end{equation}
By  Green's representation formula, we have
\bee
v_n(x) &=& \frac 1{\pi}\int_{D^+_{\delta_0}}\log\frac
1{|x-y|}(2V^2(y)|y|^{2\alpha}e^{2u_n}-V(y)|y|^{\alpha}e^{u_n}|\Psi_n|^2)dy  \nn \\
&& +\frac 1{\pi}\int_{L_{\delta_0}}\log\frac
1{|x-y|}cV(y)|y|^{\alpha}e^{u_n}dy+R_n(x)
\eee
where $R_n(x) \in C^1 (D^+_{\delta_0}\cup L_{\delta_0})$ is a regular term. By using Theorem \ref{mainthm-1}, we know
\begin{equation}\label{8.2}
v_n(x)\rightarrow \frac{m(p)}{\pi}\ln\frac{1}{|x|}+R(x), \text{ in } C^{1}_{loc}((D^+_{\delta_0}\cup L_{\delta_0})\setminus\{0\})
\end{equation}
for $R(x)\in C^1 (D^+_{\delta_0}\cup L_{\delta_0})$.
On the other hand, we observe that $(v_n,\Psi_n)$ satisfies
\begin{equation*}
\left\{
\begin{array}{lcll}
-\Delta v_n &=& 2K_n^2(x)|x|^{2\alpha}e^{2v_n}-K_n(x)|x|^{\alpha}e^{v_n}|\Psi_n|^2, &\text{in } D^+_{\delta_0},\\
\slashiii{D}{\Psi_n} &=&-K_n(x)e^{v_n}{\Psi_n}, &\text{in } D^+_{\delta_0},\\
\frac{\partial v_n}{\partial n}&=&cK_n(x)|x|^{\alpha}e^{u_n}, &\text{on }L_{\delta_0},\\
B(\Psi_n)&=&0, &\text{on }L_{\delta_0},
\end{array}
\right.
\end{equation*}
where $K_n=V(x)e^{\min_{S^+_{\delta_0}}u_n+w_n}$. Noticing  the Pohozaev identity of $(v_n,\Psi_n)$ in $D^+_{\delta_0}$ for $0<\delta<\delta_0$ is
\begin{eqnarray}\label{poho-2}
&&\delta \int_{S^+_\delta} |\frac {\partial v_n}{\partial
\nu}|^2-\frac
12|\nabla v_n|^2d\sigma \nonumber\\
&=& (1+\alpha)\{\int_{D^+_\delta}(2K_n^2(x)|x|^{2\alpha}e^{2v_n}-K_n(x)|x|^{\alpha}e^{v_n}|\Psi_n|^2)dv+
\int_{L_{\delta}}cK_n(x)|x|^{\alpha}e^{v_n}ds \} \nonumber\\
& & -\delta\int_{S^+_{\delta}}K_n^2(x)|x|^{2\alpha}e^{2v_n}d\sigma +\int_{L_{\delta}}c\frac {\partial K_n(s,0)}{\partial s}|s|^{\alpha}se^{v_n(s,0)}ds-cK_n(s,0)|s|^{\alpha}se^{v_n(s,0)}|^{s=\delta}_{s=-\delta} \nonumber\\
& &+\int_{D^+_\delta}x\cdot \nabla (K_n^2(x))|x|^{2\alpha}e^{2v_n}dv -\int_{D^+_\delta}x\cdot \nabla
K_n(x)|x|^{\alpha}e^{v_n}|\Psi_n|^2dv \nonumber\\
& & + \frac 14\int_{S^+_{\delta}}\la\frac {\partial
\Psi_n}{\partial \nu},(x+\bar{x})\cdot\Psi_n\ra d\sigma+\frac
14\int_{S^+_\delta}\la (x+\bar{x})\cdot\Psi_n, \frac
{\partial \Psi_n}{\partial \nu}\ra d\sigma.
\end{eqnarray}
We will take $n\rightarrow \infty $ first and then $\delta\rightarrow 0$ in (\ref{poho-2}). By using (\ref{8.2}) we get
$$
\lim_{\delta\rightarrow 0}\lim_{n\rightarrow \infty}\delta\int_{S^+_\delta}
|\frac {\partial v_n}{\partial \nu}|^2-\frac 12|\nabla
v_n|^2d\sigma=\lim_{\delta\rightarrow 0}\delta\int_{S^+_\delta}\frac 12 |\frac
{\partial (\frac{m(p)}{\pi}\ln\frac{1}{|x|})}{\partial \nu}|^2d\sigma=\frac
1{2\pi}m^2(p).
$$
By using $u_n\rightarrow -\infty$ uniformly on $S^+_{\delta}$, we also have
$$
\lim_{\delta\rightarrow 0}\lim_{n\rightarrow \infty} \delta\int_{S^+_\delta}K_n^2(x)|x|^{2\alpha}e^{2v_n}d\sigma =\lim_{\delta\rightarrow 0}\lim_{n\rightarrow \infty}
\delta\int_{S^+_\delta}V^2(x)|x|^{2\alpha}e^{2u_n}d\sigma=0,
$$
and
$$\lim_{\delta\rightarrow 0}\lim_{n\rightarrow \infty} cK_n(s,0)|s|^{\alpha}se^{v_n(s,0)}|^{s=\delta}_{s=-\delta} =0. $$
By using the energy condition (\ref{8.3}),  we have
$$
\lim_{\delta\rightarrow 0}\lim_{n\rightarrow \infty}\int_{D^+_\delta}(|x|^{2\alpha}e^{2v_n}x\cdot\nabla (K_n^2(x))-|x|^{\alpha}e^{v_n}|\Psi_n|^2x\cdot \nabla
K_n(x))dx=0,
$$
and
$$
\lim_{\delta\rightarrow 0}\lim_{n\rightarrow \infty}\int_{L_{\delta}}c\frac {\partial K_n(s,0)}{\partial s}|s|^{\alpha}se^{v_n(s,0)}ds=0.
$$
Since $u_n\rightarrow -\infty$ uniformly in any compact subset of $(D^+_{\delta_0}\cup L_{\delta_0} )\backslash \{0\}$, and  $|\Psi_n|$ is uniformly bounded in any compact subset of $(D^+_{\delta_0}\cup L_{\delta_0}) \backslash \{0\}$, we know
\begin{equation*}
\left\{
\begin{array}{rcll}
\slashiii{D}\Psi &=& 0, & \text { in } D^+_{\delta_0},\\
B \Psi &=& 0, &\text { on } L_{\delta_0}\setminus\{0\}.
\end{array}
\right.
\end{equation*}
 We extend $\Psi$ a harmonic
 spinor $\overline{\Psi}$ on $D_{\delta_0}\setminus\{0\}$ with bounded
 energy, i.e., $||\overline{\Psi}||_{L^4(D_{\delta_0})} < \infty$. Since the local singularity of a harmonic spinor with finite energy is removable, we have $\overline{\Psi}$ is smooth in $D_{\delta_0}$.
It follows that $\Psi$ is smooth in $D^+_{\delta_0}\cup L_{\delta_0}$.
Therefore we obtain that
$$
\lim_{\delta\rightarrow 0}\lim_{n\rightarrow \infty} (\frac 14\int_{S^+_{\delta}}\la\frac {\partial
\Psi_n}{\partial \nu},(x+\bar{x})\cdot\Psi_n\ra d\sigma+\frac
14\int_{S^+_\delta}\la (x+\bar{x})\cdot\Psi_n, \frac
{\partial \Psi_n}{\partial \nu}\ra d\sigma)=0.
$$
Putting all together, we obtain that
$$ \frac{1}{2\pi}m^2(p)=(1+\alpha)m(p).$$ It follows that
$m(p)=2\pi(1+\alpha)$. Thus we finish the proof of Theorem \ref{BV}.
\hfill{$\square$}

\

\section{Energy quantization for the global super-Liouville boundary problem}\label{global}

In this section, we will show the quantization of energy for a sequence of blowing-up solutions to the global super-Liouville boundary problem on a singular Riemann surface.
Let $(M, \A, g)$ be a compact Riemann surface  with conical singularities represented by the divisor $\A=\Sigma_{j=1}^{m}\alpha_j q_j$, $\alpha_j>0$ and with a spin structure. We assume that  $\partial M $ is not empty and $(M, g)$ has conical
singular points $q_1, q_2, \cdots, q_m$ such that $q_1, q_2, \cdots, q_l$ are in $M^o$ for $1\leq l<m$ and $q_{l+1}, q_{l+2}, \cdots, q_m$ are on $\partial M$.   Writing $g=e^{2\phi}g_0$, where $g_0$ is a smooth metric on $M$, we can deduce from the results for the local super-Liouville equations:

\
\

\noindent{\bf Proof of Theorem \ref{thmsin}}:
Since $g=e^{2\phi}g_0$ with $g_0$ being smooth, then by the well known properties of $\phi$ (see e.g. \cite{T1} or \cite{BDM}, p. 5639), we know that $(u_n,\psi_n)$ satisfies
\begin{equation*}
\left\{
\begin{array}{rclr}
-\Delta_{g_0} (u_n+\phi) &=&  2e^{2(u_n+\phi)}-e^{u_n+\phi}\left\langle e^{\frac {\phi} 2}\psi_n ,e^{\frac {\phi} 2}\psi_n\right\rangle_{g_0} -K_{g_0}-\sum_{j=1}^l  2\pi\alpha_j \delta_{q_j}   &\qquad  \text {in } M^o,\\
\slashiii{D}_{g_0}(e^{\frac {\phi} 2}\psi_n) &=& -e^{u_n+\phi}(e^{\frac {\phi} 2}\psi_n) &\qquad \text {in } M^o,\\
\frac {\partial (u_n+\phi)}{\partial n}& = & ce^{u_n+\phi}-h_{g_0}+\sum_{j=l+1}^m  \pi\alpha_j \delta_{q_j}  ,\qquad & \text { on } \partial M,\\
B^{\pm}(e^{\frac {\phi} 2}\psi_n)&=& 0,\qquad & \text { on } \partial M,
\end{array}
\right.
\end{equation*}
with the energy conditions:
$$
\int_{M}e^{2(u_n+\phi)}dg_0+|e^{\frac {\phi} {2} }\psi_n|_{g_0}^4dv_{g_0}+\int_{\partial M}e^{u_n+\phi}d\sigma_{g_0}<C.
$$
If we define the blow-up set of $u_n + \phi$ as
$$
\Sigma '_1=\left\{ x\in M,\text{ there is a sequence
}y_n\rightarrow x\text{ such that }(u_n+\phi)(y_n)\rightarrow +\infty
\right\}, $$
then by Remark \ref{rem3.4} and Remark 3.3 in \cite{JZZ3}, we have  $\Sigma_1=\Sigma'_1$. By the blow-up results of the local sytem, it follows that one of the following alternatives holds:
\begin{enumerate}
\item[i)]  $u_n$ is bounded in $L^{\infty}(M).$

\item[ii)]  $u_n$ $\rightarrow -\infty $ uniformly on $M$.

\item[iii)]  $\Sigma _1$ is finite, nonempty and
\begin{equation*}
u_n\rightarrow -\infty \text{ uniformly on compact subsets of
}M\backslash \Sigma _1.
\end{equation*}
Furthermore,
\begin{equation*}
\int_M(2e^{2(u_n+\phi)}-e^{u_n+\phi}|e^{\frac{\phi}{2}}\psi_n|_{g_0}^2)\varphi dv_{g_0}+\int_{\partial M}ce^{u_n+\phi}\varphi d\sigma_{g_0}\rightarrow \sum_{p_i\in \Sigma_1}m(p_i)\varphi(p_i)
\end{equation*} for any smooth function $\varphi$ on $M$.
\end{enumerate}
Next let $p=\frac q{q-1}>2$. Notice that
\begin{eqnarray*}
& & ||\nabla (u_n+\phi)||_{L^q(M,g_0)}\\
& \leq & \sup\{|\int_M\nabla (u_n+\phi)\nabla \varphi
dv_{g_0}||\varphi\in W^{1,p}(M,g_0),\int_{M}\varphi
dv_{g_0}=0,||\varphi||_{W^{1,p}(M,g_0)}=1\}.
\end{eqnarray*}
Due to $||\varphi ||_{L^{\infty}(M,g_0)\leq C}$ for any $\varphi\in W^{1,p}(M,g_0)$ with $\int_{M}\varphi
dv_{g_0}=0$ and $||\varphi||_{W^{1,p}(M,g_0)}=1$
by the Sobolev embedding theorem, we get that
\bee
& & |\int_{M}\nabla (u_n+\phi)\nabla\varphi dv_{g_0}| \nn\\
&= &  |-\int_M\Delta_{g_0} (u_n+\phi)\varphi dv_{g_0}+\int_{\partial M}\frac{\partial (u_n+\phi)}{\partial n}\varphi d\sigma_{g_0} |  \nn \\
&\leq& \int_{M}(2e^{2(u_n+\phi)}+e^{u_n+\phi}|e^{\frac {\phi} {2}} \psi_n|_{g_0}^2 + |K_{g_0}|)|\varphi|dg_0 +\int_{\partial M}( ce^{u_n+\phi}+|h_{g_0}|)|\varphi|d\sigma_{g_0} \nn\\
& &+ \sum_{j=1}^l  |\int_M 2\pi\alpha_j \delta_{q_j}\varphi dv_{g_0}|+ \sum_{j=1+1}^m  |\int_{\partial M} \pi\alpha_j \delta_{q_j} \varphi d\sigma_{g_0}|\nn \\
& \leq &  C.  \nn
\eee
This means that $u_n+\phi-\frac{1}{|M|}\int_{M}(u_n+\phi)dv_{g_0}$ is uniformly bounded in $W^{1,q}(M,g_0)$.

\
\

We define the Green function $G$ by
\begin{equation*}
\left \{\begin{array}{l} - \Delta_{g_0}
G=\sum_{p\in M^o\cap \Sigma_1}m(p)\delta_p-K_{g_0}-\sum_{j=1}^l 2\pi\alpha_j \delta_{q_j},\\
\frac{\partial G}{\partial n}=\sum_{p\in \partial M\cap \Sigma_1}m(p)\delta_p-h_{g_0}+\sum_{j=l+1}^m \pi\alpha_j \delta_{q_j},\\
\int_{M}Gdv_{g_0}=0.
\end{array}\right.
\end{equation*}
It is clear that  $G\in W^{1,q}(M,g_0)\cap C_{loc}^2(M\backslash \Sigma_1)$
with $\int_{M}Gdg_0=0$ for $1<q<2$.

Now we take  $R>0$ small such that, at each blow-up point $p\in \Sigma_1$, the geodesic ball of $M$, $B^{M}_R(p)$, satisfies $B^{M}_R(p)\cap (\Sigma_1 \cup \{q_1,q_2, \cdots, q_m\})=\{p\}$. Then we also have
\begin{equation*}
G(x)=\left\{
\begin{array}{lcl}
- \frac 1{2\pi}m(p)\log d(x,p)+g(x),  \quad &{\text if}&  p\in  M^o\cap (\Sigma_1 \setminus \{q_1,q_2, \cdots, q_m\}),  \\
- (\frac 1{2\pi}m(p)- \alpha_j )\log d(x,p)+g(x),  \quad    &{\text if}&  p=q_j\in  M^o\cap \Sigma_1 \cap \{q_1,q_2, \cdots, q_l\}, \\
- (\frac 1{\pi}m(p))\log d(x,p)+g(x),  \quad    &{\text if}&  p\in  \partial M\cap (\Sigma_1 \setminus \{q_{l+1},q_{l+2}, \cdots, q_m\}), \\
- (\frac 1{\pi}m(p)+a_j)\log d(x,p)+g(x),  \quad    &{\text if}&  p=q_j\in  \partial M\cap \Sigma_1 \cap  \{q_{l+1},q_{l+2}, \cdots, q_m\},
\end{array}
\right.
\end{equation*}
for $x\in B^M_R(p)\backslash \{p\}$ with $g\in C^2(B^M_R(p))$, where $d(x,p)$ denotes the Riemannian distance between $x$ and $p$ with respect to $g_0$ and
$$m(p)=\lim_{R\rightarrow 0}\lim_{n\rightarrow \infty}\{\int_{B^M_R(p)}(2e^{2(u_n+\phi)}-e^{u_n+\phi}|e^{\frac{\phi}{2}}\psi_n|_{g_0}^2-K_{g_0})dv_{g_0}+\int_{\partial M\cap B^M_R(p)}(ce^{u_n+\varphi}-h_{g_0})d\sigma_{g_0}\}.
$$

On the other  hand, since for any $\varphi\in C^{\infty}(M)$
\begin{eqnarray*}
& & \int_{M}\nabla (u_n+\phi-G)\nabla
\varphi dv_{g_0} \\
&=&  -\int_{M}\Delta_{g_0}(u_n+\phi-G)\varphi dv_{g_0}+\int_{\partial M}\frac{\partial (u_n+\phi-G)}{\partial n}\varphi d\sigma_{g_0} \\
&=& \int_{M}( 2e^{2(u_n+\phi)}-e^{u_n+\phi}|e^{\frac {\phi} 2}\psi_n|^2_{g_0}-\sum_{p\in
M^o\cap \Sigma_1}m(p)\delta_p)\varphi dv_{g_0}+\int_{\partial M}(ce^{u_n+\phi}-\sum_{p\in \partial M\cap
\Sigma_1}m(p)\delta_p)\varphi d\sigma_{g_0}\\
& \rightarrow & 0,\quad \text { as }
n\rightarrow \infty,
\end{eqnarray*}
by using the fact that $u_n+\phi-\frac{1}{|M|}\int_{M}(u_n+\phi)dg_0$ is uniformly bounded in $W^{1,q}(M,g_0)$, we get
$$
u_n+\phi-\frac 1{|M|}\int_{M}(u_n+\phi)dg_0\rightarrow G
$$
stronly in $C_{loc}^2(M\backslash \Sigma_1)$ and weakly in $W^{1,q}(M,g_0)$. Consequently we have
\begin{equation*}
 \max_{M^o\cap \partial B^M_{R}(p)}u_n-\min_{M^o\cap \partial B^M_{R}(p)}u_n\leq C.\\
\end{equation*}
Therefore we get the blow-up value $m(p)=4\pi$ when $p\in  M^o\cap (\Sigma_1 \setminus \{q_1,q_2, \cdots, q_m\})$,  $m(p)=4\pi(1+\alpha_j)$ when $p=q_j\in  M^o\cap \Sigma_1 \cap \{q_1,q_2, \cdots, q_l\}$, $m(p)=2\pi$ when $p\in  \partial M\cap (\Sigma_1 \setminus \{q_{l+1},q_{l+2}, \cdots, q_m\})$, and $m(p)=2\pi(1+\alpha_j)$ when $p=q_j\in  \partial M\cap \Sigma_1 \cap  \{q_{l+1},q_{l+2}, \cdots, q_m\}$.
By using that
$$\int_{M}2e^{2u_n}-e^{u_n}|\psi_n|^2_gdv_g+\int_{\partial M}ce^{u_n}d\sigma_{g}=\int_{M}2e^{2(u_n+\phi)}-e^{u_n+\phi}|e^{\frac{\phi}2}\psi_n|^2_{g_0}dv_{g_0}+\int_{\partial M}ce^{u_n}d\sigma_{g_0},$$
we get the conclusion of the Theorem.
\hfill{$\square$}

\section*{\bf{Acknowledgements}} Chunqin Zhou was supported partially by NSFC of China (No. 11771285). Miaomiao Zhu was supported in part by National Natural Science Foundation of China (No. 11601325).

\end{document}